\documentclass[a4paper]{article}

\usepackage{amsmath,amssymb,color}

\usepackage{graphicx}
\usepackage{subfigure}
\usepackage{subdepth}
\usepackage{lineno,etoolbox,units}

\allowdisplaybreaks
\DeclareGraphicsExtensions{.pdf,.eps}

\newcommand{\bu}{\boldsymbol u}

\newcommand{\bv}{\boldsymbol v}
\newcommand{\bV}{\boldsymbol V}
\newcommand{\bw}{\boldsymbol w}
\newcommand{\bbeta}{\boldsymbol  \eta}
\newcommand{\bzeta}{\boldsymbol \zeta}

\newcommand{\bz}{\boldsymbol z}

\newcommand{\bx}{\boldsymbol x}
\newcommand{\by}{\boldsymbol y}

\newcommand{\be}{\boldsymbol e}

\newcommand{\bvar}{\boldsymbol \varphi}

\newcommand{\bU}{\boldsymbol U}

\newcommand{\bs}{\boldsymbol s}

\newcommand{\bff}{\boldsymbol f}

\newtheorem{Theorem}{Theorem}[section]
\newtheorem{lema}[Theorem]{Lemma}

\newtheorem{remark}[Theorem]{Remark}

\newtheorem{Proof}{{\em Proof:}}

\newenvironment{proof}{\begin{Proof}\rm}{\hfill $\Box$ \end{Proof}}

\newcommand{\linenomathpatch}[1]{%
  \cspreto{#1}{\linenomath}%
  \cspreto{#1*}{\linenomath}%
  \csappto{end#1}{\endlinenomath}%
  \csappto{end#1*}{\endlinenomath}%
}
\linenomathpatch{equation}
\linenomathpatch{gather}
\linenomathpatch{multline}
\linenomathpatch{align}
\linenomathpatch{alignat}
\linenomathpatch{flalign}

\usepackage{hyperref}

\usepackage{macros}

\title{POD-ROMs for incompressible flows including snapshots of the temporal derivative of the
full order solution: Error bounds for the pressure}
\author{Bosco
Garc\'{\i}a-Archilla\thanks{Departamento de Matem\'atica Aplicada
II, Universidad de Sevilla, Sevilla, Spain. Research is supported by
Spanish MCINYU under grants PGC2018-096265-B-I00 and PID2019-104141GB-I00 (bosco@esi.us.es)}
\and
Volker John\thanks{Weierstrass Institute for Applied Analysis and Stochastics,
Leibniz Institute in Forschungsverbund Berlin e. V. (WIAS), Mohrenstr. 39, 10117 Berlin, Germany.
Freie Universit\"at of Berlin,
Department of Mathematics and Computer Science,
Arnimallee 6, 14195 Berlin, Germany.}
\and
Sarah Katz\thanks{Weierstrass Institute for Applied Analysis and Stochastics,
Leibniz Institute in Forschungsverbund Berlin e. V. (WIAS), Mohrenstr. 39, 10117 Berlin, Germany.
Research is supported by the Deutsche Forschungsgemeinschaft
(DFG) within the RTG 2433 \emph{Differential Equation- and Data-driven Models in
Life Sciences and Fluid Dynamics (DAEDALUS).}}
  \and Julia Novo\thanks{Departamento de
Matem\'aticas, Universidad Aut\'onoma de Madrid, Spain. Research is supported
by Spanish MINECO
under grants PID2019-104141GB-I00 and VA169P20  (julia.novo@uam.es)}}
\date{\today}
\date{\today}
\begin{document}

\maketitle

\begin{abstract}
Reduced order methods (ROMs) for the incompressible Navier--Stokes equations,
based on proper orthogonal decomposition (POD), are studied that
include snapshots which approach the temporal derivative of the velocity from a
full order mixed finite element method (FOM).
In addition, the set of snapshots contains the mean velocity of the FOM. Both the FOM and the POD-ROM are
equipped with a grad-div stabilization. A velocity error analysis for this method can be found already in
the literature. The present paper studies two different procedures to compute approximations to the pressure
and proves error bounds for the pressure that are independent of inverse powers of the viscosity.
Numerical studies support the analytic results and compare both methods. 
 \end{abstract}


\section{Introduction}

It is sometimes necessary to compute numerical approximations of solutions of partial differential
equations very fast without requiring the highest accuracy. An example is the solution of optimal control
problems, which needs in each iteration the solution of a time-dependent partial differential
equation with only slightly changed data. A popular approach for performing highly efficient simulations
consists in using so-called reduced order models (ROMs). These models utilize a Galerkin approach and the
corresponding basis functions are extracted from one accurate numerical solution, which is called in this context
the full order model (FOM).

This paper studies incompressible flow problems that are modeled by the
incompressible Navier--Stokes equations
\begin{equation}
\begin{array}{rcll}
\label{NS} \partial_t\bu -\nu \Delta \bu + (\bu\cdot\nabla)\bu + \nabla p &=& \bff &\text{in }\ (0,T]\times\Omega, \\
\nabla \cdot \bu &=&0&\text{in }\ (0,T]\times\Omega,
\end{array}
\end{equation}
in a bounded domain $\Omega \subset {\mathbb R}^d$, $d \in \{2,3\}$.
The boundary of $\Omega$ is assumed to be polyhedral and Lipschitz.
In~\eqref{NS},
$\bu$ denotes the velocity field, $p$ the kinematic pressure, $\nu>0$ the kinematic viscosity coefficient,
 and $\bff$ represents the accelerations due to external body forces acting
on the fluid. The Navier--Stokes equations \eqref{NS} have to be complemented
with an initial condition $\bu(0,\bx)=\bu^0(\bx)$ and with
boundary conditions. For simplicity,
we only consider the case of homogeneous
Dirichlet boundary conditions $\bu = \boldsymbol 0$ on $[0,T]\times \partial \Omega$.

The most popular approach for computing basis functions for ROMs is probably the use of the proper
orthogonal decomposition (POD), which will be also  considered here. The complete method is often called
POD-ROM. In the past few years there has been a tremendous progress in the numerical analysis of 
POD-ROM methods, e.g., see \cite{Rub20,kean_sch,LS20,koc_rubino_et_al,novo_rubino,samu_et_al_pres,JMN22,GNR22,IRI22,wir_NS}.
In \cite{wir_NS}, we introduced a POD-ROM model based on a set of snapshots including
the mean value of the velocity at different time instants together with approximations to the velocity
time derivative at different times. Including grad-div stabilization, both in the snapshots computation
and the POD-ROM method, we were able to prove error bounds for the velocity that are convection-robust, i.e.,
the constants in the error bounds are independent of inverse powers of the viscosity.

There are applications that require the pressure solution for computing quantities of interest, e.g.,
lift or drag coefficients at bodies in the flow field. In other applications, the pressure is even more
important than the velocity, e.g., the pressure gradient is the most important biomarker for detecting
stenoses in blood vessels. Using in the FOM a pair of finite element spaces that satisfies a discrete
inf-sup condition, then the basis functions from the POD are discretely divergence-free. Consequently,
the POD-ROM, as a Galerkin method, does not contain the pressure. This situation was considered in \cite{wir_NS}.

The goal of this paper consists in analyzing two algorithms for computing a reduced order pressure. Both
algorithms use a set of velocity snapshots that are based on approximations of the temporal derivative
of the FOM velocity and on pressure snapshots that are solutions of the FOM. The first algorithm is
the supremizer enrichment algorithm from \cite{novo_rubino}. In the present paper, the analysis of \cite{novo_rubino} is extended
from a ROM that uses standard  FOM velocity snapshots
to a ROM that utilizes the above mentioned approximations of the temporal derivative of the velocity. The
second algorithm is the so-called stabilization-motivated POD-ROM proposed in \cite{John_et_al_vp}. This method
is already analyzed in \cite{samu_et_al_pres} and the present paper aims to improve this analysis in several
aspects. First, we avoid the strong restriction $h\le C \Delta t $, where $\Delta t$ and $h$ are
the time step and spatial mesh diameter, respectively. Second, differently to \cite{samu_et_al_pres},
we use the same procedure to compute a POD-ROM pressure in two and three spatial dimensions. Note that in
\cite{samu_et_al_pres} a so-called truncation of the velocity approximation is applied in the nonlinear
convective term to compute the reduced order pressure. Third, in contrast to \cite{samu_et_al_pres},
where only a first order convergence for the considered norm could be proved, our analysis shows for the same
norm an order that is induced from the pair of mixed finite element spaces that were used for performing the FOM simulation.
And finally, again in contrast to \cite{samu_et_al_pres}, the constant in the derived error bound does not
blow up if the viscosity coefficient tends to zero.

Numerical studies will support the analytic results. These studies will also provide an initial comparison 
of the supremizer enrichment and the stabilization-mo\-ti\-vated pressure ROMs. 

The paper is organized as follows. Section~\ref{sec:PN} introduces some notations, the considered finite
element spaces, recalls some inequalities used in the numerical analysis, and presents the FOM. In
Section~\ref{sec:POD} the error estimates for the POD-ROM velocity from \cite{wir_NS} are recalled and a few
new estimates for velocity terms are proved. The numerical analysis for the pressure ROMs is presented 
in Section~\ref{sec:press_rom_ana} and Section~\ref{sec:numres} contains the numerical studies. 
Finally, a summary of the results and an outlook is given in Section~\ref{sec:summary}.

\section{Finite Element Spaces and the FOM}\label{sec:PN}

Standard symbols will be used for Lebesgue and Sobolev spaces, with the usual convention that
$W^{s,2}(\Omega)= H^s(\Omega)$, $s\ge 1$. The inner product in $L^2(\Omega)^d$, $d\ge 1$, is denoted
by $(\cdot,\cdot)$.
In the sequel we denote
$L_0^2(\Omega)= \{q\in L^2(\Omega)\ \mid \ (q,1)=0\}$.
Let us recall
the Poincar\'e inequality,
\begin{equation}\label{poincare}
\|\bv\|_0\le C_p\|\nabla \bv\|_0\quad  \forall\ \bv\in H^1_0(\Omega)^d,
\end{equation}
and the estimate of the divergence of a velocity field by its gradient,
see  \cite[Remark~3.35]{John},
\begin{equation}\label{diver_vol}
\|\nabla \cdot \bv\|_0\le \|\nabla   \bv \|_0\quad \forall\ \bv\in H^1_0(\Omega)^d.
\end{equation}
The following Sobolev embeddings \cite{Adams} will be used in the analysis: For
$q \in [1, \infty)$, there exists a constant $C=C(\Omega, q)$ such
that
\begin{equation}\label{sob1}
\|v\|_{L^{q'}} \le C \| v\|_{W^{s,q}}, \,\,\quad
\frac{1}{q'}
\ge \frac{1}{q}-\frac{s}{d}>0,\quad q<\infty, \quad v \in
W^{s,q}(\Omega)^{d}.
\end{equation}
Generic constants independent of inverse powers of the viscosity and the mesh width are denoted by
$C$, $C_0$ and similar symbols.

Let us denote by  $\{\mathcal{T}_{h}\}=\{(K_j,\phi_{j}^{h})_{j \in J_{h}}\}$, $h>0$, a family of partitions of
$\overline\Omega$, where $h$ is the maximum diameter of the mesh cells $K_j\in \mathcal{T}_{h}$
and $\phi_j^h$ are the mappings from the reference simplex $K_0$ onto $K_j$.
We assume that the family of partitions is shape-regular and  quasi-uniform.
On $\mathcal T_h$, the following finite element spaces are defined
\begin{eqnarray*}
Y_h^l&=& \left\{v_h\in C^0(\overline\Omega)\ \mid \ {v_h}_{\mid_K}\in {\Bbb P}_l(K),\ \forall\ K\in \mathcal T_h\right\}, \ l\ge 1,\quad {\boldsymbol Y}_h^l=\left(Y_h^l\right)^d,\nonumber\\
{\boldsymbol X}_h^l&=&{\boldsymbol Y}_h^l\cap H_0^1(\Omega)^d, \nonumber\\
Q_h^l&=&Y_h^l\cap L_0^2(\Omega), \nonumber\\
{\boldsymbol V}_h^l&=&{\boldsymbol X}_h^l\cap \left\{ {\boldsymbol v}_{h} \in H_0^1(\Omega)^d \ \mid \
(\nabla\cdot{\boldsymbol v}_{h}, q_{h}) =0  \ \forall\ q_{h} \in Q_{h}^{l-1}
\right\},\quad l\geq 2.%\label{eq:V}
\end{eqnarray*}
Hence, ${\boldsymbol V}_h^l$ is the space of discretely divergence-free functions.

The Lagrangian interpolant in ${\boldsymbol X}_h^l$ is denoted by $I_h(\cdot)$. It is well known, e.g.,
see \cite{Brenner-Scott}, that there exists a constant $C>0$ such that
for every~$K\in {\cal T}_h$
\begin{equation}
\label{interp}
\left\| \bu - I_h(\bu)\right\|_{0,K} + h\left\| \nabla(\bu - I_h(\bu))\right\|_{0,K} +h^2
 \| \Delta( \bu -  I_h(\bu))\|_{0,K}
\le C \|\bu\|_{l+1,K} h^{l+1}.
\end{equation}
The inverse inequality
\begin{equation}
\label{inv} | v_{h} |_{W^{m,p}(K)} \leq c_{\mathrm{inv}}
h_K^{n-m-d\left(\frac{1}{q}-\frac{1}{p}\right)}\quad
| v_{h}|_{W^{n,q}(K)} \quad \forall\ v_{h} \in Y_{h}^{l},
\end{equation}
holds,
with $0\leq n \leq m \leq 1$, $1\leq q \leq p \leq \infty$, and $h_K$
being the diameter of~$K \in \mathcal T_h$, because the family of partitions is quasi-uniform, e.g., see \cite[Theorem~3.2.6]{Cia78}. Define $h=\max_{K\in{\cal T}_h} h_K$.

In this paper, we consider Taylor--Hood pairs of finite element spaces \cite{BF,hood0}, i.e.,
pairs of the form $({\boldsymbol X}_h^l, Q_{h}^{l-1})$, $l \ge 2$. These pairs satisfy
a discrete inf-sup condition, see \cite{Bof94,Bof97}, that is, there is a constant
$\beta_{\rm is}>0$ independent of $h$ such that
\begin{equation}\label{lbbh}
 \inf_{q_{h}\in Q_{h}^{l-1}}\sup_{\bv_{h}\in{\boldsymbol X}_h^l}
\frac{(\nabla \cdot \bv_{h},q_{h})}{\|\nabla\bv_{h}\|_{0}
\|q_{h}\|_{0}} \geq \beta_{\rm{is}}.
\end{equation}
Taylor--Hood pairs, in particular for $l=2$, are probably the most popular pairs of
inf-sup stable finite element spaces.

Let ${\boldsymbol V}$ be the space of functions in $H_0^1(\Omega)^d$ with $\nabla \cdot \bu=0$.
The following modified Stokes projection  $\bs_h^m\ :\ {\boldsymbol V}\rightarrow {\boldsymbol V}_h^l$  was introduced in \cite{NS_grad_div} and  is defined by
\begin{equation}\label{stokespro_mod_def}
(\nabla \bs_h^m,\nabla \bvar_h)=(\nabla \bu,\nabla \bvar_h),\quad \forall\
\bvar_{h} \in {\bV_h^l}.
\end{equation}
This projection satisfies the following error bound, see \cite{NS_grad_div},
\begin{equation}
\|\bu-\bs_h^m\|_0+h\|\bu-\bs_h^m\|_1\le C\|\bu\|_j h^j,\quad
1\le j\le l+1.
\label{stokespro_mod}
\end{equation}

As FOM, we use a Galerkin method with grad-div stabilization. The continuous-in-time method reads as follows:
Find $(\bu_h,p_h)\in {\boldsymbol X}_h^l\times Q_h^{l-1}$ such that
\begin{equation}\label{eq:gal_grad_div}
\begin{array}{rcll}
\left(\partial_t\bu,\bv_h\right)+\nu(\nabla \bu_h,\nabla \bv_h)+b(\bu_h,\bu_h,\bv_h)
\\
-(\nabla \cdot \bv_h,p_h)+
\mu(\nabla \cdot\bu_h,\nabla \cdot \bv_h) & = & ({\boldsymbol f},\bv_h) & \forall\ \bv_h\in {\boldsymbol X}_h^l,\nonumber\\
(\nabla \cdot \bu_h,q_h)&=&0 &\forall\ q_h\in Q_h^{l-1},\nonumber
\end{array}
\end{equation}
where $\mu$ is the positive grad-div stabilization parameter and
\[
b(\bu,\bv,\bw) = ((\bu\cdot\nabla)\bv,\bw) +\frac12 ((\nabla\cdot\bu)\bv,\bw).
\]
The pressure drops out if the problem is considered for functions
from the discretely divergence-free space $\bV_h^l$, since
$\bu_h\in \bV_h^l$ satisfies
\begin{eqnarray}\label{eq:gal_grad_div2}
\left(\partial_t\bu_{h},\bv_h\right)+\nu(\nabla \bu_h,\nabla \bv_h)+b(\bu_h,\bu_h,\bv_h)&&
\nonumber\\
+
\mu(\nabla \cdot\bu_h,\nabla \cdot \bv_h)&=&({\boldsymbol f},\bv_h) \quad \forall\ \bv_h\in {\bV}_h^l.
\end{eqnarray}
For this method the following bound holds, see \cite{NS_grad_div} for the first term and
the explanation in \cite{wir_NS} for the second term,
\begin{equation}\label{eq:cota_grad_div}
\|\bu(t,\cdot)-\bu_h(t,\cdot)\|_0+ h\|\bu(t,\cdot)-\bu_h(t,\cdot)\|_1\le C(\bu,p,l+1) h^{l},\quad t\in (0,T],
\end{equation}
where the constant $C(\bu,p,l+1)$ does not explicitly depend on inverse powers of $\nu$.
To bound the error in the pressure, arguing as in \cite{NS_grad_div},  and using results in~\cite{cor}, one can prove
\begin{equation}\label{eq:cota_pre}
\left(\sum_{j=1}^n\Delta t \left\|  p^j-p_h^j\right\|_{0}^2 \right)^{1/2} \le C_{\rm press}(\bu,p,l+1) h^{l-1/2}.
\end{equation}

\section{Velocity POD-ROMs}\label{sec:POD}

This section describes briefly the velocity POD-ROMs that were investigated in \cite{wir_NS} and summarizes
the error estimates derived in this paper.

\subsection{The General Method}

Let $M\in\mathbb N$ be a positive integer and set $\Delta t=T/M$, i.e., the time instants are given
by $t_j=j\Delta t$, $j=0,\ldots, M$. For simplicity of presentation, let $\bu_h^j=\bu_h(t_j,\cdot)$,
$p_h^j=p_h(t_j,\cdot)$
denote the FOM approximation of the velocity at time instant $t_j$ and $\partial_t \bu_h^j = \partial_t \bu_h(t_j,\cdot)$
the approximation of the temporal derivative (e.g., see \cite[Remark~2.1]{wir_NS} on the computation of time derivatives). %Having computed $(\bu_h^j,p_h^j)$ using some temporal
%discretization, then $\partial_t \bu_h(t_j,\cdot)$ can be obtained by considering the
%continuous-in-time problem \eqref{eq:gal_grad_div}, inserting $(\bu_h^j,p_h^j)$ in all terms but the
%first one on the left-hand side, and solving a system with the finite element mass matrix, see \cite{wir_NS}
%for a detailed description. Considering the fully discrete FOM, with the backward Euler or BDF2 method as
%discretization in time, and transforming all terms with the mass matrix to the left-hand side and the
%other terms to the right-hand side, then one gets a system with the same matrix and same right-hand side.
%Hence, the solution of this system can be represented as a linear combination of FOM velocities at different
%times. In particular, $\partial_t \bu_h^j$ is discretely divergence-free, i.e., $\partial_t \bu_h^j\in\bV_{h,l}$.
The mean value of the velocity snapshots is defined by
$\overline \bu_h=\frac{1}{N}\sum_{j=0}^M \bu_h^j$ with $N=M+1$. Then, the following space, which is
based on data from the FOM simulation, is defined
\begin{equation}\label{eq:span_U}
{\mathcal \bU} =
\mbox{span}\left\{\sqrt{N}\overline \bu_h,\tau\partial_t \bu_h^1,\ldots,\tau\partial_t \bu_h^M\right\}=\mbox{span}\left\{\by_h^1,\by_h^2,\ldots,\by_h^N\right\}.
\end{equation}
The factor $\tau$ in front of the temporal derivatives is a time scale. Its introduction aims
to make the vectors $ \by_h^j$, $j=1,\ldots,N$, dimensionally correct, i.e.,
all vectors possess the physical unit $\unitfrac{m}s$.  The dimension of  $\mathcal \bU$ is denoted by $d_v$.

Notice that by taking derivatives with respect to~time in the second equation in~\eqref{eq:gal_grad_div}
it follows that $\partial_t \bu_h^j$ is discretely divergence-free, i.e., $\partial_t \bu_h^j\in\bV_{h,l}$ for all $j=1,\ldots,M$, so that $\mathcal \bU\subset \bV_{h,l}$.

Let the space $\mathcal \bU$ be equipped with an inner product, denoted by $(\cdot, \cdot)_X$. For the
Navier--Stokes equations, this might be the product from $L^2(\Omega)^d$ or from
$H_0^1(\Omega)^{d\times d}$. Both cases will be considered in the present paper. Then, the
first step of the POD approach consists in defining the correlation matrix
$K^{\mathrm{v}}=((k_{i,j}^{\mathrm{v}}))\in {\mathbb R}^{N\times N}$
with the entries
\[
k_{i,j}^{\mathrm{v}}=\frac{1}{N}\left(\by_h^i,\by_h^j\right)_X, \quad i,j=1,\ldots,N.
\]

Denote by $ \lambda_1\ge  \lambda_2,\ldots\ge \lambda_{d_v}>0$ the positive eigenvalues
of $K^{\mathrm{v}}$ and by
$\bv_1,\ldots,\bv_{d_v}\in {\mathbb R}^{N}$  associated eigenvectors with Euclidean norm $1$.
Then, the  POD basis functions of $\mathcal \bU$, which are orthonormal, are defined by
\begin{equation*}%\label{lachi}
\bvar_k=\frac{1}{\sqrt{N}}\frac{1}{\sqrt{\lambda_k}}\sum_{j=1}^{N} v_k^j \by_h^j,
\end{equation*}
with  $v_k^j$ being the $j$-th component of $\bv_k$.
The following representation of the error was derived in \cite[Proposition~1]{kunisch}
\[
\frac{1}{N}\sum_{j=1}^N\left\|\by_h^j-\sum_{k=1}^r\left(\by_h^j,\bvar_k\right)_X\bvar_k\right\|_{X}^2=\sum_{k=r+1}^{d_v}\lambda_k.
\]
Inserting the functions from  \eqref{eq:span_U} yields
\begin{eqnarray}\label{eq:cota_pod_deriv}
\left\|\overline \bu_h-\sum_{k=1}^r\left(\overline\bu_h,\bvar_k\right)_X\bvar_k\right\|_{X}^2 && \nonumber \\
+\frac{\tau^2}{M+1}\sum_{j=1}^M\left\|\partial_t\bu_{h}^{j}-\sum_{k=1}^r\left(\partial_t\bu_{h}^{j},\bvar_k\right)_X\bvar_k\right\|_{X}^2 & =&
\sum_{k=r+1}^{d_v}\lambda_k.
\end{eqnarray}

%The mass matrix of the POD basis is given by $M^{\mathrm{v}}=((m_{i,j}^{\mathrm{v}}))\in {\mathbb R}^{d_v\times d_v}$ with the entries
%$m_{i,j}^{\mathrm{v}}=( \bvar_j,\bvar_i)_X$.
%If  $X=H_0^1(\Omega)^{d\times d}$, then the
%following inverse inequality was proved in \cite[Lemma~2]{kunisch}:
%\begin{equation}\label{eq:inv_M}
%\|\nabla \bv \|_0\le \sqrt{\|(M^{\mathrm{v}})^{-1}\|_2}\|\bv\|_0\quad \forall\ \bv \in {\mathcal %\bU}.
%\end{equation}
The stiffness matrix of the POD basis is defined by $S^{\mathrm{v}}=((s_{i,j}^{\mathrm{v}}))\in {\mathbb R}^{d_v\times d_v}$, where
$s_{i,j}^{\mathrm{v}}=(\nabla \bvar_j,\nabla \bvar_i)_X$. In the case $X=L^2(\Omega)^d$, the following estimate
was shown in \cite[Lemma~2]{kunisch}:
\begin{equation}\label{eq:inv_S}
\|\nabla \bv \|_0\le \sqrt{\|S^{\mathrm{v}}\|_2}\|\bv\|_0\quad \forall\ \bv \in {\mathcal \bU}.
\end{equation}

We will denote the space spanned by the first $r$ POD basis functions by
\[
{\mathcal \bU}^r= \mbox{span}\left\{\bvar_1,\bvar_2,\ldots,\bvar_r\right\},\quad 1\le r\le d_v,
\]
and by $P_r^{\mathrm{v}}\ : \  {\boldsymbol X}_h^l  \to {\mathcal \bU}^r$,  the $X$-orthogonal projection onto ${\mathcal \bU}^r$.

In  \cite[(3.11)]{wir_NS}, we proved for any Banach space $Y$ defined on $\Omega$ the estimate
\begin{equation}\label{eq:zetast_mean}
\max_{0\le k\le N }\|\bz^k\|_Y^2 \le  {3}\|\overline\bz\|_Y^2+\frac{12 T^2}{M}\sum_{n=1}^M \|\partial_t \bz^n\|_Y^2+\frac{16T}{3}(\Delta t)^2\int_0^T\|\partial_{tt}\bz\|_Y^2\ ds,
\end{equation}
provided that $\bz\in H^2(0,T;Y)$.
Applying \eqref{eq:zetast_mean} to $\bu_h^n-P_r^{\mathrm{v}}\bu_h^n$ yields
\begin{eqnarray}\label{eq:bound_2nd_term_pre}
\max_{0\le n\le M }\|\bu_h^n-P_r^{\mathrm{v}}\bu_h^n\|_Y^2 &\le& 3\|\overline{\bu_h^n-P_r^{\mathrm{v}}\bu_h^n}\|_Y^2 
+\frac{12T^2}{M}\sum_{n=1}^M\|\partial_t\bu_h^n-P_r^{\mathrm{v}}\partial_t\bu_h^n\|_Y^2 \nonumber\\
&&\quad+\frac{16T}{3}(\Delta t)^2 \int_0^T\|\partial_{tt}\bu_{h}-P_r^{\mathrm{v}}\partial_{tt}\bu_h^n\|_Y^2\ ds.
\end{eqnarray}

Let $Y=X$, the space that is connected to the projection $P_r^{\mathrm{v}}$. Using the $X$-stability 
of the projection gives 
\[
 \|\partial_{tt}\bu_{h}-P_r^{\mathrm{v}}\partial_{tt}\bu_h\|_X^2\le 2\|P_r^{\mathrm{v}}\partial_{tt}\bu_h\|_X^2
 +2 \|\partial_{tt}\bu_{h}\|_X^2\le 4\|\partial_{tt}\bu_{h}\|_X^2,
\] 
so that we obtain with \eqref{eq:cota_pod_deriv}
\begin{equation}\label{eq:bound_2nd_term}
\max_{0\le n\le M }\|\bu_h^n-P_r^{\mathrm{v}}\bu_h^n\|_X^2 \le C_{X}^2 = \rho^2\sum_{k={r+1}}^{d_v}\lambda_k
%\nonumber\\
+\frac{64}{3}T(\Delta t)^2 \int_0^T\|\partial_{tt}\bu_{h}\|_X^2\ ds,
\end{equation}
where
\begin{equation}
\label{rho} \rho= \max \left\{\sqrt3,\frac{\sqrt{24} T}\tau\right\}.
\end{equation}
All explicit constants in the bounds \eqref{eq:zetast_mean}--\eqref{rho} will immediately be absorbed in generic 
constants in the following analysis.

Summing over all time steps leads to
\begin{equation}\label{eq:cota_pod_0}
\frac{1}{M}\sum_{n=1}^M\left\|\bu_h^n-P_r^{\mathrm{v}}\bu_h^n\right\|_{X}^2\le C_X^2.
\end{equation}

\subsection{Velocity Error Estimate for $X=H_0^1(\Omega)^{d\times d}$}

As usual in the analysis of discretizations of the Navier--Stokes equations, errors for the pressure
are bounded by velocity errors, which have been estimated before. The velocity error bounds were derived
in \cite{wir_NS} and they are provided here for completeness of presentation.

For the sake of concentrating the numerical analysis to the essential points, the grad-div POD-ROM model
studied in \cite{wir_NS} was equipped with the implicit Euler method as temporal discretization:
For $n\ge 1$, find $\bu_r^n\in {\mathcal \bU}^r$ such that
\begin{eqnarray}\label{eq:pod_method2}
\lefteqn{\hspace*{-9em}\left(\frac{\bu_r^{n}-\bu_r^{n-1}}{\Delta t},\bvar\right)+\nu(\nabla \bu_r^n,\nabla\bvar)+b\left(\bu_r^n,\bu_r^n,\bvar\right)
+\mu(\nabla \cdot\bu_r^n,\nabla \cdot\bvar)}\nonumber\\
&=&(\bff^{n},\bvar)\quad \forall\ \bvar\in {\mathcal \bU}^r.
\end{eqnarray}

%Taking $t=t_j$ in \eqref{eq:gal_grad_div} and considering the second equation we get $\bu_h^j\in\bV_{h,l}$.
%Differentiating
%the second equation in \eqref{eq:gal_grad_div} with respect to $t$ and taking again $t=t_j$, we also
%obtain $\partial_t \bu_h^j\in\bV_{h,l}$.
%As a consequence, we observe that ${\mathcal \bU}^r\subset \bV_{h,l}$,
Since $\partial_t \bu_h^j\in\bV_{h,l}$, it follows that $\bu_r^n$ belongs to the
space of discretely divergence-free functions. Hence, there is no pressure term in
\eqref{eq:pod_method2} and this equation does not need to be augmented with a requirement on the
divergence of the solution.

The velocity error estimates in \cite{wir_NS} that we present below are convection-robust, i.e., all constants do not
depend on inverse powers of the viscosity.
Denoting by $\be_r^j=\bu_r^j-P_r^{\mathrm{v}}\bu_h^j$ the following estimate holds for $j=1,\ldots,M$, see \cite[(4.21)]{wir_NS},
\begin{eqnarray}\label{eq:pre_cota_finalSUPv}
\lefteqn{\|\be_r^j\|_0^2+2\nu\sum_{j=1}^M\Delta t \|\nabla \be_r^j\|_0^2
+\mu\sum_{j=1}^M\Delta t \|\nabla \cdot \be_r^j\|_0^2}\nonumber\\
&\le& C_0\|\be_r^0\|_0^2+
C_1\sum_{k=r+1}^{d_v} \lambda_k
+C_2
(\Delta t)^2\int_0^T\|\nabla\partial_{tt}\bu_h\|_0^2\ ds.
\end{eqnarray}
The following estimate
% in a time-discrete $L^2(0,T;L^2(\Omega))$ norm
is proved in \cite[Theorem~4.1]{wir_NS}.
\begin{eqnarray}\label{eq:cota_finalSUPv}
\frac{1}{T}
\sum_{j=1}^M\Delta t \|\bu_r^j-\bu^j\|_0^2
%\nonumber\\
&\le& C_0\|\bu_r^0 -\bu_h(0)\|_0^2+
C_1\sum_{k=r+1}^{d_v} \lambda_k
+C^2(\bu,p,l+1)h^{2l}
%\nonumber\\
%&&
 \nonumber\\
&&
{}+C_2
(\Delta t)^2\int_0^T\|\nabla(\partial_{tt}\bu_h)\|_0^2\ ds.
\end{eqnarray}
%\begin{Theorem}\label{thm:bound_L2L2}
%Let $(\bu,p)$ be the velocity in the Navier--Stokes equations \eqref{NS}, let $\bu_r$ be the grad-div POD
%stabilized approximation defined in \eqref{eq:pod_method2}, assume that the solution $(\bu,p)$ of \eqref{NS} is regular enough, and assume that
%\begin{equation}\label{eq:time}
%\Delta t \left(2C_{1,inf}+\frac{C_{inf}^2}{2\mu}+\frac{2}{T}\right)\le\frac{1}{2}
%\end{equation}
%holds. Then, the
%following bound holds
%\begin{eqnarray}\label{eq:cota_finalSUPv}
%\lefteqn{\sum_{j=1}^n\Delta t \|\bu_r^j-\bu^j\|_0^2}\nonumber\\
%&\le& 3Te^{2C_u}\left[\|\be_r^0\|_0^2+
%\left(2T(\mu+C_m^2T)(3+6(T/\tau)^2)+4C_p^2(T/\tau)^2\right)\sum_{k=r+1}^{d_v} \lambda_k\right.\nonumber\\
%&& \left.{}+\left(CC_p^2T+2T(\mu+C_m^2T)\frac{16T}{3}+16T^2C_p^2\right)
%(\Delta t)^2\int_0^T\|\nabla(\partial_{tt}\bu_h)\|_0^2\ ds\right]\nonumber\\
%&& +3TC(\bu,p,l)^2h^{2l}+3TC_p^2 \left(3+6(T/\tau)^2\right)\sum_{k=r+1}^{d_v}\lambda_k.
%\end{eqnarray}
%\end{Theorem}
The term with the second order time derivative can be bounded, see \cite[Appendix]{wir_NS}. However,
a robust bound in the case $X=H_0^1(\Omega)^{d\times d}$ was obtained only for $l\ge3$. In addition,
in~\cite[Theorem~4.3]{wir_NS}, it is shown that the right-hand side of~\eqref{eq:cota_finalSUPv} is also an upper
bound of pointwise-in-time estimates, that is
\begin{eqnarray*}%\label{eq:time}
\max_{0\le n\le M} \|\bu_r^n-\bu^n\|_0^2
%\nonumber\\
&\le& C_0\|\bu_r^0 -\bu_h(0)\|_0^2+
C_1\sum_{k=r+1}^{d_v} \lambda_k
+C^2(\bu,p,l+1)h^{2l}
%\nonumber\\
%&&
 \nonumber\\
&&+C_2
(\Delta t)^2\int_0^T\|\nabla(\partial_{tt}\bu_h)\|_0^2\ ds.
\end{eqnarray*}

\subsection{Velocity Error Estimate for $X=L^2(\Omega)^d$}

The following bound can be found in \cite[(4.25)]{wir_NS}: for $j=1\ldots,M$,
\begin{eqnarray}\label{eq:pre_error3_b_L2}
\lefteqn{\|\be_r^j\|_0^2+2\nu\sum_{j=1}^M\Delta t \|\nabla \be_r^j\|_0^2
+\mu\sum_{j=1}^M\Delta t \|\nabla \cdot \be_r^j\|_0^2}\nonumber\\
&\le& C_0\|\be_r^0\|_0^2+(C_{1,1} +C_{1,2}\left\|S^{\mathrm{v}}\right\|_2)\sum_{k=r+1}^{d_v} \lambda_k
\nonumber\\
&&+(C_{2,1} + C_{2,2}\left\|S^{\mathrm{v}}\right\|_2)
(\Delta t)^2\int_0^T\|
\partial_{tt}\bu_h\|_0^2\ ds.
\end{eqnarray}
Then, it is proved in \cite[Theorem~4.6]{wir_NS} that
\begin{eqnarray*}%\label{eq:error3_b_L2}
\frac{1}{T}
\sum_{j=1}^M\Delta t \|\bu_r^j-\bu^j\|_0^2
%\nonumber\\
&\le & C_0\|\bu_r^0 -\bu_h(0)\|_0^2+
(C_{1,1} +C_{1,2}\left\|S^{\mathrm{v}}\right\|)\sum_{k=r+1}^{d_v} \lambda_k
\nonumber\\
&& +C(\bu,p,l+1)h^{2l}
\\
& & +(C_{2,1} + C_{2,2}\left\|S^{\mathrm{v}}\right\|)
(\Delta t)^2\int_0^T\|
\partial_{tt}\bu_h\|_0^2\ ds.
\nonumber
\end{eqnarray*}

\subsection{Further Estimates of Velocity Terms}

It was shown in \cite[(3.18), (3.16)]{{wir_NS}}  that
\begin{equation}
\label{bound_uh}
\max_{0\le n\le M}\|\nabla\bu_h^n\|_{L^{2d/(d-1)}}\le C_{\bu,{\rm ld}},\qquad \max_{0\le n\le M}\|\bu_h^n\|_\infty\le C_{\bu,\infty},
\end{equation}
where $C_{\bu,{\rm ld}}$ and $C_{\bu,\infty}$ depend on $C(\bu,p,3)$ and $\max_{0\le t\le T} \| \bu\|_2$, so that they
are valid if $\bu_h^n$ is replaced by~$\bu^n$.
In addition, one can find  in~\cite[(3.20), (3.25)]{{wir_NS}} that
\begin{equation}
\label{bound_Pr}
\max_{0\le n\le M}\|P_r^{\mathrm{v}} \bu_h^n\|_\infty\le C_{\rm inf},\quad \max_{0\le n\le M}
%\|\nabla \cdot P_r^{\mathrm{v}}\bu_h^n\|_{L^{2d/(d-1)}}\le
\|\nabla P_r^{\mathrm{v}}\bu_h^n\|_{L^{2d/(d-1)}}\le C_{\rm ld},
\end{equation}
where~$C_{\rm inf}$ and~$C_{\rm ld}$ depend on~$C_{\bu,{\rm ld}}$, $C_{\bu,\infty}$ and~$C_{L^2}$ (i.e., the constant~$C_X$ in~\eqref{eq:bound_2nd_term} when $X=L^2(\Omega)^d$).

From \eqref{bound_uh} and \eqref{bound_Pr},  adding and subtracting $\bu_r^n$ and applying the inverse inequality \eqref{inv},
we obtain
\begin{eqnarray*}
\|\bu_r^n\|_\infty&\le& \|\be_r^n\|_\infty+\|P_r^{\mathrm{v}} \bu_h^n\|_\infty\le c_{\mathrm{inv}} h^{-d/2}\|\be_r^n\|_0+C_{\rm inf}\\
\|\nabla \bu_r^n\|_{L^{2d/(d-1)}}&\le& \|\nabla \be_r^n\|_{L^{2d/(d-1)}}+\|\nabla P_r^{\mathrm{v}}\bu_h^n\|_{L^{2d/(d-1)}}
\\&\le& c_{\mathrm{inv}}h^{-3/2}\|\be_r^n\|_0+C_{\rm ld}.
\end{eqnarray*}
Now, in view of error bounds~\eqref{eq:pre_cota_finalSUPv} and \eqref{eq:pre_error3_b_L2} for $X=H_0^1(\Omega)^{d\times d}$ and $X=L^2(\Omega)^d$, respectively, we notice that for given $h$ it is possible to choose $r$ and~$\Delta t$ so that
\begin{equation}
\label{bound_ur}
\max_{0\le n\le M}\|\bu_r^n\|_\infty \le 2C_{\rm inf},\qquad \max_{0\le n\le M}\|\nabla \bu_r^n\|_{L^{2d/(d-1)}}
\le 2C_{\rm ld}.
\end{equation}
In the sequel, we will assume that this is the case so that estimates~\eqref{bound_ur} hold.

The next lemma provides estimates for the convective term.

\begin{lema}\label{le:nonli} There exist constants~$C_3$ and $C_4$ such that the following bounds hold for $n=1,\ldots,M$,
\begin{eqnarray}\label{eq:nonli}
|b(\bu_r^n,\bu_r^n,\bbeta)-b(\bu_h^n,\bu_h^n,\bbeta)|
&\le& C_3\|\bu_r^n-\bu_h^n\|_0\|\nabla\bbeta\|_0\quad \forall\ \bbeta\in H^1_0(\Omega)^d, \\
\label{eq:nonli2}
\left\|(\bu^n\cdot\nabla) \bu^n-(\bu_r^n\cdot \nabla) \bu_r^n\right\|_0 &\le& C_4 \left\|\nabla (\bu^n - \bu_r^n)\right\|_0,
\end{eqnarray}
where $C_3$ depends on $\|\bu_r^n\|_\infty$, $\|\nabla \bu_r^n\|_{L^{2d/(d-1)}}$, $\|\nabla\bu_h^n\|_{L^{2d/(d-1)}}$,
$\|\bu_h^n\|_\infty$ and $C_4$ depends on $\|\bu_r^n\|_\infty$, $\|\nabla \bu^n\|_{L^{2d/(d-1)}}$.
\end{lema}
\begin{proof}
We argue as in \cite[(83)]{novo_rubino}. Applying the skew-symmetric property of the trilinear form and \eqref{sob1}, we get
\begin{equation}
\label{nl00}
b(\bu_r^n,\bu_r^n,\bbeta)-b(\bu_h^n,\bu_h^n,\bbeta)=b(\bu_r^n,\bu_r^n-\bu_h^n,\bbeta)+b(\bu_r^n-\bu_h^n,\bu_h^n,\bbeta).
\end{equation}
Noticing that $b(\bu_r^n,\bu_r^n-\bu_h^n,\bbeta)=-b(\bu_r^n,\bbeta,\bu_r^n-\bu_h^n)$ and applying H\"older's
inequality gives
\begin{eqnarray}
\label{nla1}
\lefteqn{\nonumber
\left| b(\bu_r^n,\bu_r^n-\bu_h^n,\bbeta)\right\| } \\
&\le &
\left(\|\bu_r^n\|_\infty\|\nabla \bbeta\|_0+\frac{1}{2}\|\nabla \cdot\bu_r^n\|_{L^{2d/(d-1)}}\|\bbeta\|_{L^{2d}}\right)\|\bu_r^n-\bu_h^n\|_{0}
\nonumber\\
&\le&
C\left(\|\bu_r^n\|_\infty+\frac{1}{2}\|\nabla \cdot\bu_r^n\|_{L^{2d/(d-1)}}C_p^{\frac{3-d}{2}}\right)\|\nabla\bbeta\|_0\|\bu_r^n-\bu_h^n\|_{0},
\end{eqnarray}
where, in the last inequality, we have applied Sobolev's estimate~\eqref{sob1} and the Poincar\'e inequality
\eqref{poincare} to bound
$\|\bbeta\|_{L^{2d}}\le C \|\nabla \bbeta\|_0$ if $d=3$. And, if $d=2$, we used in addition a Sobolev
interpolation inequality to obtain $\|\bbeta\|_{L^{4}} \le C\|\bbeta\|_{1/2}
\le C(\|\bbeta\|_{0}\|\nabla\bbeta\|_{0})^{1/2} \le CC_p^{1/2} \|\nabla\bbeta\|_{0}$.
For the second term on the right-hand side of~\eqref{nl00}, by writing
\[ b(\bu_r^n-\bu_h^n,\bu_h^n,\bbeta)=\frac{1}{2} ((\bu_r^n-\bu_h^n)\cdot \nabla\bu_h^n,\bbeta) -\frac{1}{2}
((\bu_r^n-\bu_h^n)\cdot \nabla\bbeta,\bu_h^n),
\]
we find with similar arguments that
\begin{eqnarray}
\label{nla2}
\lefteqn{\left| b(\bu_r^n-\bu_h^n,\bu_h,\bbeta)\right|} \nonumber\\
&\le& C\left(
\|\nabla \bu_h^n\|_{L^{2d/(d-1)}}\|\bbeta\|_{L^{2d}} + \| \bu_h^n\|_{\infty}\|\nabla\bbeta\|_{0}\right) \|\bu_r^n-\bu_h^n\|_{0}
\nonumber\\
&\le& C\left( \|\nabla \bu_h^n\|_{L^{2d/(d-1)}}C_p^{\frac{3-d}{2}} + \|\bu_h^n\|_{\infty}\right)\|\nabla\bbeta\|_{0} \|\bu_r^n-\bu_h^n\|_{0}.
\end{eqnarray}
Thus, \eqref{eq:nonli} follows from \eqref{nla1}, \eqref{nla2} and by applying estimates~\eqref{bound_uh}, \eqref{bound_ur}
one obtains that the constant $C_3$ is bounded.

Similarly, by using the identity
\[
(\bu^n\cdot\nabla) \bu^n-(\bu_r^n\cdot \nabla) \bu_r^n = ((\bu^n-\bu_r^n)\cdot\nabla) \bu^n+(\bu_r^n\cdot \nabla) (\bu^n - \bu_r^n),
\]
it follows that
\begin{eqnarray*}
\left\|(\bu^n\cdot\nabla) \bu^n-(\bu_r^n\cdot \nabla) \bu_r^n\right\|_0
&\le& \left\| \bu^n - \bu_r^n\right\|_{L^{2d}} \left\|\nabla \bu^n\right\|_{L^{2d/(d-1)}}
\nonumber\\
&& + \left\|\bu_r^n\right\|_{\infty} \left\|\nabla(\bu^n - \bu_r^n)\right\|_0.
\end{eqnarray*}
Then, \eqref{eq:nonli2} is obtained from  the fact that, as a consequence of~\eqref{sob1} and~\eqref{poincare},
$\left\| \bu^n - \bu_r^n\right\|_{L^{2d}}\le C C_p^{(3-d)/2} \left\|\nabla (\bu^n - \bu_r^n)\right\|_0$.
From \eqref{bound_ur}, one infers the boundedness of
the constant in \eqref{eq:nonli2}.
\end{proof}

\begin{remark}\label{re:C3} We notice
that~\eqref{eq:nonli} is valid if either~$\bu_h^n$ or~$\bu_r^n$ is replaced by~$\bu$, because that proof did
not apply any argument that is valid only for finite element functions.
\end{remark}

In \cite[Appendix]{wir_NS},  It is shown  that
\[
\int_{0}^T \|\partial_t\bu(t)-\partial_t\bu_h(t)\|_0^2\ dt \le C_Ah^{2l-2}.
\]
With similar arguments, one can also prove the following lemma.

\begin{lema} \label{le:bosco} There exists a constant $C_A$ such that for $\Delta t \le Ch$ the following bound holds
\begin{equation}\label{pre_defi_1}
 \sum_{j=1}^n\Delta t \|\partial_t\bu^j-\partial_t\bu_h^j\|_0^2\le C_A   h^{2{l-2}}.
\end{equation}
\end{lema}

\begin{proof}
We will prove the following bound for $\be_h=\bu_h-\bs_h^m$, with $\bs_h^m$ being the modified Stokes projection defined in \eqref{stokespro_mod_def}
\begin{equation}
\label{cota}
\sum_{j=1}^n \Delta t \left\| \partial_t\be_h^j\right\|_0^2 \le C h^{2(l-1)}.
\end{equation}
From \eqref{cota} and \eqref{stokespro_mod} we reach \eqref{pre_defi_1}.

To derive \eqref{cota}, we first show that
\begin{equation}
\label{cota2}
\sum_{j=1}^n \left| \Delta t \left\| \partial_t\be_h^j\right\|_0^2 -\int_{t_{j-1}}^{t_j} \left\| \partial_t\be_h(t)\right\|_0^2\,dt   \right|\le C h^{2(l-3/2)}\Delta t,
\end{equation}
from which, applying~\eqref{cota3} below, the estimate~\eqref{cota} follows as long as $\Delta t\le Ch$.

To prove~\eqref{cota2}, we first apply the fundamental theorem of calculus and the Cauchy--Schwarz inequality, to obtain
\begin{align*}
\biggr| \Delta t \left\| \partial_t\be_h^j\right\|_0^2 & -\int_{t_{j-1}}^{t_j} \left\| \partial_t\be_h(t)\right\|_0^2\ dt \biggr| =\left|\int_{t_{j-1}}^{t_j}  \left(\left\| \partial_t\be_h(t_j)\right\|_0^2-\left\| \partial_t\be_h(t)\right\|_0^2\right)\ dt\right|\\
&=
\left|\int_{t_{j-1}}^{t_j} \int_{t}^{t_j} 2\left( \partial_t \be_h,\partial_{tt}\be_h\right)\ ds\ dt\right|
\nonumber\\
&\le 2\int_{t_{j-1}}^{t_j} \biggl(\int_{t_{j-1}}^{t_j} \left\| \partial_t\be_h\right\|_0^2\ ds\biggr)^{1/2}
 \biggl(\int_{t_{j-1}}^{t_j }\left\| \partial_{tt}\be_h\right\|_0^2\ ds\biggr)^{1/2}\ dt
 \nonumber\\
 &=2\Delta t\biggl(\int_{t_{j-1}}^{t_j} \left\| \partial_t\be_h\right\|_0^2\ ds\biggr)^{1/2}
 \biggl(\int_{t_{j-1}}^{t_j }\left\| \partial_{tt}\be_h\right\|_0^2\ ds\biggr)^{1/2}.
\nonumber\\
\end{align*}
Now, applying the Cauchy--Schwarz inequality for sums yields
\begin{align*}
\sum_{j=1}^n \biggr| \Delta t \left\| \partial_t\be_h^j\right\|_0^2 &-\int_{t_{j-1}}^{t_j} \left\| \partial_t\be_h(t)\right\|_0^2\,dt   \biggr|
\nonumber\\
 &\le 2\Delta t \biggl(\sum_{j=1}^n
\int_{t_{j-1}}^{t_j} \left\| \partial_t\be_h\right\|_0^2\ ds\biggr)^{1/2}
\biggl(\sum_{j=1}^n \int_{t_{j-1}}^{t_j }\left\| \partial_{tt}\be_h\right\|_0^2\ ds\biggr)^{1/2}
\nonumber\\
& = 2\Delta t \biggl(\int_{t_{0}}^{t_n} \left\| \partial_t\be_h\right\|_0^2\ ds\biggr)^{1/2}
 \biggl(\int_{t_{0}}^{t_n }\left\| \partial_{tt}\be_h\right\|_0^2\ ds\biggr)^{1/2}.
 \nonumber
\end{align*}
Since the bounds
\begin{equation}
\label{cota3}
\int_{t_{0}}^{t_n} \left\| \partial_t\be_h\right\|_0^2\ ds \le C h^{2(l-1)}, \quad
\int_{t_{0}}^{t_n }\left\| \partial_{tt}\be_h\right\|_0^2\ ds \le C h^{2(l-2)}
\end{equation}
are proved in \cite[(A.1)-(A.2)]{wir_NS}, we conclude \eqref{cota2}.
\end{proof}

\begin{remark}
The assumption $\Delta t \le C h$ required in the proof of Lemma \ref{le:bosco} can be suppressed with a more complicated proof and for that reason it will not be assumed in the rest of the paper. We have decided to include a simpler proof to show \eqref{pre_defi_1} in order not to increase the length of the paper.
 \end{remark}

\section{Reduced Order Pressure Approximations}\label{sec:press_rom_ana}

In this section, we analyze two different approaches for approximating the pressure within
the context of a POD-ROM, namely the supremizer enrichment algorithm following \cite{novo_rubino}
and a stabilization-motivated approach introduced in \cite{John_et_al_vp}.

Define the space spanned by the snapshots of the FOM pressure by
\[
{\mathcal  W}=\mbox{span}\left\{p_h^1,\ldots,p_h^M\right\},
\]
i.e.,
the situation of using
the FOM pressure solutions at all time instants, after the initial time, i.e., $M=n$, will be considered.
For the sake of brevity, we consider only the situation that the $L^2(\Omega)$ inner product is used for the pressure. 
Thus, let $K^{\mathrm p}$ be the correlation matrix
$K^{\mathrm p}=((k_{i,j}^{\mathrm{p}}))\in {\mathbb R}^{M\times M}$ with
\[
k_{i,j}^{\mathrm{p}}=\frac{1}{M}\left(p_h^{j},p_h^{i}\right).
\]
We denote by
$ \gamma_1\ge  \gamma_2,\ldots\ge \gamma_{d_p}>0$ the positive eigenvalues of $K^{\mathrm p}$ and by
$\bw_1,\ldots,\bw_{d_p}\in {\mathbb R}^{M}$ the associated eigenvectors.
An orthonormal basis of $\mathcal  W$ is then given by
\begin{equation*}
\psi_k=\frac{1}{\sqrt{M}}\frac{1}{\sqrt{\gamma_k}}\sum_{j=1}^M w_k^j p_h(t_{j},\cdot),
\end{equation*}
where $w_k^j$  is the $j$-th component of  $\bw_k$. The following relation is proved in
\cite[Proposition~1]{kunisch}.
\begin{equation}\label{eq:cota_pod_0_pre}
\frac{1}{M}\sum_{j=1}^M\left\|p_h^{j}-\sum_{k=1}^r\left(p_h^{j},\psi_k\right)\psi_k\right\|_{0}^2 = \sum_{k=r+1}^{d_p}\gamma_k.
\end{equation}
Similarly as for the velocity, we define the stiffness matrix  $S^{\mathrm{p}}=((s_{i,j}^{\mathrm{p}}))\in {\mathbb R}^{d_p\times d_p}$, with
$s_{i,j}^{\mathrm{p}}=(\nabla \psi_j,\nabla \psi_i)$. Then the analog to \eqref{eq:inv_S} is 
\begin{equation}\label{eq:inv_Sp}
\|\nabla q \|_0\le \sqrt{\|S^{\mathrm{p}}\|_2}\|q\|_0\quad \forall\ q \in {\mathcal W}.
\end{equation}

The POD-ROM pressure space will be denoted by
\[
{\mathcal  W}^r= \mbox{span}\left\{\psi_1,\psi_2,\ldots,\psi_r\right\},\quad 1\le r\le d_p,
\]
where $r$ is the same dimension as for the velocity POD-ROM space.  The  orthogonal projection onto ${\mathcal  W}^r$
with respect to the $L^2(\Omega)$ inner product is denoted by $P_r^{\mathrm{p}}\ : \  L^2_0(\Omega)  \to {\mathcal  W}^r$.

\subsection{A Supremizer Enrichment Algorithm}

Following \cite{schneier,novo_rubino} we study a first way to compute a POD-ROM pressure approximation.
The %principle
main
idea of the supremizer enrichment algorithm consists in computing a basis of a $r$-dimensional subspace of the
orthogonal complement (with respect to the inner product in~$H^1_0(\Omega)^d$) of~${\boldsymbol V_h^l}$,
%which has the same dimension $r$ as ${\mathcal  W}^r$
and then to compute a POD-ROM pressure via the discrete
momentum equation, where the already computed POD-ROM velocity is an input data.

Given a function $p_r\in{\mathcal  W}^r$, we consider the following problem: find $\bw_h\in {\boldsymbol X}_h^l$ such that
\begin{equation*}%\label{eq:supremizer}
(\nabla \bw_h,\nabla\bv_h)=-(\nabla \cdot \bv_h,p_r),\quad \forall\ \bv_h \in {\boldsymbol X}_h^l.
\end{equation*}
Solving this equation for each basis function $\psi_k$, $k=1,\cdots,r$, gives a set of linearly independent
solutions. Since $p_r\in Q_h^l$, it follows that
\[
(\nabla \bw_h,\nabla\bv_h)=0,\quad \forall\ \bv_h \in {\boldsymbol V}_h^l.
\]
Applying a Gram--Schmidt orthonormalization procedure to the set of linearly independent solutions, with respect to the inner product of $H^1_0(\Omega)^d$,
results in a set of basis functions $\left\{\bzeta_1,\bzeta_2,\ldots,\bzeta_{r}\right\}$. Denoting by
\[
{\boldsymbol S}^r=\mbox{span}\left\{\bzeta_1,\bzeta_2,\ldots,\bzeta_{r}\right\}\subset\left({\boldsymbol V}_h^l\right)^\bot
\subset {\boldsymbol X}_h^l,
\]
then the following inf-sup stability condition holds for the spaces ${\boldsymbol S}^{r}$ and ${\mathcal  W}^r$, see \cite[Lemma~4.2]{schneier}:
\begin{equation}\label{eq:inf_sup_supre}
\beta_{r}= \inf_{\psi\in {\mathcal  W}^r}\sup_{\bzeta\in{{\boldsymbol S}^{r}}}
\frac{(\nabla \cdot \bzeta,\psi)}{\|\nabla\bzeta\|_{0}
\|\psi\|_0} \geq \beta_{\rm{is}},
\end{equation}
where $\beta_{\rm{is}}$ is the constant in the inf-sup condition \eqref{lbbh}.
Using the space ${\boldsymbol S}^{r}$, a pressure $p_r^n\in {\mathcal  W}^r$ can be computed satisfying for all $\bzeta\in {\boldsymbol S}^{r}$
\begin{eqnarray}\label{eq:pres}
(\nabla \cdot \bzeta,p_r^n)=\left(\frac{\bu_r^{n}-\bu_r^{n-1}}{\Delta t},\bzeta\right)+b(\bu_r^n,\bu_r^n,\bzeta)
+\mu(\nabla \cdot \bu_r^n,\nabla \cdot \bzeta)
-(\bff^{n},\bzeta).
\end{eqnarray}
Notice that the viscous term is not present since $\bzeta\in {\boldsymbol S}^{r}\subset
({\boldsymbol V}_h^l)^\bot$ and~$\bu_h^n\in{\boldsymbol V}_h^l$ implies that $\nu(\nabla \bu_h^n,
\nabla \bzeta)=0$.

Define the following norms
\begin{equation}\label{eq:supr_alg_norms}
\|\bv\|_{{\boldsymbol S}^{r,*}}=\sup_{\bzeta\in {\boldsymbol S}^{r}}\frac{(\bv,\bzeta)}{\|\nabla \bzeta\|_0},
\quad
\|\bv\|_{{\mathcal  \bU}^{r,*}}=\sup_{\bzeta\in {\mathcal  \bU}^{r}}\frac{(\bv,\bzeta)}{\|\nabla \bzeta\|_0},
\end{equation}
denote  by $\alpha \in [0,1)$ the constant in the strengthened Cauchy--Schwarz inequality between the spaces ${\mathcal  \bU}^r$ and ${\boldsymbol S}^{r}$, that is,
\[
\left|({\boldsymbol\zeta},\boldsymbol{\varphi})\right|\le \alpha\left\|{\boldsymbol\zeta} \right\|_0\left\| {\boldsymbol\varphi}\right\|_0,\qquad {\boldsymbol\zeta}\in{\boldsymbol S}^r,\quad{\boldsymbol\varphi}\in {\boldsymbol U}^r,
\]
and by
\begin{equation}\label{cotaCr}
C_r^{H^1}= \sum_{k=1}^r\|\nabla \bvar_k\|_0.
\end{equation}
The right-hand side in \eqref{cotaCr} is bounded by  $ r \|S^{\mathrm{v}}\|_2^{1/2}$ in the case $X=L^2(\Omega)^d$ and
by $r$ in the case $X=H_0^1(\Omega)^{d\times d}$, for the velocity projections.
Since in practice $\|S^{\mathrm{v}}\|_2 > 1$, e.g., see \cite{novo_rubino},
the constant $C_r^{H^1}$ is usually smaller in the case $X=H_0^1(\Omega)^{d\times d}$.
Letting $\bvar\in {\mathcal  \bU}^{r}$, it holds, see \cite[Lemma~5.14]{schneier},
\begin{eqnarray}\label{eq:cota_dual}
\|\bvar\|_{{\boldsymbol S}^{r,*}}\le \alpha C_p C_r^{H^1}\|\bvar\|_{{\mathcal  \bU}^{r,*}},
\end{eqnarray}
where $C_p$ is the constant in the Poincar\'e inequality \eqref{poincare}.

\begin{Theorem} There exists a constant~$C>0$ such that for using $X=H^1_0(\Omega)^{d\times d}$ the following bound holds
\begin{eqnarray}\label{eq:cota_pre_supr}
\lefteqn{\sum_{j=1}^n\Delta t\left\|  p^j-p_r^j\right\|_{0}^2 \le 2C_{\rm press}^2(\bu,p,l+1) h^{2l-1}
+C\frac{T}{\beta_r^2}\sum_{k=r+1}^{d_p}\gamma_k }
\\
&&  + C\left(1+\frac{C_pC_r^{H^1}}{\beta_r}\right)\left[ (TC_3^2+\nu+\mu )C_{\ref{eq:pre_cota_finalSUPv}} +
(C_3^2C_p^2+\nu^2+\mu^2)T C_{\ref{eq:cota_pod_0}} \right], \nonumber
%+  C\Gamma_r^2 \left(
%C_0\|\be_r^0\|_0^2 +
%C_1\sum_{k=r+1}^{d_v} \lambda_k
%+C_2
%(\Delta t)^2\int_0^T\|\nabla(\partial_{tt}\bu_h)\|_0^2\ ds\right),
% \nonumber
\end{eqnarray}
%where
%\[
%\Gamma_r = \left(1+\frac{C_pC_r^{H^1}}{\beta_r}\right)(TC_3^2+\nu+\mu ),
%\]
%$C_0$, $C_1$, and~$C_2$ are the constants in~\eqref{eq:pre_cota_finalSUPv}, $C_3$ is
%the constant in~\eqref{eq:nonli}, and $C_{\rm press}$ is the constant in~\eqref{eq:cota_pre}.
where $C_{\ref{eq:pre_cota_finalSUPv}}$ and $C_{\ref{eq:cota_pod_0}}$ are the bounds on the right-hand sides of
\eqref{eq:pre_cota_finalSUPv} and \eqref{eq:cota_pod_0}, respectively, $C_3$ is
the constant in~\eqref{eq:nonli}, and $C_{\rm press}$ is the constant in~\eqref{eq:cota_pre}

If $X=L^2(\Omega)^d$, the same bound holds if $C_{\ref{eq:pre_cota_finalSUPv}}$ is replaced by the bound on
the right-hand side of \eqref{eq:pre_error3_b_L2} and the last term is replaced by
\[
(C_3^2 +(\nu^2+\mu^2)c_{\mathrm{inv}}^2h^{-2})T C_{\ref{eq:cota_pod_0}}.
\]
\end{Theorem}

\begin{proof} The proof is obtained arguing as in \cite[Theorem 5.4]{novo_rubino}.

Adding and subtracting $(\nabla \cdot \bzeta,P_r^{\mathrm{p}} p_h^n)$
in the FOM problem \eqref{eq:gal_grad_div} at time $t_n$ and observing again that the viscous
term vanishes, yields
\begin{eqnarray}\label{eq:p_hres}
\left(\nabla \cdot \bzeta,P_r^{\mathrm{p}} p_h^n\right)&=&\left(\partial_t\bu_h^n,\bzeta\right)
+b(\bu_h^n,\bu_h^n,\bzeta)
+\mu(\nabla \cdot \bu_h^n,\nabla \cdot \bzeta)
-(\bff^{n},\bzeta)\nonumber\\
&&+\left(\nabla \cdot \bzeta, P_r^{\mathrm{p}} p_h^n-p_h^n\right) \quad \forall\ \bzeta\in {\boldsymbol S}^{r}.
\end{eqnarray}
Decomposing the error $p_r^n-p_h^n = z_r^n + \xi_h^n$ with
\[
z_r^n=p_r^n-P_r^{\mathrm{p}} p_h^n \in \mathcal  W^r,\quad \xi_h^n=P_r^{\mathrm{p}} p_h^n-p_h^n \in \mathcal W,
\]
subtracting the POD-ROM problem \eqref{eq:p_hres} from \eqref{eq:pres} and applying \eqref{eq:inf_sup_supre}, \eqref{eq:supr_alg_norms},
\eqref{eq:cota_dual},
the Cauchy--Schwarz inequality, and  \eqref{diver_vol} leads to
\begin{eqnarray}\label{eq:erpre1}
\|z_r^n\|_0 &\le& \frac{1}{\beta_{r}}\left(\alpha C_p C_r^{H^1} \left\|\frac{\bu_r^n-\bu_r^{n-1}}{\Delta t }-\partial_t\bu_h^n\right\|_{{\mathcal  \bU}^{r,*}}\right.\\
&&\left. + \sup_{\bzeta\in {\boldsymbol S}^{r}}\frac{b(\bu_r^n,\bu_r^n,\bzeta)-b(\bu_h^n,\bu_h^n,\bzeta)}{\|\nabla \bzeta\|_0}
+\mu  \|\nabla \cdot(\bu_r^n-\bu_h^n)\|_0+\|\xi_h^n\|_0\right).\nonumber
\end{eqnarray}
We will now bound the terms on the right-hand side of \eqref{eq:erpre1}.
For the first one, we subtract the velocity FOM equation \eqref{eq:gal_grad_div2} from
the velocity POD-ROM equation \eqref{eq:pod_method2} for an arbitrary test function from $\mathcal  \bU^r$
to get
\begin{eqnarray*}%\label{pre_su_1}
\left\|\frac{\bu_r^n-\bu_r^{n-1}}{\Delta t }-\partial_t\bu_h^n\right\|_{{\mathcal  \bU}^{r,*}}
&\le&  \nu\|\nabla (\bu_h^n-\bu_r^n)\|_0+\mu\|\nabla \cdot(\bu_h^n-\bu_r^n)\|_0\nonumber\\
&&+\sup_{\bbeta\in {\mathcal  \bU}^r}\frac{|b(\bu_h^n,\bu_h^n,\bbeta)-b(\bu_r^n,\bu_r^n,\bbeta)|}{\|\nabla \bbeta\|_0}.
\end{eqnarray*}
Concerning the last term on the right-hand side, we apply~\eqref{eq:nonli}. We also apply~\eqref{eq:nonli} to the second term on the right-hand side of \eqref{eq:erpre1}. Thus, we have
\begin{eqnarray*}%\label{eq:erpre2}
\|z_r^n\|_0&\le& \frac{1}{\beta_{r}}\Big[\alpha C_p C_r^{H^1} \left(C_3\|\bu_r^n-\bu_h^n\|_0+\nu\|\nabla (\bu_h^n-\bu_r^n)\|_0+\mu\|\nabla \cdot(\bu_h^n-\bu_r^n)\|_0\right)\nonumber\\
&&+C_3\|\bu_r^n-\bu_h^n\|_0+\mu\|\nabla \cdot(\bu_h^n-\bu_r^n)\|_0+\|\xi_h^n\|_0\Big].
\end{eqnarray*}
Squaring both sides and using the triangle inequality $\|p_r^j-p_h^j\|_0^2 \le 2\|z_h\|^2+ 2\|\xi^j\|_0^2$,
multiplying by~$\Delta t$ and summing over all time steps yields
\begin{eqnarray*}%\label{cota_pre_last}
\lefteqn{\sum_{j=1}^n\Delta t\|p_r^j-p_h^j\|_0^2}\nonumber \\
&\le&C\left(1+\frac{C_pC_r^{H^1}}{\beta_r}\right)^2 \left[C_3^2\sum_{j=1}^n\Delta t\|\bu_r^j-\bu_h^j\|_0^2
+\nu\sum_{j=1}^n\Delta t\nu\|\nabla(\bu_r^j-\bu_h^j)\|_0^2\right.\nonumber\\
&&\left.{}+\mu\sum_{j=1}^n\Delta t\mu\|\nabla\cdot(\bu_r^j-\bu_h^j)\|_0^2\right]+ \frac{C}{\beta_r^2}
\sum_{j=1}^n \Delta t \|P_r^{\mathrm{p}} p_h^j-p_h^j\|_0^2.
\end{eqnarray*}
Now the proof is finished by applying the triangle inequality $\|p_j- p_r^j\|_0^2 \le 2\|p_j- p_h^j\|_0^2 +2\|p_h^j- p_r^j\|_0^2$ and utilizing estimates \eqref{eq:cota_pod_0},
\eqref{eq:cota_pre}, \eqref{eq:pre_cota_finalSUPv}, \eqref{eq:pre_error3_b_L2} and~\eqref{eq:cota_pod_0_pre}.
In order to apply \eqref{eq:cota_pod_0} for  $X=L^2(\Omega)^d$ one can use inequality
\eqref{eq:inv_S} and for $X=H_0^1(\Omega)^{d\times d}$ one has to utilize Poincar\'e's inequality.
\end{proof}

As already noted above, the constant $ C_r^{H^1}$ is usually smaller in the case $X=H_0^1(\Omega)^{d\times d}$.

Comparing \eqref{eq:cota_pre_supr} with the respective bound in \cite{novo_rubino}, one notes 
that the power of $h$ in  \eqref{eq:cota_pre_supr} is $2l-1$ instead of $2l$ as in 
\cite[(89)]{novo_rubino}. The reason is that for obtaining  \eqref{eq:cota_pre_supr} the result 
from \cite{cor} was used, which corrects an older statement that was still applied in the numerical analysis
of \cite{novo_rubino}.

%Concerning pointwise-in-time estimates, we can
%bound all but the last term on the right-hand side of \eqref{eq:erpre2} to get pointwise-in-time
%estimates. However,
%this term is only bounded in \eqref{eq:cota_pod_0_pre}
%inside a sum of terms of the time instants. Of course, it is
%\[
%\|\xi_h^n\|_0 \le M \sum_{k=r+1}^{d_p}\gamma_k = \frac{T}{\Delta t} \sum_{k=r+1}^{d_p}\gamma_k.
%\]
%The inverse of the time step appears also in the final estimate.
%An alternative approach could be to use for the pressure space snapshots approaching the time derivative,
%as we did with the velocity, we need  to take also a sum in \eqref{eq:erpre2} to get
%a pointwise-in-time error bound. THINK ABOUT
%THE POSSIBILITY OF USING APPROXIMATIONS TO THE DERIVATIVE IN THE PRESSURE SPACE.

\subsection{Stabilization-Motivated Pressure ROM}

In this section we consider a different procedure for computing a POD-ROM pressure, which was proposed in \cite{John_et_al_vp}.
The starting point of this method consists in formally considering a residual-based stabilization of a
coupled velocity-pressure POD-ROM system. Since the functions from the velocity POD-ROM space are
discretely divergence-free, the system decouples and one obtains a Poisson-type equation for a
POD-ROM pressure whose right-hand side contains terms with the velocity POD-ROM approximation.
This method was recently analyzed in \cite{samu_et_al_pres}. However, the error analysis in
\cite{samu_et_al_pres} has some drawbacks that will be overcome in the subsequent analysis.

The stabilization-motivated approach results in the following problem for a POD-ROM pressure:
Find $p_r^n\in {\mathcal  W}^r$  such that for all $\psi\in {\mathcal  W}^r$
\begin{eqnarray}\label{pre_moti}
\lefteqn{
\sum_{K\in \mathcal T_h}\tau_K\left(\nabla p_r^n,\nabla \psi\right)_K}\nonumber\\
&=&-\sum_{K\in \mathcal T_h}
\tau_K\left(\frac{\bu_r^n-\bu_r^{n-1}}{\Delta t}+(\bu_r^n\cdot \nabla) \bu_r^n-\nu\Delta \bu_r^n-\bff^{n},\nabla\psi\right)_K,
\end{eqnarray}
where $\tau_K$ are positive parameters of order $h_K^2$:
\begin{equation}\label{tau_K}
c_1 h_K^2\le \tau_K\le c_2 h_K^2,\quad K\in \mathcal T_h.
\end{equation}
We notice that, as it is usually the case in the literature, the nonlinear terms do not involve $\frac{1}{2}(\nabla\cdot \bu_r^n \bu_r^n)$. Also we observe that although the POD velocities $\bu_r^j$, $j=0,\ldots N$, are discretely divergence-free, the term $(\bu_r^n-\bu_r^{n-1})/\Delta t$ cannot be omitted in~\eqref{pre_moti} unless all the $\tau_K$ are equal.

It was already mentioned in \cite{John_et_al_vp} that on quasi-uniform triangulations the choice of $\tau_K$ in terms 
of $h_K$ is of minor importance, since these parameters appear on both sides of \eqref{pre_moti}. The following 
analysis can be performed with a generic parameters. However, the concrete norm on the left-hand sides of the error estimates 
in Theorem~\ref{thm:conv_smrom} and the order of convergence on the right-hand sides depend on the concrete choice given in \eqref{tau_K}. We also performed numerical simulations with the choice
$\tau_K = C h_K$, as it was proposed in \cite{John_et_al_vp}. The results are very similar to 
those presented in Section~\ref{sec:numres} and they will not be presented here for the sake 
of brevity.

To facilitate the notations in the forthcoming, quite technical, numerical analysis, the following
quantities are defined: 
\begin{align}
\label{hatC_0}
\hat C_0 &= \left(\max\left\{\nu,\mu\right\} + T(C_3^2+C_4^2)\right)C_0,\\
%\label{hatC_1}
\hat C_1 &= \left(\max\left\{\nu,\mu\right\} + T(C_3^2+C_4^2)\right)C_1 \nonumber\\
&\hphantom{{}=}{}+ T\left((\nu^2+\mu^2 + C_p^2 C_3^2+ C_4^2h^2)\rho^2  +C_p^2\left({h}/{\tau}\right)^2\right),\qquad\\
\label{hatC_2}
\hat C_2 &= \left(\max\left\{\nu,\mu\right\} + T(C_3^2+C_4^2)\right)C_2\nonumber\\
&\hphantom{{}=}{} +  T^2\left(\nu^2+\mu^2 + C_p^2C_3^2
+ C_4^2h^2\right) + C_p^2h^2,
\\
\label{hatC11}
\hat C_{1,1} &= \left(\max\left\{\nu,\mu\right\} + T(C_3^2+C_4^2)\right)C_{1,1}
+ TC_3^2\rho^2 + T(h/\tau)^2,\\
\hat C_{1,2} &=  \left(\max\left\{\nu,\mu\right\} + T(C_3^2+C_4^2)\right)C_{1,2} + T (\nu^2+\mu^2+C_4^2h^2)\rho^2, \\%+C_3^2h^2\right),
\hat C_{2,1} &= \left(\max\left\{\nu,\mu\right\} + T(C_3^2+C_4^2)\right)C_{2,1}+ T^2\left(C_3^2+C_4^2\right)+ h^2,\\
\hat C_{2,2} &= \left(\max\left\{\nu,\mu\right\} + T(C_3^2+C_4^2)\right)C_{2,2} + T^2 (\nu^2+\mu^2+C_4^2h^2),
%+C_3^2(1+\rho^2)h^2\right),
\label{hatC22}\\
\hat C_{3,1} &= T^2 \left(\nu^2+\mu^2+ C_4^2h^2\right),
\label{hatC31}
\end{align}
where $C_0,\ldots,C_4$, $C_{i,j}$, $i,j\in\{1,2\}$, are defined in \eqref{eq:pre_cota_finalSUPv}, 
\eqref{eq:pre_error3_b_L2}, \eqref{eq:nonli}, and \eqref{eq:nonli2}, 
respectively. Each of the quantities defined in \eqref{hatC_0}--\eqref{hatC31} is dimensionally
correct, i.e., all terms that are added have the same physical unit. This statement can be checked 
by straightforward but lengthy calculations, whose presentation is omitted for the sake of brevity.

Before proving the main result of the section we prove an auxiliary lemma.
\begin{lema}\label{le:long_proof} For $\be_r^j=\bu_r^j-P_r^{\mathrm{v}}\bu_h^j\in {\mathcal  \bU}^r$, $j=1,\ldots,M$, and
$X= H_0^1(\Omega)^{d\times d}$, the following bound holds
\begin{equation}\label{pre_defi_3}
\sum_{j=1}^n\Delta t \left\|\frac{\be_r^j -\be_r^{j-1}}{\Delta t}\right\|_0^2
\le \frac{C}{h^2}\biggl(
\hat C_0 \left\|\be_r^0\right\|_0^2 + \hat C_1\sum_{k=r+1}^{d_v} \lambda_k + \hat C_2(\Delta t)^2 \int_0^T\left\|\nabla\partial_{tt}\bu_h\right\|_0^2\ ds\biggr),
\end{equation}
where $\hat C_0$, $\hat C_1$ and~$\hat C_2$ are given by \eqref{hatC_0}--\eqref{hatC_2}. 
For~$X=L^2(\Omega)^d$ the following bound holds
\begin{eqnarray}\label{pre_defi_3bis}
&&\sum_{j=1}^n\Delta t \left\|\frac{\be_r^j -\be_r^{j-1}}{\Delta t}\right\|_0^2
\le \frac{C}{h^2}\left(
\hat C_0 \left\|\be_r^0\right\|_0^2 + (\hat C_{1,1}+\hat C_{1,2}\left\| S^{\mathrm{v}}\right\|_2)\sum_{k=r+1}^{d_v} \lambda_k \right.\\
& & \left. +  (\hat C_{2,1}+\hat C_{2,2}\left\| S^{\mathrm{v}}\right\|_2)(\Delta t)^2 \int_0^T\left\| \partial_{tt}\bu_h\right\|_0^2\ ds+\hat C_{3,1} (\Delta t)^2 \int_0^T\left\| \nabla\partial_{tt}\bu_h\right\|_0^2\ ds\right),\nonumber
\end{eqnarray}
where the constants $\hat C_{i,j}$ %, $i \in \{1,2,3\}$, $j\in\{1,2\}$, 
are given in \eqref{hatC11}--\eqref{hatC31}.
\end{lema}

\begin{proof}
Let $X= H_0^1(\Omega)^{d\times d}$. The following relation was derived in
\cite[(4.5)]{wir_NS}
\begin{eqnarray}\label{eq:error_need}
\lefteqn{\left(\frac{\be_r^j -\be_r^{j-1}}{\Delta t},\bvar\right)+\nu(\nabla  \be_r^j ,\nabla\bvar)+\mu(\nabla \cdot \be_r^j,\nabla \cdot\bvar)}
\nonumber\\
&=&\left(\partial_t\bu_h^j-\frac{P_r^{\mathrm{v}} \bu^{j}_h-P_r^{\mathrm{v}} \bu^{j-1}_h}{\Delta t},\bvar\right)
 -\mu(\nabla \cdot\bbeta_h^j,\nabla \cdot \bvar)\nonumber\\
 &&+b(\bu^{j}_h,\bu^{j}_h,\bvar)-b(\bu_r^j,\bu_r^j,\bvar),\quad \forall\ \bvar\in {\mathcal  \bU}^r,
\end{eqnarray}
where $\bbeta_h^j=P_r^{\mathrm{v}} \bu_h^j-\bu_h^j\in \mathcal \bU$.

Choosing $\bvar=(\be_r^j-\be_r^{j-1})/\Delta t$ and using \eqref{eq:nonli} and the inverse inequality \eqref{inv} gives
\begin{eqnarray*}
\left\|\frac{\be_r^j -\be_r^{j-1}}{\Delta t}\right\|_0 &\le& h^{-1}c_{\mathrm{inv}}
\left(\nu\|\nabla \be_r^j\|_0+ \mu \|\nabla\cdot \be_r^j\|_0 +\mu\|\nabla\cdot \bbeta_h^j\|_0\right)
\nonumber\\
&&+ C_3c_{\mathrm{inv}}h^{-1} \|\bu_h^j-\bu_r^j\|_0+ \left\|\partial_t\bu_h^j-\frac{P_r^{\mathrm{v}} \bu^{j}_h-P_r^{\mathrm{v}} \bu^{j-1}_h}{\Delta t}\right\|_0.
\end{eqnarray*}
Taking the square and the sum over the time instants and using \eqref{diver_vol} leads to
\begin{eqnarray}\label{aux2}
\lefteqn{\hspace*{-3em}\sum_{j=1}^n\Delta t \left\|\frac{\be_r^j -\be_r^{j-1}}{\Delta t}\right\|_0^2
\le C \frac{\nu^2}{h^2}\sum_{j=1}^n\Delta t \|\nabla \be_r^j\|_0^2
+C  \frac{\mu^2}{h^2}\sum_{j=1}^n\Delta t \|\nabla\cdot\be_r^j\|_0^2} \nonumber\\
&& +C \frac{\mu^2}{h^2} \sum_{j=1}^n\Delta t \|\nabla(P_r^{\mathrm{v}}\bu_h^j-\bu_h^j)\|_0^2
+ C\frac{C_3^2}{h^2}\sum_{j=1}^n\Delta t \|\bu_h^j-\bu_r^j\|_0^2 \\
&& +C
\sum_{j=1}^n\Delta t \left\|\partial_t\bu_h^j-\frac{P_r^{\mathrm{v}} \bu_h^j-P_r^{\mathrm{v}} \bu_h^{j-1}}{\Delta t}\right\|_0^2:=X_1+X_2+X_3+X_4+X_5.\nonumber
\end{eqnarray}
We notice that the first two sums on the right-hand side of~\eqref{aux2} appear on the left-hand side of~\eqref{eq:pre_cota_finalSUPv}, so that the first two terms on the right-hand side of~\eqref{aux2} can be bounded by the right-hand side of~\eqref{eq:pre_cota_finalSUPv} times~$\max\{\nu,\mu\}/h^2$, hence
\begin{eqnarray}\label{eq:x1}
\lefteqn{X_1+X_2}\\
& \le& C\frac{\max\{\nu,\mu\}}{h^2}\left(C_0\|\be_r^0\|_0^2+
C_1\sum_{k=r+1}^{d_v} \lambda_k
+C_2(\Delta t)^2\int_0^T\|\nabla(\partial_{tt}\bu_h)\|_0^2 \ ds \right).\nonumber
\end{eqnarray}
We observe that $\max\{\nu,\mu\}C_i$, $i=0,1,2$, is part of of the values $\hat C_i$, $i=0,1,2$, respectively, so that
$X_1+X_2$ is bounded by the right-hand side of~\eqref{pre_defi_3}.

For the third term on the right-hand side of~\eqref{aux2}, we apply \eqref{eq:cota_pod_0}, so that it can be bounded by $C(\mu^2T/h^2)C_{H_0^1}^2$.  To bound the fourth term we add and subtract $P_r^{\mathrm{v}} \bu_h^j$
and apply Poincar\'e's inequality to obtain
\[
 \sum_{j=1}^n\Delta t \|\bu_h^j-\bu_r^j\|_0^2 \le  2C_p^2\sum_{j=1}^n\Delta t \|\nabla(\bu_h^j-P_r^{\mathrm{v}} \bu_h^j)\|_0^2 + 2 T \max_{1\le j\le n} \|\be_r^j\|_0^2.
\]
The first term on the right-hand side is then estimated by \eqref{eq:cota_pod_0}. The second term is bounded in \eqref{eq:pre_cota_finalSUPv}.
Combining the previous estimates yields
\begin{eqnarray}\label{eq:x3}
\lefteqn{X_3+X_4}\nonumber\\
&\le& C\frac{(\mu^2+C_p^2C_3^2)T}{h^2}\left(
\rho^2
\sum_{k={r+1}}^{d_v}\lambda_k
+ T(\Delta t)^2
\int_0^T\|\nabla\partial_{tt}\bu_{h}\|^2 ds\right)
\\
&& + C\frac{TC_3^2}{h^2}\left(C_0\|\be_r^0\|_0^2+
C_1\sum_{k=r+1}^{d_v} \lambda_k
+C_2(\Delta t)^2\int_0^T\|\nabla\partial_{tt}\bu_h\|_0^2 \ ds \right).\nonumber
%\nonumber\\
%&&\le \frac{C}{h^2}\left(\hat C_0\|\be_r^0\|_0^2+
%\hat C_1\sum_{k=r+1}^{d_v} \lambda_k
%+\hat C_2(\Delta t)^2\int_0^T\|\nabla(\partial_{tt}\bu_h)\|_0^2 \ ds \right).
\end{eqnarray}
We observe that $TC_3^2C_i$, $i=0,1,2$ is part of the values~$\hat C_i$, $i=0,1,2$, respectively; similarly,
$(\mu^2 + C_p^2C_3^2)T^i\rho^{2(2-i)}$, $i=1,2$, is part of the constants~$\hat C_i$, $i=1,2$. Thus, $X_3+X_4$ can be bounded
by the right-hand side of~\eqref{pre_defi_3}.

Finally,
in the case $X=H_0^1(\Omega)^{d\times d}$, the following bound for
the last term in \eqref{aux2} is given in  \cite[(4.17)-(4.18)]{wir_NS},
\begin{eqnarray*}
\sum_{j=1}^n\Delta t \left\|\partial_t\bu_h^j-\frac{P_r^{\mathrm{v}} \bu_h^j-P_r^{\mathrm{v}} \bu_h^{j-1}}{\Delta t}\right\|_0^2
&\le& \frac{2 TC_p^2}{\tau^2}\sum_{k=r+1}^{d_v}\lambda_k
\nonumber\\
&&
+ C C_p^2 (\Delta t)^2\int_0^T\|\nabla \partial_{tt}\bu_h\|_0^2 \ ds,
\end{eqnarray*}
where we notice that $TC_p^2(h/\tau)^2$ is part of~$\hat C_1$ and~$C_p^2h^2$ is part of~$\hat C_2$.
Consequently, it is
\begin{eqnarray}\label{pre_defi_2}
\lefteqn{
\sum_{j=1}^n\Delta t \left\|\partial_t\bu_h^j-\frac{P_r^{\mathrm{v}} \bu_h^j-P_r^{\mathrm{v}} \bu_h^{j-1}}{\Delta t}\right\|_0^2}
\nonumber\\
&\le&
 \frac{C}{h^2}\left(\hat C_1 \sum_{k=r+1}^{d_v}\lambda_k + \hat C_2(\Delta t)^2\int_0^T\|\nabla \partial_{tt}\bu_h\|_0^2 \ ds\right),
\end{eqnarray}
so that
$X_5$ is also bounded by the right-hand side of~\eqref{pre_defi_3}.
Thus, inserting \eqref{eq:x1}, \eqref{eq:x3} and \eqref{pre_defi_2} into \eqref{aux2}, we obtain~\eqref{pre_defi_3}.
\medskip

If $X=L^2(\Omega)^d$ one starts analogously to the other case, but due to the different inner product, instead of \eqref{eq:error_need}
now one gets
\begin{eqnarray}\label{eq:error_need_l2}
\lefteqn{\left(\frac{\be_r^j -\be_r^{j-1}}{\Delta t},\bvar\right)+\nu(\nabla  \be_r^j ,\nabla\bvar)+\mu(\nabla \cdot \be_r^j,\nabla \cdot\bvar)}
\nonumber\\
&=&\left(\partial_t\bu_h^j-\frac{\bu^{j}_h- \bu^{j-1}_h}{\Delta t},\bvar\right)
 -\nu(\nabla \bbeta_h^j,\nabla \bvar)-\mu(\nabla \cdot\bbeta_h^j,\nabla \cdot \bvar)\nonumber\\
 &&+b(\bu^{j}_h,\bu^{j}_h,\bvar)-b(\bu_r^j,\bu_r^j,\bvar),\quad \forall\ \bvar\in {\mathcal  \bU}^r,
\end{eqnarray}
compare \cite[(4.23)]{wir_NS}. Notice that~\eqref{eq:error_need_l2} is as~\eqref{eq:error_need} but with  the extra term $-\nu(\nabla\bbeta_h^j,\nabla\bvar)$ on the right-hand side.
Arguing as before, we obtain
\begin{eqnarray}\label{aux3}
\lefteqn{\sum_{j=1}^n\Delta t \left\|\frac{\be_r^j -\be_r^{j-1}}{\Delta t}\right\|_0^2
\le C \frac{\nu^2}{h^2}\sum_{j=1}^n\Delta t \|\nabla \be_r^j\|_0^2
+C \frac{\mu^2}{h^2}\sum_{j=1}^n\Delta t \|\nabla\cdot\be_r^j\|_0^2} \nonumber\\
&& +C\frac{\nu^2+\mu^2}{h^2} \sum_{j=1}^n\Delta t \|\nabla(P_r^{\mathrm{v}}\bu_h^j-\bu_h^j)\|_0^2
+ C\frac{C_3^2}{h^2} \sum_{j=1}^n\Delta t \|\bu_h^j-\bu_r^j\|_0^2 \nonumber\\
&& +C
\sum_{j=1}^n\Delta t \left\|\partial_t\bu_h^j-\frac{\bu^{j}_h- \bu^{j-1}_h}{\Delta t}\right\|_0^2:=X_1+X_2+X_3+X_4+X_5,
\end{eqnarray}
which is as~\eqref{aux2} but with~$\mu^2$ replaced by~$(\nu^2+\mu^2)$ in the third term on the right-hand side. Thus, we argue similarly as we did with~\eqref{aux2}.
Now, we use~\eqref{eq:pre_error3_b_L2} instead of~\eqref{eq:pre_cota_finalSUPv} to bound the first two terms on the right-hand side of~\eqref{aux3}
\begin{eqnarray*}%\label{eq:x1l2_0}
X_1+X_2%\nonumber\\
&\le& C\frac{\max\{\nu,\mu\}}{h^2}\left( C_0\|\be_r^j\|_0^2+(C_{1,1} +C_{1,2}\left\|S^{\mathrm{v}}\right\|_2)\sum_{k=r+1}^{d_v} \lambda_k
\right.\nonumber\\
&&\left.
+(C_{2,1} + C_{2,2}\left\|S^{\mathrm{v}}\right\|_2)
(\Delta t)^2\int_0^T\|
\partial_{tt}\bu_h\|_0^2\ ds\right).
\end{eqnarray*}
We have already commented that $\max\{\nu,\mu\}C_0$ is part of $\hat C_0$. So it is the case
for $\max\{\nu,\mu\}C_{i,j}$, $i,j=1,2$, with respect to~$\hat C_{i,j}$, $i,j=1,2$, respectively, so that
we have
\begin{eqnarray}\label{eq:x1l2}
X_1+X_2
%&&\ \le C\frac{\max(\nu,\mu)}{h^2}\left( C_0\|\be_r^j\|_0+(C_{1,1} +C_{1,2}\left\|S^{\mathrm{v}}\right\|)\sum_{k=r+1}^{d_v} \lambda_k
%\right.\nonumber\\
%&&\left.
%+(C_{2,1} + C_{2,2}\left\|S^{\mathrm{v}}\right\|)
%(\Delta t)^2\int_0^T\|
%\partial_{tt}\bu_h\|_0^2\ ds\right).
%\nonumber\\
&\le & \frac{C}{h^2}\left( \hat C_0\|\be_r^j\|_0^2+(\hat  C_{1,1} +\hat C_{1,2}\left\|S^{\mathrm{v}}\right\|_2)\sum_{k=r+1}^{d_v} \lambda_k
\right.\nonumber\\
&&\left.
+(\hat C_{2,1} + \hat C_{2,2}\left\|S^{\mathrm{v}}\right\|_2)
(\Delta t)^2\int_0^T\|
\partial_{tt}\bu_h\|_0^2\ ds\right).
\end{eqnarray}
Bounding the third term on the right-hand side of~\eqref{aux3} is achieved by applying
\eqref{eq:bound_2nd_term_pre} with $Y=H_0^1$ and then using \eqref{eq:inv_S} and  \eqref{eq:cota_pod_deriv}
to estimate the first two terms in \eqref{eq:bound_2nd_term_pre}.  To bound the third term in \eqref{eq:bound_2nd_term_pre}, we use the triangle inequality, apply again \eqref{eq:inv_S} and 
also the stability in $L^2$ of the $P_r^{\mathrm{v}}$ projection 
\begin{eqnarray*}
 \|\nabla(\partial_{tt}\bu_{h}-P_r^{\mathrm{v}}\partial_{tt}\bu_h)\|_0^2&\le& 2\|\nabla P_r^{\mathrm{v}}\partial_{tt}\bu_h\|_0^2
 +2 \|\nabla\partial_{tt}\bu_{h}\|_0^2\\
 &\le& 2\|S^{\mathrm{v}}\|_2\|\partial_{tt}\bu_h\|_0^2 + 2 \|\nabla\partial_{tt}\bu_{h}\|_0^2.
 \end{eqnarray*}
Collecting the estimates leads to 
\begin{eqnarray}\label{eq:need0}
X_3&\le& C \frac{\nu^2+\mu^2}{h^2}  T\left(\rho^2\|S^{\mathrm{v}}\|_2\sum_{k={r+1}}^{d_v}\lambda_k\right.\nonumber\\
&& \left. + T (\Delta t)^2 \left(\|S^{\mathrm{v}}\|_2 \int_0^T\|\partial_{tt}\bu_{h}\|_0^2\ ds
+ \int_0^T\|\nabla\partial_{tt}\bu_{h}\|_0^2\ ds\right)\right).
\end{eqnarray}
%\begin{eqnarray}\label{eq:need0}
%X_3\le C \frac{\nu^2+\mu^2}{h^2}  T\left(\rho^2\|S^{\mathrm{v}}\|_2\sum_{k={r+1}}^{d_v}\lambda_k
%+ 6T(\Delta t)^2 \int_0^T\|\nabla\partial_{tt}\bu_{h}\|_0^2\ ds\right).
%\end{eqnarray}
We notice that $(\nu^2+\mu^2)T^i \rho^{2(2-i)}$ is part of $\hat C_{i,2}$ and $(\nu^2+\mu^2)T$ is 
part of $\hat C_{3,1}$. Then the right-hand side above is included
on the right-hand side of \eqref{pre_defi_3bis}.

After having
applied the triangle inequality, the fourth term is estimated in the following way
\[
\sum_{j=1}^n\Delta t \|\bu_h^j-\bu_r^j\|_0^2 \le  2 \sum_{j=1}^n\Delta t \|\bu_h^j-P_r^{\mathrm{v}} \bu_h^j\|_0^2 + 2 T  \max_{1\le j\le n} \|\be_r^j\|_0^2,
\]
and then we use \eqref{eq:cota_pod_0} and   \eqref{eq:pre_error3_b_L2} to obtain
\begin{eqnarray*}
\lefteqn{X_4 \le C\frac{TC_3^2}{h^2}\left(\rho^2\sum_{k={r+1}}^{d_v}\lambda_k
+  T(\Delta t)^2 \int_0^T\|\partial_{tt}\bu_{h}\|_0^2\ ds+C_0\|\be_r^j\|_0^2
\right.}\nonumber\\
&& +(C_{1,1} +C_{1,2}\left\|S^{\mathrm{v}}\right\|_2)\sum_{k=r+1}^{d_v} \lambda_k
+(C_{2,1} + C_{2,2}\left\|S^{\mathrm{v}}\right\|_2)
(\Delta t)^2\int_0^T\|
\partial_{tt}\bu_h\|_0^2\ ds\biggr).\nonumber
\end{eqnarray*}
We remark that $C_3^2T^i\rho^{2(2-i)}$ is part of~$\hat C_{i,1}$, $C_3^2TC_0$ is part of~$\hat C_0$, and $C_3^2TC_{i,j}$,
$i,j=1,2$, is part of~$\hat C_{i,j}$, $i,j=1,2$, so that we have,
\begin{eqnarray}\label{eq:x3l2}
X_4
&\le & \frac{C}{h^2}\left( \hat C_0\|\be_r^j\|_0^2+(\hat  C_{1,1} +\hat C_{1,2}\left\|S^{\mathrm{v}}\right\|_2)\sum_{k=r+1}^{d_v} \lambda_k
\right.\nonumber\\
&&\left.
{}+(\hat C_{2,1} + \hat C_{2,2}\left\|S^{\mathrm{v}}\right\|_2)
(\Delta t)^2\int_0^T\|
\partial_{tt}\bu_h\|_0^2\ ds\right).
\end{eqnarray}
Finally, the last term on the right-hand side of~\eqref{aux3} has the following simple bound:
\begin{equation}\label{eq:need_last}
X_5\le C  (\Delta t)^2\int_0^T \|\partial_{tt}\bu_h\|_0^2\ ds.
\end{equation}
Altogether, for $X=L^2(\Omega)^d$, from \eqref{eq:x1l2}, \eqref{eq:need0}, \eqref{eq:x3l2} and \eqref{eq:need_last} we derived the bound \eqref{pre_defi_3bis},
and the proof is finished.
\end{proof}

\begin{Theorem} \label{thm:conv_smrom}
Assume that $\nu\le h\| \bu\|_{L^\infty(L^\infty)}$. The following bound holds when $X=H_0^1(\Omega)^{d\times d}$
\begin{eqnarray}\label{cota_th_pre}
\lefteqn{\sum_{j=1}^n\Delta t \sum_{K\in \mathcal T_h}\tau_K\|\nabla(p^j-p_r^j)\|_{0,K}^2}\nonumber\\ 
&\le& C\Bigg[ C_A   h^{2l} +   C_4^2 T C^2(\bu,p,l+1) h^{2l}+  C_1^2(\bu,p,l+1) h^{2l}  \nonumber \\
&& + C_{\rm press}^2(\bu,p,l+1) h^{2l-1}  + \hat C_0 \|\be_r^0\|_0^2 + \hat C_1\sum_{k=r+1}^{d_v} \lambda_k \\
&&  + 
h^{2l} \sum_{j=1}^n\Delta t \left\| p^j\right\|_{l}^2 + h^2 \|S^{\mathrm{p}}\|_2 T \sum_{k=r+1}^{d_p}\gamma_k
+ \hat C_2(\Delta t)^2 \int_0^T\left\| \nabla \partial_{tt}\bu_h\right\|_0^2\ ds
\Bigg],\nonumber
\end{eqnarray}
where $\hat C_0$, $\hat C_1$ and~$\hat C_2$ are defined in \eqref{hatC_0}--\eqref{hatC_2} and
\begin{equation}\label{laC1u}
C_1^2(\bu,p,l+1) =\nu^2 T\|\bu \|^2_{L^\infty(H^{l+1})}+\|\bu\|^2_{L^\infty(L^\infty)}TC^2(\bu,p,l+1)
\end{equation}
with  $C(\bu,p,l+1)$ being the constant in \eqref{eq:cota_grad_div}.
In the case that $X=L^2(\Omega)^d$ is used, the following error estimate is valid
\begin{eqnarray}\label{cota_th_pre_l2}
\lefteqn{\sum_{j=1}^n\Delta t \sum_{K\in \mathcal T_h}\tau_K\|\nabla(p^j-p_r^j)\|_{0,K}^2}\nonumber\\ 
&\le& C\Bigg[ C_A   h^{2l} + C_4^2 T C^2(\bu,p,l+1) h^{2l} + C_1^2(\bu,p,l+1) h^{2l}\nonumber\\&&  +  C_{\rm press}^2(\bu,p,l+1) h^{2l-1} +
\hat C_0 \|\be_r^0\|_0^2 + (\hat C_{1,1}+\hat C_{1,2}\left\| S^{\mathrm{v}}\right\|)\sum_{k=r+1}^{d_v} \lambda_k \nonumber \\
&&  + h^{2l} \sum_{j=1}^n\Delta t \left\| p^j\right\|_{l}^2 + h^2 \|S^{\mathrm{p}}\|_2 T \sum_{k=r+1}^{d_p}\gamma_k \\
&& +  (\hat C_{2,1}+\hat C_{2,2}\left\| S^{\mathrm{v}}\right\|)(\Delta t)^2 \int_0^T\left\| \partial_{tt}\bu_h\right\|_0^2\ ds
+\hat C_{3,1}(\Delta t)^2 \int_0^T\left\| \nabla\partial_{tt}\bu_h\right\|_0^2\ ds
\Bigg], \nonumber
\end{eqnarray}
with  $\hat C_{i,j}$ %, $i\in \{1,2,3\}$, $j\in\{1,2\}$, 
given in \eqref{hatC11}--\eqref{hatC31}.
\end{Theorem}

\begin{proof}
For the solution of the Navier--Stokes equations at time $t_n$, it holds for all $\psi \in L_0^2(\Omega)$ that
\begin{eqnarray*}%\label{pre_sol}
\lefteqn{\sum_{K\in \mathcal T_h}\tau_K\left(\nabla P_r^{\mathrm{p}} p^n,\nabla \psi\right)_K
=  \sum_{K\in \mathcal T_h}\tau_K \left(\nabla P_r^{\mathrm{p}} p^n,\nabla \psi\right)_K}\nonumber\\
&& -\sum_{K\in \mathcal T_h}\tau_K\left(\partial_t\bu^n+(\bu^n\cdot \nabla) \bu^n-\nu\Delta \bu^n-\bff^{n}+\nabla p^n,\nabla\psi\right)_K.
\end{eqnarray*}
Let $z_r^n= p_r^n- P_r^{\mathrm{p}} p^n \in {\mathcal  W}^r$, then it follows with \eqref{pre_moti} that
for all $\psi \in {\mathcal  W}^r \subset L_0^2(\Omega)$ 
\begin{eqnarray*}%\label{error_pre}
\lefteqn{\sum_{K\in \mathcal T_h}\tau_K(\nabla z_r^n,\nabla \psi)_K
= \sum_{K\in \mathcal T_h} \tau_K\left(\partial_t\bu^n-\frac{\bu_r^n-\bu_r^{n-1}}{\Delta t},\nabla\psi\right)_K
} \nonumber\\
&&+\sum_{K\in \mathcal T_h}
\tau_K\left((\bu^n\cdot\nabla) \bu^n-(\bu_r^n\cdot \nabla) \bu_r^n,\nabla\psi\right)_K\\
&&\quad+\sum_{K\in \mathcal T_h}
\tau_K\left(\nu \Delta \bu_r^n-\nu\Delta \bu^n,\nabla\psi\right)_K
+\sum_{K\in \mathcal T_h}
\tau_K\left(\nabla (p^n-P_r^{\mathrm{p}} p^n),\nabla\psi\right)_K.\nonumber
\end{eqnarray*}
Choosing $\psi=z_r^n$, applying the Cauchy--Schwarz inequality 
%yields
%\begin{eqnarray*}
%\lefteqn{\left(\sum_{K\in \mathcal T_h}\tau_K\|\nabla z_r^n\|_{0,K}^2\right)^{1/2}
%\le \left(\sum_{K\in \mathcal T_h}\tau_K\left\|\partial_t\bu^n-\frac{\bu_r^n-\bu_r^{n-1}}{\Delta t}\right\|_{0,K}%^2\right)^{1/2}}\\
%&&+\left(\sum_{K\in \mathcal T_h}\tau_K\left\|(\bu^n\cdot\nabla) \bu^n-(\bu_r^n\cdot \nabla) \bu_r^n\right\|_{0,K}^2\right)^{1/2}\\
%&&+\left(\sum_{K\in \mathcal T_h}\tau_K\left\|\nu \Delta \bu_r^n-\nu\Delta \bu^n\right\|_{0,K}^2\right)^{1/2}
% +\left(\sum_{K\in \mathcal T_h}\tau_K\left\|\nabla (p^n-P_r^{\mathrm{p}} p^n)\right\|_{0,K}^2\right)^{1/2}.
%\end{eqnarray*}
and  the bounds of the parameter in~\eqref{tau_K} gives
\begin{eqnarray*}
\lefteqn{
\sum_{K\in \mathcal T_h}\tau_K\|\nabla z_r^n\|_{0,K}^2}\\
&\le& C h^2\left\|\partial_t\bu^n-\frac{\bu_r^n-\bu_r^{n-1}}{\Delta t}\right\|_{0}^2
+Ch^2\left\|(\bu^n\cdot\nabla) \bu^n-(\bu_r^n\cdot \nabla) \bu_r^n\right\|_0^2\\
&&+Ch^2\sum_{K\in \mathcal T_h}\|\nu \Delta \bu_r^n-\nu\Delta \bu^n\|_{0,K}^2
+Ch^2\left\|\nabla (p^n-P_r^{\mathrm{p}} p^n)\right\|_{0}^2.
\end{eqnarray*}

We will consider the error in the following discrete $L^2$ norm in time
\[
\sum_{j=1}^n\Delta t \sum_{K\in \mathcal T_h}\tau_K\|\nabla z_r^j\|_{0,K}^2.
\]
Summing over all time instants leads to
\begin{eqnarray*}
\lefteqn{
\sum_{j=1}^n\Delta t \sum_{K\in \mathcal T_h}\tau_K\|\nabla z_r^j\|_{0,K}^2}\nonumber \\
&\le& C h^2 \Bigg[\sum_{j=1}^n\Delta t \left\|\partial_t\bu^j-\frac{\bu_r^j-\bu_r^{j-1}}{\Delta t}\right\|_{0}^2
+\sum_{j=1}^n\Delta t \left\|(\bu^j\cdot\nabla) \bu^j-(\bu_r^j\cdot \nabla) \bu_r^j\right\|_0^2\nonumber\\
&&+\sum_{j=1}^n\Delta t \sum_{K\in \mathcal T_h}\|\nu \Delta \bu_r^j-\nu\Delta \bu^j\|_{0,K}^2
+\sum_{j=1}^n\Delta t \left\|\nabla (p^j-P_r^{\mathrm{p}} p^j)\right\|_{0}^2\Bigg].
\end{eqnarray*}
Applying the triangle inequality and utilizing the previous estimate gives 
\begin{eqnarray}\label{pre_defi}
\lefteqn{
\sum_{j=1}^n\Delta t \sum_{K\in \mathcal T_h}\tau_K\|\nabla(p^j-p_r^j)\|_{0,K}^2}\nonumber \\
&\le & 2 \sum_{j=1}^n\Delta t \sum_{K\in \mathcal T_h}\tau_K\|\nabla(p^j-P_r^{\mathrm{p}} p^j)\|_{0,K}^2 + 
\sum_{j=1}^n\Delta t \sum_{K\in \mathcal T_h}\tau_K\|\nabla z_r^j\|_{0,K}^2 \nonumber \\
&\le& C h^2 \Bigg[\sum_{j=1}^n\Delta t \left\|\partial_t\bu^j-\frac{\bu_r^j-\bu_r^{j-1}}{\Delta t}\right\|_{0}^2
+\sum_{j=1}^n\Delta t \left\|(\bu^j\cdot\nabla) \bu^j-(\bu_r^j\cdot \nabla) \bu_r^j\right\|_0^2\nonumber\\
&&+\sum_{j=1}^n\Delta t \sum_{K\in \mathcal T_h}\|\nu \Delta \bu_r^j-\nu\Delta \bu^j\|_{0,K}^2
+\sum_{j=1}^n\Delta t \left\|\nabla (p^j-P_r^{\mathrm{p}} p^j)\right\|_{0}^2\Bigg].
\end{eqnarray}
We will bound the terms on the right-hand side of \eqref{pre_defi}.
\medskip

As in the proof of Lemma~\ref{le:long_proof}, we start with the case  $X= H_0^1(\Omega)^{d\times d}$.

\emph{First term on the right-hand side of \eqref{pre_defi}.}
To bound the first one, we
start with the triangle inequality
\begin{eqnarray}\label{eq:stab_mot_pres_time_terms}
\lefteqn{
\sum_{j=1}^n\Delta t \left\|\partial_t\bu^j-\frac{\bu_r^j-\bu_r^{j-1}}{\Delta t}\right\|_{0}^2}\nonumber\\
&\le& C \sum_{j=1}^n\Delta t \|\partial_t\bu^j-\partial_t\bu_h^j\|_0^2
+
C \sum_{j=1}^n\Delta t \left\|\partial_t\bu_h^j-\frac{P_r^{\mathrm{v}} \bu_h^j-P_r^{\mathrm{v}} \bu_h^{j-1}}{\Delta t}\right\|_0^2 \nonumber
\\
&&+C\sum_{j=1}^n\Delta t \left\|\frac{P_r^{\mathrm{v}} \bu_h^j-P_r^{\mathrm{v}} \bu_h^{j-1}}{\Delta t}-\frac{\bu_r^j-\bu_r^{j-1}}{\Delta t}\right\|_0^2.
\end{eqnarray}
The first term on the right-hand side of~\eqref{eq:stab_mot_pres_time_terms} is bounded in~\eqref{pre_defi_1},
the second one in \eqref{pre_defi_2}, and the third one in \eqref{pre_defi_3}. Summarizing these estimates 
yields
\begin{eqnarray}\label{eq:eins_h1}
\lefteqn{\sum_{j=1}^n\Delta t \left\|\partial_t\bu^j-\frac{\bu_r^j-\bu_r^{j-1}}{\Delta t}\right\|_{0}^2
\le C_A   h^{2{l-2}}}
\nonumber\\
&&+\frac{C}{h^2}\left(
\hat C_0 \|\be_r^0\|_0^2 + \hat C_1\sum_{k=r+1}^{d_v} \lambda_k + \hat C_2(\Delta t)^2 \int_0^T\left\| \nabla \partial_{tt}\bu_h\right\|_0^2\ ds\right).
\end{eqnarray}

\emph{Second term on the right-hand side of \eqref{pre_defi}.}
Bounding the second term on the right-hand side of \eqref{pre_defi} starts with \eqref{eq:nonli2}
and the triangle inequality
\begin{eqnarray}\label{eq:pres_est_second_term}
\lefteqn{\sum_{j=1}^n\Delta t \left\|(\bu^j\cdot\nabla) \bu^j-(\bu_r^j\cdot \nabla) \bu_r^j\right\|_0^2 }\nonumber\\
&\le & 2 C_4^2\left( \sum_{j=1}^n\Delta t \|\nabla(\bu^j-\bu_h^j)\|_0^2 +
\sum_{j=1}^n\Delta t \|\nabla(\bu_h^j-\bu_r^j)\|_0^2
\right).
\end{eqnarray}
The first term on the right-hand side is bounded by \eqref{eq:cota_grad_div}.
Using again the triangle inequality and the inverse estimate \eqref{inv} for the second term yields 
\begin{equation}\label{eq:pres_est_second_term0}
 \sum_{j=1}^n\Delta t \|\nabla(\bu_h^j-\bu_r^j)\|_0^2 \le  2\sum_{j=1}^n\Delta t \|\nabla(\bu_h^j-P_r^{\mathrm{v}} \bu_h^j)\|_0^2 +  C \frac{T}{h^2} \max_{1\le j\le n} \|\be_r^j\|_0^2.
\end{equation}
Now, the estimate of this term is finished by applying \eqref{eq:cota_pod_0} and \eqref{eq:pre_cota_finalSUPv}.
Summarizing the bounds gives 
\begin{eqnarray*}\lefteqn{\sum_{j=1}^n\Delta t \left\|(\bu^j\cdot\nabla) \bu^j-(\bu_r^j\cdot \nabla) \bu_r^j\right\|_0^2} \\
&\le& 2 C_4^2 T C^2(\bu,p,l+1) h^{2l-2}\\
&& + C C_4^2 T \left(\rho^2\sum_{k={r+1}}^{d_v}\lambda_k
+ T(\Delta t)^2 \int_0^T\|\nabla \partial_{tt}\bu_{h}\|_0^2\ ds\right)\\
&& + C C_4^2 \frac{T}{h^2} \left( C_0\|\be_r^0\|_0^2+
C_1\sum_{k=r+1}^{d_v} \lambda_k +C_2
(\Delta t)^2\int_0^T\|\nabla \partial_{tt}\bu_h\|_0^2\ ds  \right).
\end{eqnarray*}
We notice that $C_4^2TC_i$, $i=0,1,2$ is part of $\hat C_i$,  $i=0,1,2$, respectively, $TC_4^2\rho^2h^2$ is part of $\hat C_1$ and $T^2C_4^2h^2$ of $\hat C_2$, so that
\begin{eqnarray}\label{eq:zwei}
\lefteqn{\sum_{j=1}^n\Delta t \left\|(\bu^j\cdot\nabla) \bu^j-(\bu_r^j\cdot \nabla) \bu_r^j\right\|_0^2
\le  2 C_4^2 T C^2(\bu,p,l+1) h^{2l-2}}
\nonumber\\
&&+  \frac{C}{h^2}\left(
\hat C_0 \|\be_r^0\|_0^2 + \hat C_1\sum_{k=r+1}^{d_v} \lambda_k + \hat C_2(\Delta t)^2 \int_0^T\left\|\nabla \partial_{tt}\bu_h\right\|_0^2\ ds\right).
\end{eqnarray}

\emph{Third term on the right-hand side of \eqref{pre_defi}.}
Applying triangle inequality and the inverse estimate \eqref{inv} yields
\begin{eqnarray}\label{eq_y}
\lefteqn{\sum_{j=1}^n\Delta t \sum_{K\in \mathcal T_h}\|\nu \Delta \bu_r^j-\nu\Delta \bu^j\|_{0,K}^2
\le C \nu^2 \sum_{j=1}^n\Delta t \sum_{K\in \mathcal T_h}\|\Delta (\bu_r^j-P_r^{\mathrm{v}} \bu_h^j)\|_{0,K}^2} \nonumber\\
&& +C \nu^2 \sum_{j=1}^n\Delta t\sum_{K\in \mathcal T_h}\|\Delta(P_r^{\mathrm{v}} \bu_h^j-\bu_h^j)\|_{0,K}^2
+C \nu^2 \sum_{j=1}^n\Delta t\sum_{K\in \mathcal T_h}\|\Delta(\bu_h^j-\bu^j)\|_{0,K}^2\nonumber\\
&\le& C \frac{\nu^2}{h^2}  \sum_{j=1}^n\Delta t \|\nabla \be_r^j\|_{0}^2 +C \frac{\nu^2}{h^2} \sum_{j=1}^n\Delta t\|\nabla(P_r^{\mathrm{v}} \bu_h^j-\bu_h^j)\|_{0}^2\nonumber\\
&&+C \nu^2 \sum_{j=1}^n\Delta t\sum_{K\in \mathcal T_h}\|\Delta(\bu_h^j-\bu^j)\|_{0,K}^2=Y_1+Y_2+Y_3.
\end{eqnarray}
To bound the last term in \eqref{eq_y}, we start by adding and subtracting the Lagrangian interpolant and 
utilizing the triangle inequality
\[
\|\Delta(\bu_h^j-\bu^j)\|_{0,K}^2 \le 2\|\Delta(\bu_h^j-I_h(\bu^j))\|_{0,K}^2+ 2\|\Delta(I_h(\bu^j)-\bu^j)\|_{0,K}^2.
\]
The second term on the right-hand side is bounded with \eqref{interp}. For the first one, 
we apply the inverse inequality \eqref{inv} and add and subtract~$\bu^j$ to arrive at 
\[
\|\Delta(\bu_h^j-I_h(\bu^j))\|_{0,K}^2\le Ch^{-2} \|\nabla(\bu_h^j-\bu^j)\|_{0,K}^2 + Ch^{-2}
\|\nabla(\bu^j-I_h(\bu^j))\|_{0,K}^2,
\]
so that the last term can be bounded again with \eqref{interp}. 
Collecting terms, applying \eqref{eq:cota_grad_div}, and 
assuming $\nu\le h\|\bu\|_{L^\infty(L^\infty)}$ yields 
\begin{eqnarray*}%\label{eq_y3}
Y_3&\le& C \nu^2 h^{2l-2} T\left\|\bu\right\|_{L^\infty(H^{l+1})}^2 + C \frac{\nu^2}{h^2} \sum_{j=1}^n\Delta t\left\|\nabla(\bu_h^j-\bu^j)\right\|_0^2\nonumber\\
&\le&C \left(\nu^2 h^{2l-2} T\left\|\bu\right\|_{L^\infty(H^{l+1})}^2 + \|\bu\|_{L^\infty(L^\infty)}^2T C^2(\bu,p,l+1) h^{2l-2}\right).
\end{eqnarray*}
In view of \eqref{laC1u}, we obtain 
\begin{equation}
\label{eq_y3}
Y_3 \le   C_1^2(\bu,p,l+1) h^{2l-2}.
\end{equation}
The first term on the right-hand side of \eqref{eq_y} is bounded in \eqref{eq:pre_cota_finalSUPv} while the second one is bounded in \eqref{eq:cota_pod_0}, hence it is 
\begin{eqnarray*}\label{eq_y1y2}
Y_1+Y_2 &\le& C \frac{\nu}{h^2} \left(C_0\|\be_r^0\|_0^2+
C_1\sum_{k=r+1}^{d_v} \lambda_k
+C_2
(\Delta t)^2\int_0^T\|\nabla(\partial_{tt}\bu_h)\|_0^2\ ds\right)\nonumber\\
&&\ +C\frac{\nu^2 T}{h^2}\left(\rho^2\sum_{k={r+1}}^{d_v}\lambda_k +  T(\Delta t)^2 \int_0^T\|\nabla (\partial_{tt}\bu_{h})\|_0^2\ ds\right).
\end{eqnarray*}
By definition, $\nu C_i$, $i=0,1,2$, is part of $\hat C_i$, $i=0,1,2$, respectively, as well as 
$\nu^2 T^i \rho^{2(2-i)}$, $i=1,2$, is part of~$\hat C_i$, $i=1,2$, respectively.
We conclude for the third term on the right-hand side of~\eqref{pre_defi} that 
\begin{eqnarray}\label{eq:dreih1}
\lefteqn{ \sum_{j=1}^n\Delta t \sum_{K\in \mathcal T_h}\|\nu \Delta \bu_r^j-\nu\Delta \bu^j\|_{0,K}^2
\le    C_1^2(\bu,p,l+1) h^{2l-2} }  \nonumber\\
&&+
\frac{C}{h^2}\left(
\hat C_0 \|\be_r^0\|_0^2 + \hat C_1\sum_{k=r+1}^{d_v} \lambda_k + \hat C_2(\Delta t)^2 \int_0^T\left\|\nabla \partial_{tt}\bu_h\right\|_0^2\ ds\right).
\end{eqnarray}

\emph{Fourth term on the right-hand side of \eqref{pre_defi}.} Also the estimate of this term is 
started with the triangle inequality 
\begin{eqnarray}
\label{finalp_1}
\lefteqn{\sum_{j=1}^n\Delta t \left\|\nabla (p^j-P_r^{\mathrm{p}} p^j)\right\|_{0}^2
\le C \sum_{j=1}^n\Delta t \left\|\nabla (p^j-p_h^j)\right\|_{0}^2} 
\nonumber\\
&&+
C \sum_{j=1}^n\Delta t \left\|\nabla (p_h^j-P_r^{\mathrm{p}} p_h^j)\right\|_{0}^2
+C \sum_{j=1}^n\Delta t \left\|\nabla (P_r^{\mathrm{p}} p_h^j-P_r^{\mathrm{p}} p^j)\right\|_{0}^2.
\end{eqnarray}
Recall that $P_r^{\mathrm{p}}$ is the $L^2(\Omega)$ projection. 
Using the triangle inequality, the interpolation 
estimate \eqref{interp}, and the inverse estimate \eqref{inv} yields for the first term 
\begin{eqnarray*}
\left\|\nabla (p^j-p_h^j)\right\|_{0}^2&\le & 2 \left\|\nabla (p^j-I_h(p^j))\right\|_{0}^2 +2 \left\|\nabla (I_h(p^j)-p_h^j)\right\|_{0}^2
\\
&\le &Ch^{2l-2} \left\| p^j\right\|_{l}^2  + C h^{-2} \left(\left\| I_h(p^j)-p^j\right\|_0^2+\left\|p^j-p_h^j\right\|_{0}^2\right)\\ 
&\le& Ch^{2l-2} \left\| p^j\right\|_{l}^2  + C h^{-2} \left\|p^j-p_h^j\right\|_{0}^2.
\end{eqnarray*}
After having taken the sum over the time instants, \eqref{eq:cota_pre} is applied to estimate 
the term with the $L^2(\Omega)$ norm of the FOM pressure error
\begin{equation*}%\label{eq:thm_4th_term_00}
 \sum_{j=1}^n\Delta t \left\|\nabla (p^j-p_h^j)\right\|_{0}^2 \le C
 \left( h^{2l-2} \sum_{j=1}^n\Delta t \left\| p^j\right\|_{l}^2 + C_{\rm press}^2(\bu,p,l+1) h^{2l-3}\right).
\end{equation*}
The second term is on the right-hand side of \eqref{finalp_1} is estimated with \eqref{eq:inv_Sp}, observing that the pressure snapshots
are contained in $\mathcal W$, and \eqref{eq:cota_pod_0_pre}, giving 
\[
\sum_{j=1}^n\Delta t \left\|\nabla (p_h^j-P_r^{\mathrm{p}} p_h^j)\right\|_{0}^2
\le \|S^{\mathrm{p}}\|_2 \sum_{j=1}^n\Delta t \left\|p_h^j-P_r^{\mathrm{p}} p_h^j\right\|_{0}^2
\le C \|S^{\mathrm{p}}\|_2 T \sum_{k=r+1}^{d_p}\gamma_k.
\]
For estimating the last term on the right-hand side of \eqref{finalp_1}, the inverse estimate 
\eqref{inv} and the stability of the projection are applied 
\begin{eqnarray*}
\sum_{j=1}^n\Delta t \left\|\nabla (P_r^{\mathrm{p}} p_h^j-P_r^{\mathrm{p}} p^j)\right\|_{0}^2
&\le&  Ch^{-2}\sum_{j=1}^n\Delta t  \left\|P_r^{\mathrm{p}} p_h^j-P_r^{\mathrm{p}} p^j\right\|_{0}^2\\
&\le& Ch^{-2} \sum_{j=1}^n\Delta t \left\|p_h^j-p^j\right\|_{0}^2.
\end{eqnarray*}
The last term is bounded in \eqref{eq:cota_pre}. Thus, the estimate of the fourth term on the 
right-hand side of \eqref{pre_defi} is 
\begin{eqnarray}\label{eq:thm_4th_term_2}
\lefteqn{\sum_{j=1}^n\Delta t \left\|\nabla (p^j-P_r^{\mathrm{p}} p^j)\right\|_{0}^2}\\
&\le& C\left(  h^{2l-2} \sum_{j=1}^n\Delta t \left\| p^j\right\|_{l}^2 + \|S^{\mathrm{p}}\|_2 T \sum_{k=r+1}^{d_p}\gamma_k
+ C_{\rm press}^2(\bu,p,l+1) h^{2l-3}\right).\nonumber
\end{eqnarray}

Estimate \eqref{cota_th_pre} is now obtained by 
 inserting \eqref{eq:eins_h1}, \eqref{eq:zwei}, \eqref{eq:dreih1}, 
and \eqref{eq:thm_4th_term_2} in \eqref{pre_defi}.

\medskip

Next, we consider the case $X=L^2(\Omega)^d$.

\emph{First term on the right-hand side of \eqref{pre_defi}.} Starting point of our estimate is \eqref{eq:stab_mot_pres_time_terms}. To bound the first term on the right-hand side, \eqref{pre_defi_1} is 
used. Arguing in the same way as for obtaining \cite[(4.17)-(4.18)]{wir_NS}, one derives for  
the second term
\[
\sum_{j=1}^n\Delta t \left\|\partial_t\bu_h^j-\frac{P_r^{\mathrm{v}} \bu_h^j-P_r^{\mathrm{v}} \bu_h^{j-1}}{\Delta t}\right\|_0^2
\le \frac{2T}{\tau^2}\sum_{k=r+1}^{d_v}\lambda_k +C (\Delta t)^2\int_0^T\|\partial_{tt}\bu_h\|_0^2 \ ds.
\]
We notice that $T(h/\tau)^2$ is part of $\hat C_{1,1}$ and $h^2$ is part of $\hat C_{2,1}$, so that 
\begin{eqnarray*}%\label{pre_defi_2_bis}
\lefteqn{\sum_{j=1}^n\Delta t \left\|\partial_t\bu_h^j-\frac{P_r^{\mathrm{v}} \bu_h^j-P_r^{\mathrm{v}} \bu_h^{j-1}}{\Delta t}\right\|_0^2
}\nonumber\\
&\le& \frac{C}{h^2}\left(\hat C_{1,1}\sum_{k=r+1}^{d_v}\lambda_k +\hat C_{2,1} (\Delta t)^2\int_0^T\|\partial_{tt}\bu_h\|_0^2 \ ds\right).
\end{eqnarray*}
Finally, the third term in \eqref{eq:stab_mot_pres_time_terms} is bounded by applying Lemma~\ref{le:long_proof}.
Collecting all bounds gives 
\begin{eqnarray}\label{eq:eins_l2}
\lefteqn{\sum_{j=1}^n\Delta t \left\|\partial_t\bu^j-\frac{\bu_r^j-\bu_r^{j-1}}{\Delta t}\right\|_{0}^2}\nonumber\\
& \le &  C_A   h^{2{l-2}} +\frac{C}{h^2}\left(
\hat C_0 \|\be_r^0\|_0^2 + (\hat C_{1,1}+\hat C_{1,2}\left\| S^{\mathrm{v}}\right\|)\sum_{k=r+1}^{d_v} \lambda_k\right.
\\
&&\left. +  (\hat C_{2,1}+\hat C_{2,2}\left\| S^{\mathrm{v}}\right\|)(\Delta t)^2 \int_0^T\left\| \partial_{tt}\bu_h\right\|_0^2\ ds
+\hat C_{3,1}(\Delta t)^2 \int_0^T\left\| \nabla\partial_{tt}\bu_h\right\|_0^2\ ds\right).\nonumber
\end{eqnarray}

\emph{Second term on the right-hand side of \eqref{pre_defi}.}
Starting from \eqref{eq:pres_est_second_term}, the first term on the right-hand side is again bounded by \eqref{eq:cota_grad_div} and the other term is estimated by \eqref{eq:pres_est_second_term0}.
Note that the first term on the right-hand side of  \eqref{eq:pres_est_second_term0} already appeared in \eqref{aux3} (main part of $X_3$) and it was bounded in \eqref{eq:need0}. The other term is 
bounded by applying  \eqref{eq:pre_error3_b_L2}.  Collecting
all bounds yields 
\begin{eqnarray*}
\lefteqn{\sum_{j=1}^n\Delta t \left\|(\bu^j\cdot\nabla) \bu^j-(\bu_r^j\cdot \nabla) \bu_r^j\right\|_0^2} \\
&\le& 2 C_4^2 T C^2(\bu,p,l+1) h^{2l-2} +  C C_4^2 T \left( \rho^2\|S^{\mathrm{v}}\|_2\sum_{k={r+1}}^{d_v}\lambda_k\right.\nonumber\\
&& \left. + T (\Delta t)^2 \left(\|S^{\mathrm{v}}\|_2 \int_0^T\|\partial_{tt}\bu_{h}\|_0^2\ ds
+ \int_0^T\|\nabla\partial_{tt}\bu_{h}\|_0^2\ ds\right)\right)\\
&& + C C_4^2 \frac{T}{h^2} \left(  C_0\|\be_r^0\|_0^2+(C_{1,1} +C_{1,2}\left\|S^{\mathrm{v}}\right\|_2)\sum_{k=r+1}^{d_v} \lambda_k \right.
\nonumber\\
&&\left. +(C_{2,1} + C_{2,2}\left\|S^{\mathrm{v}}\right\|_2)
(\Delta t)^2\int_0^T\|
\partial_{tt}\bu_h\|_0^2\ ds \right).
\end{eqnarray*}
We notice that, e.g.,  $C_4^2TC_{i,j}$, $i,j=1,2$ is part of $\hat C_{i,j}$, respectively, as well as 
$C_4^2 T^i \rho^{2(2-i)}h^2$, $i=1,2$ is part of
$\hat C_{i,2}$, $i=1,2$, respectively, so that 
\begin{eqnarray}\label{eq:zweil2}
\lefteqn{\sum_{j=1}^n\Delta t \left\|(\bu^j\cdot\nabla) \bu^j-(\bu_r^j\cdot \nabla) \bu_r^j\right\|_0^2
\le 2 C_4^2 T C^2(\bu,p,l+1) h^{2l-2}} \nonumber\\
&& + \frac{C}{h^2}\left(\hat C_0\|\be_r^0\|_0^2 \vphantom{\sum_{k=r+1}^{d_v} \lambda_k}
+
(\hat C_{1,1} +\hat C_{1,2}\left\|S^{\mathrm{v}}\right\|)\sum_{k=r+1}^{d_v} \lambda_k \right. \\
&& \left. +(\hat C_{2,1} + \hat C_{2,2}\left\|S^{\mathrm{v}}\right\|)
(\Delta t)^2\int_0^T\|
\partial_{tt}\bu_h\|_0^2\ ds + \hat C_{3,1} (\Delta t)^2 \int_0^T\left\|\nabla \partial_{tt}\bu_h\right\|_0^2\ ds \right).\nonumber
\end{eqnarray}

\emph{Third term on the right-hand side of \eqref{pre_defi}.} The estimate starts with \eqref{eq_y}.
Then, the term $Y_1$ is already bounded  in \eqref{eq:pre_error3_b_L2}. Apart of the factor in front of
the sum, the term $Y_2$ is the same as $X_3$ in \eqref{aux2} and this term is estimated in 
\eqref{eq:need0}. Finally,  $Y_3$, which does not contain a projection, is bounded as in \eqref{eq_y3}
under the assumption that $\nu\le h\|\bu\|_{L^\infty(L^\infty)}$. Collecting all estimates 
gives 
\begin{eqnarray*}
X_3 &\le& C\frac{\nu}{h^2} \left( C_0\|\be_r^0\|_0^2+(C_{1,1} +C_{1,2}\left\|S^{\mathrm{v}}\right\|_2)\sum_{k=r+1}^{d_v} \lambda_k\right.
\nonumber\\
&&\left.+(C_{2,1} + C_{2,2}\left\|S^{\mathrm{v}}\right\|_2)
(\Delta t)^2\int_0^T\|
\partial_{tt}\bu_h\|_0^2\ ds \right)\\
&& + C\frac{\nu^2}{h^2} T \left(\rho^2\|S^{\mathrm{v}}\|_2\sum_{k={r+1}}^{d_v}\lambda_k\right.\nonumber\\
&& \left. + T (\Delta t)^2 \left(\|S^{\mathrm{v}}\|_2 \int_0^T\|\partial_{tt}\bu_{h}\|_0^2\ ds
+ \int_0^T\|\nabla\partial_{tt}\bu_{h}\|_0^2\ ds\right)\right)\\
&&+  C_1^2(\bu,p,l+1) h^{2l-2}.
\end{eqnarray*}
Notice that $\nu C_{0}$ and $\nu C_{i,j}$ are part of $\hat C_0$ and $\hat  C_{i,j}$, for $1\le i,j\le 2$, respectively,
that
$\nu^2 T\rho^2$ is part of~$\hat C_{1,2}$ and~$\nu^2 T^2$ part of $\hat C_{2,2}$ and $\hat C_{3,1}$.
Thus,
\begin{eqnarray} \label{eq:thm_l2_third}
\lefteqn{X_3 \le  C_1^2(\bu,p,l+1) h^{2l-2} }\nonumber\\
&& + \frac{C}{h^2}\left(\hat C_0\|\be_r^0\|_0^2 + (\hat C_{1,1} +\hat C_{1,2}\left\|S^{\mathrm{v}}\right\|_2)\sum_{k=r+1}^{d_v} \lambda_k \right. \\
&& \left.+ \hat C_{2,2} (\Delta t)^2 \left\|S^{\mathrm{v}}\right\|_2 \int_0^T\|\partial_{tt}\bu_{h}\|_0^2\ ds
+ \hat C_{3,1} (\Delta t)^2 \int_0^T\|\nabla\partial_{tt}\bu_{h}\|_0^2\ ds\right).\nonumber
\end{eqnarray}

\emph{Fourth term on the right-hand side of \eqref{pre_defi}.} 
This term does not depend on the projection that is used for the velocity and thus estimate \eqref{eq:thm_4th_term_2} is applicable. 

Now, the estimate \eqref{cota_th_pre_l2} is obtained by 
inserting \eqref{eq:eins_l2} \eqref{eq:zweil2} \eqref{eq:thm_l2_third}, and \eqref{eq:thm_4th_term_2} in \eqref{pre_defi}.
\end{proof}

\section{Numerical Studies} \label{sec:numres}

As usual, the numerical studies shall support the numerical analysis. Since the new analytic results are only
with respect to the computation of a POD-ROM pressure, the focus of the presented numerical results will be
on the pressure. Important aspects are the order of convergence, the robustness for small viscosity
coefficients and a comparison between the supremizer enrichment (SE-ROM) approach \eqref{eq:pres}
and the stabilization-motivated (SM-ROM) method \eqref{pre_moti}.

For assessing the pressure, a discrete-in-time approximation of the error in $L^2(0,T;L^2(\Omega))$ is used
\[
\Vert p - p_r \Vert_{l^2(L^2)}^2 :=  \sum_{j=1}^n\Delta t \|p^j - p_r^j\|_{0}^2,
\]
which is the norm on the left-hand side of \eqref{eq:cota_pre_supr},
as well as a locally scaled error of the form
\[
\vert p - p_r \vert_{l^2(H^1_K)}^2
  :=  \sum_{j=1}^n\Delta t \sum_{K\in \mathcal T_h}\tau_K\|\nabla p^j - \nabla p_r^j\|_{0,K}^2.
\]
Note that this error is connected to the numerical analysis since it is on the left-hand side of \eqref{cota_th_pre_l2}. For demonstrating exemplarily the impact of the grad-div stabilization on the
POD-ROM velocity, the following errors will be considered
\begin{eqnarray*}
  \Vert \bu - \bu_r \Vert_{l^2(L^2)}^2
  &:= & \sum_{j=1}^n \Delta t \|\bu^j - \bu_r^j\|_0^2, \\
%  \vert \bu_r - \bu \vert_{\mathrm L^2(\mathrm H^1_0)}^2
%  &\approx \sum_{j=1}^n \Delta t \|\nabla \bu_r^j - \nabla \bu^j\|_0^2, \\
  \Vert \div \bu_r\Vert_{l^2(L^2)}^2
  &:=& \sum_{j=1}^n \Delta t \|\div \bu_r^j\|_0^2.
\end{eqnarray*}

All simulations were performed with the code \emph{ParMooN} \cite{WB_J17}.

\subsection{No-Flow Problem with Complicated Pressure} \label{ssec:noflow}

This example possesses a prescribed solution where the velocity is very simple and the pressure
quite complicated in the sense that it cannot be separated with respect to the variables.
Let $\Omega = (0,1)^2$ be the unit square and $T = 1$, and prescribe the following solution of \eqref{NS}:
\begin{align*}
  \bu(t, \bx) &= && \boldsymbol{0} && \text{ in } [0, T] \times \Omega, \\
  p(t, \bx)   &=&& \tanh\Big(25(x_1 - v(t))\Big) - \tanh\Big(25(1 - x_1 + v(t))\Big) && \\
                          &&&+ \tanh\Big(17(x_2 - w(t))\Big) - \tanh\Big(17(1 - x_2 + w(t))\Big) && \text{ in } [0, T] \times \Omega,
\end{align*}
with oscillatory displacements
\[
    v(t) = \frac 1 2 \big(1 + \sin(4 \pi t)\big), \quad w(t) = \frac 1 2 \big(1 + \sin(6 \pi t)\big).
\]
See Figure~\ref{fig:noflow_pressure} for an illustration of $p$.
This pressure field is smooth and periodic in time with period $T = 1/2$ and it holds that
$\int_\Omega  p \  d\bx = 0$ for all $t$.
The relatively sharp and moving steps induced by the hyperbolic tangent ensure that it cannot be well
approximated by a small number of spatial pressure modes at all times.
The pair $(\bu, p)$ satisfies \eqref{NS}
with homogeneous Dirichlet boundary conditions and right-hand side
$\bff = \nabla p$.

\begin{figure}[t!]
  \centering
  \includegraphics[width = 0.45 \textwidth]{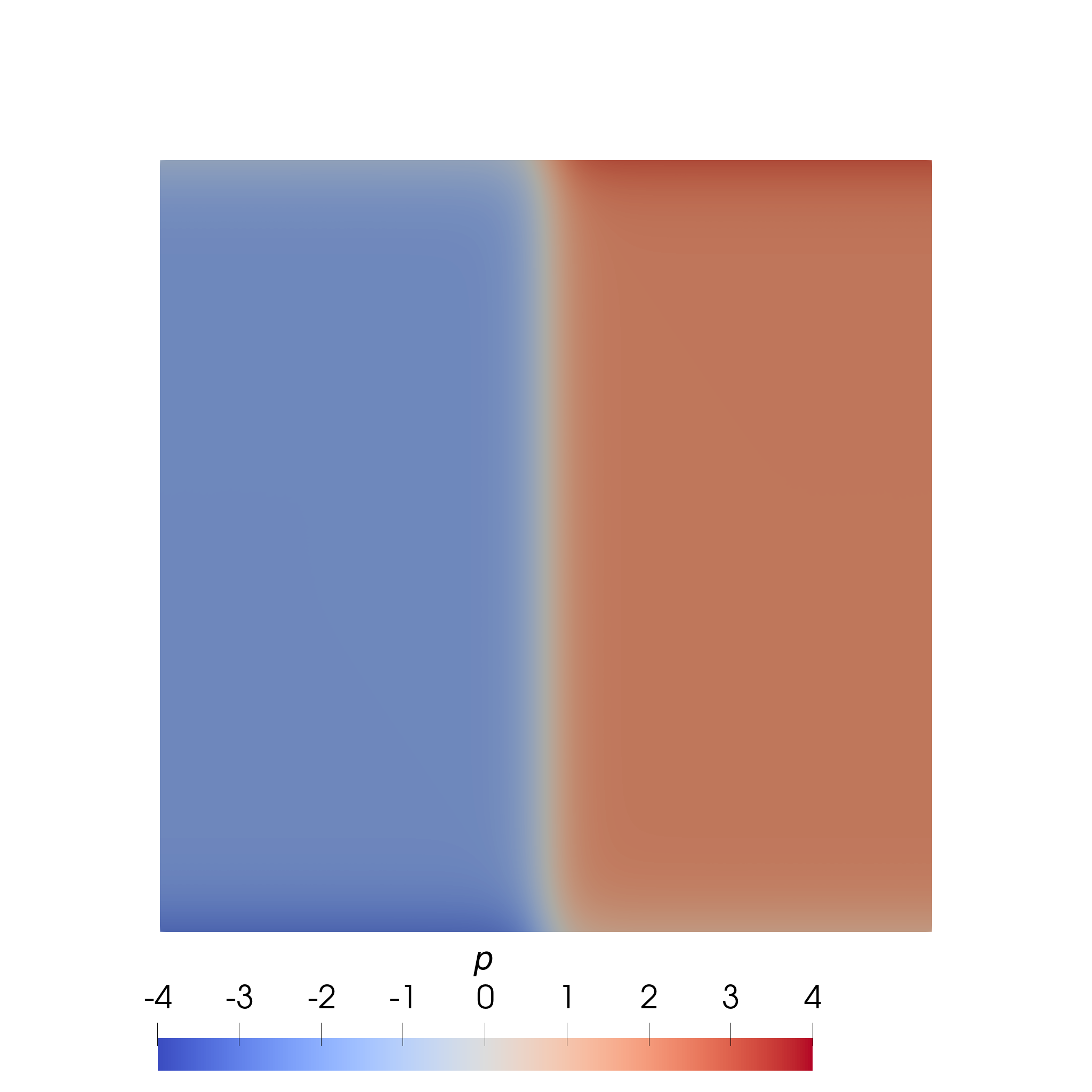}
  \includegraphics[width = 0.45 \textwidth]{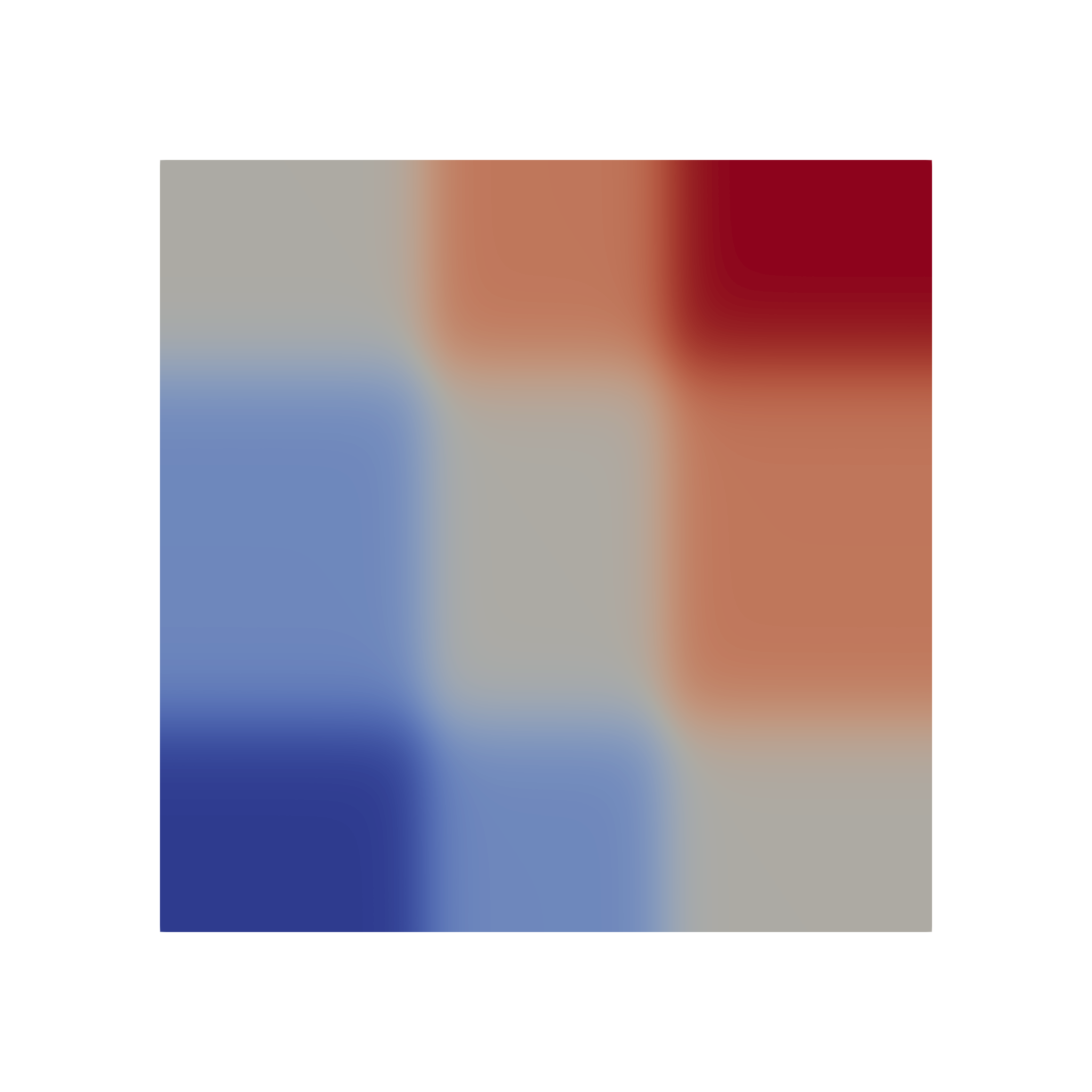}
  \caption{Pressure field $p$  for the no-flow problem at $t = 0$ (left) and $t = 0.725$ (right).}
  \label{fig:noflow_pressure}
\end{figure}

The impact of varying the grid, the viscosity coefficient, and the number of
velocity and pressure modes on the POD-ROM errors will be studied.

For the computations we used Taylor--Hood pairs of spaces with
$l = 2$, i.e., continuous piecewise quadratic velocities and continuous piecewise
linear pressures on the irregular triangular mesh shown in Figure~\ref{fig:noflow_mesh}
and its first three uniform refinements. Mesh statistics are provided in
Table~\ref{tab:noflow_mesh}.

\begin{figure}[t!]
  \centering
  \includegraphics[width = 0.5 \textwidth]{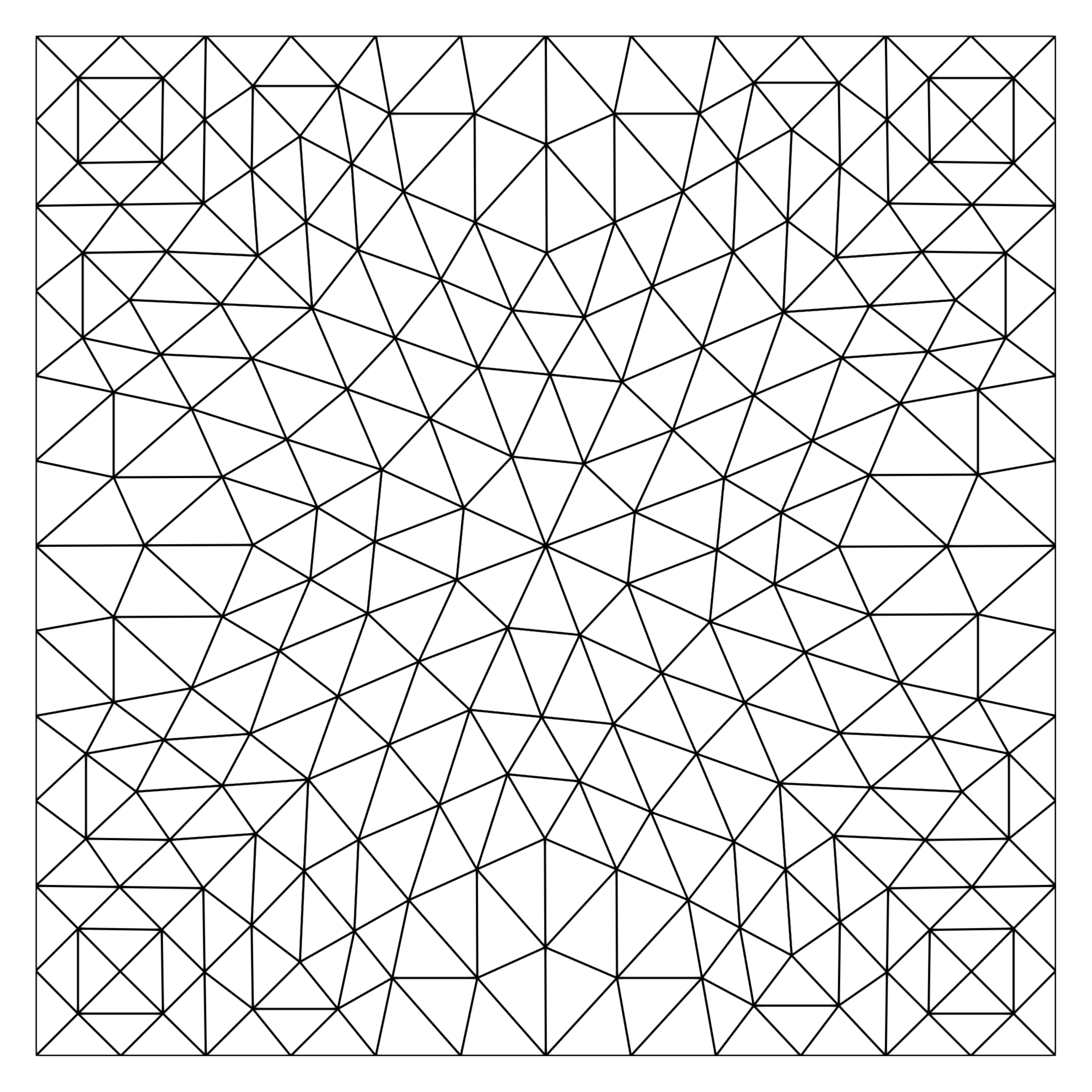}
  \caption{The computational mesh for the no-flow problem, $L = 1$.}
  \label{fig:noflow_mesh}
\end{figure}

\begin{table}[t!]
  \begin{center}
    \caption{Mesh statistics for the no-flow problem:
      refinement level $L$,
      number of triangular cells $N_K$,
      largest cell diameter $h$,
      velocity and pressure space dimensions.
    }
    \label{tab:noflow_mesh}
    \begin{tabular}{l|c|c|c|c}
      $L$ & $N_K$ &   $h$ & $\mathrm{dim}(\boldsymbol X^2_h)$ & $\mathrm{dim}(Q^1_h)$ \\ [0.2em]
      \hline
        1 &   416 & 0.106 &                              1762 &                               233 \\ [0.2em]
        2 &  1664 & 0.053 &                              6850 &                               881 \\ [0.2em]
        3 &  6656 & 0.027 &                             27010 &                              3425 \\ [0.2em]
        4 & 26624 & 0.013 &                            107266 &                             13505
    \end{tabular}
  \end{center}
\end{table}

Regardless of $L$, every simulation used the BDF2 scheme as time integrator with the same
relatively small time step $\Delta t = 0.005$. The grid comparison in
Section~\ref{sssec:noflow_grid} will show that this step size was small enough to make
the temporal discretization error negligible. Except in the case of the viscosity
comparison presented in Section~\ref{sssec:noflow_visc}, the viscosity was
$\nu = 0.01$. We will also compare results with grad-div stabilization using
$\mu = 0.1$, which is a common order of magnitude for this parameter and second order
Taylor--Hood elements, and without grad-div stabilization, i.e., $\mu = 0$.

For each ROM computation, the reduced order velocity space $\bU^r$ was
built by applying POD  with the $L^2(\Omega)^d$ inner product to the full order computation snapshots' temporal
derivatives $\tau \partial_t \bu_h^1, \ldots, \tau \partial_t \bu_h^{201}$
with $\tau = 1/6$ and the scaled snapshot average $\sqrt{201} \bar \bu_h$.
The choice of $\tau$ corresponds roughly to a characteristic time scale of
$p$. The reduced order velocity was computed using
\eqref{eq:pod_method2} with $\bu_r \in \bar \bu_h + \bU_r$.
Notice that as $\bu \equiv \boldsymbol{0}$, the computed velocities
$\bu_h$ and $\bu_r$ are in all cases pure noise that should be expected to
decrease with mesh refinement. It is well known that the divergence of the discrete velocity
fields computed with pairs of Taylor--Hood finite elements might be quite large, e.g., see
\cite[Example~4.31]{John}.

Pressure modes were computed by applying POD to the pressure fluctuation
snapshots $p_h^1 - \bar p_h, \ldots, p_h^{201} - \bar p_h$ to give the
reduced order pressure space $\mathcal W^r$. The reduced order pressure was
computed for $p_r \in \bar p_h + \mathcal W^r$ using both the SE-ROM method
\eqref{eq:pres} and the SM-ROM method \eqref{pre_moti} with $\tau_K = h_K^2$.

In each case, POD modes corresponding to eigenvalues $\lambda < 10^{-10}$ were
discarded as noisy. Table~\ref{tab:noflow_eigenvalues} shows an overview of the
results for $\nu = 0.01$. The lower noise level of the velocity results with
grad-div stabilization is clearly visible in the smaller number of
velocity modes corresponding to large indices, as well as in the eigenvalue plots in
Figure~\ref{fig:noflow_eigenvalues}.

\begin{table}[t]
  \begin{center}
    \caption{POD statistics for the no-flow problem, $\nu = 0.01$:
      refinement level $L$,
      number of velocity modes $d_{v, \mathrm{gd}}$, $d_{v, \mathrm{ngd}}$
        with and without grad-div stabilization,
      number of pressure modes $d_{p, \mathrm{gd}}$, $d_{p, \mathrm{ngd}}$
        with and without grad-div stabilization.
    }
    \label{tab:noflow_eigenvalues}
    \begin{tabular}{l||c|c||c|c}
      $L$ & $d_{v, \mathrm{gd}}$ & $d_{p, \mathrm{gd}}$ & $d_{v, \mathrm{ngd}}$ & $d_{p, \mathrm{ngd}}$ \\ [0.2em]
      \hline
        1 &                  109 &                   57 &                   111 &                    56 \\ [0.2em]
        2 &                   92 &                   58 &                    96 &                    59 \\ [0.2em]
        3 &                   61 &                   37 &                    75 &                    48 \\ [0.2em]
        4 &                   40 &                   38 &                    64 &                    38
    \end{tabular}
  \end{center}
\end{table}

\begin{figure}[t]
  \centering
  \includegraphics[width = 0.45 \textwidth]{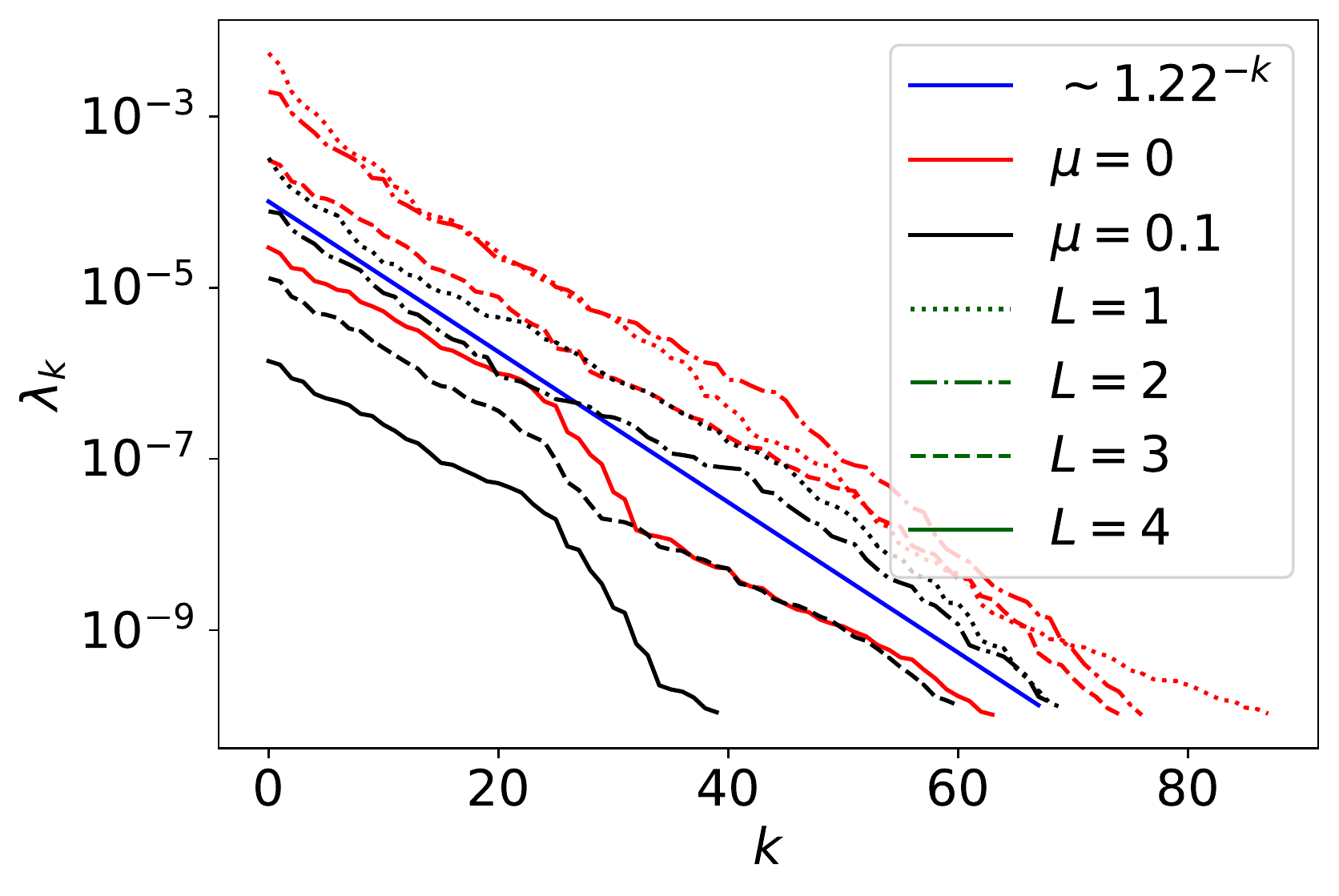}
  \includegraphics[width = 0.45 \textwidth]{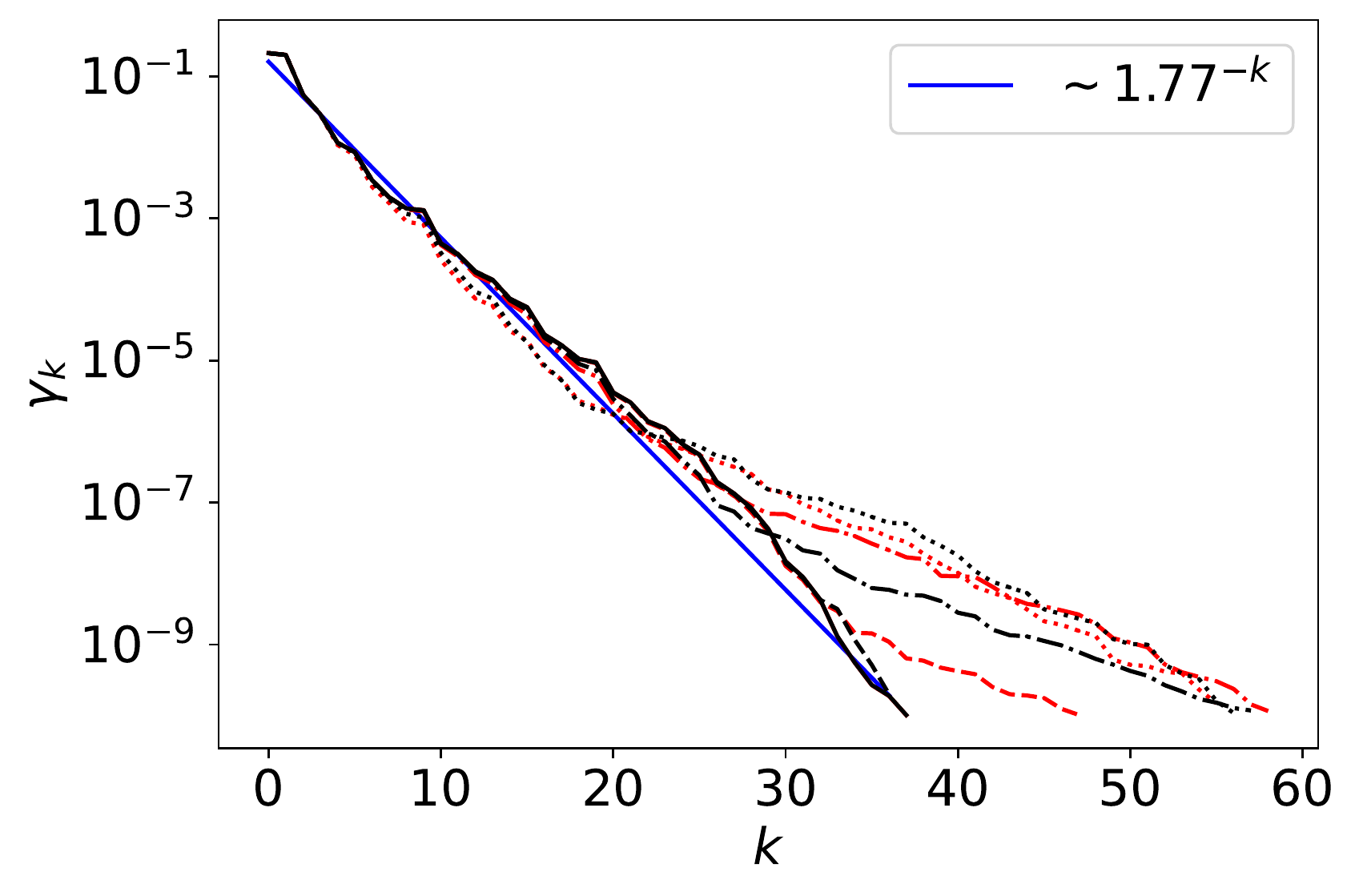}
  \caption{Velocity and pressure eigenvalues for the no-flow problem, varying grid size.}
  \label{fig:noflow_eigenvalues}
\end{figure}

Notice that, in each comparison below, the SE-ROM and SM-ROM velocities are the
same, only the pressure is computed using a different scheme.

\subsubsection{Convergence with Respect to Grid Refinement} \label{sssec:noflow_grid}

Using as many velocity and pressure modes as were available, we investigated the
reduced order models' convergence as the underlying mesh is refined.

Figure~\ref{fig:noflow_grid_u} shows velocity and divergence errors along with
expected and empirically computed orders of convergence.
These orders, three and two, respectively, meet the expectations. They are
even slightly better on this range of mesh sizes. A crucial observation is that using grad-div stabilization
improves both errors notably, e.g., with respect to the divergence by about an order of
magnitude.

\begin{figure}[t]
  \centering
  \includegraphics[width = 0.45 \textwidth]{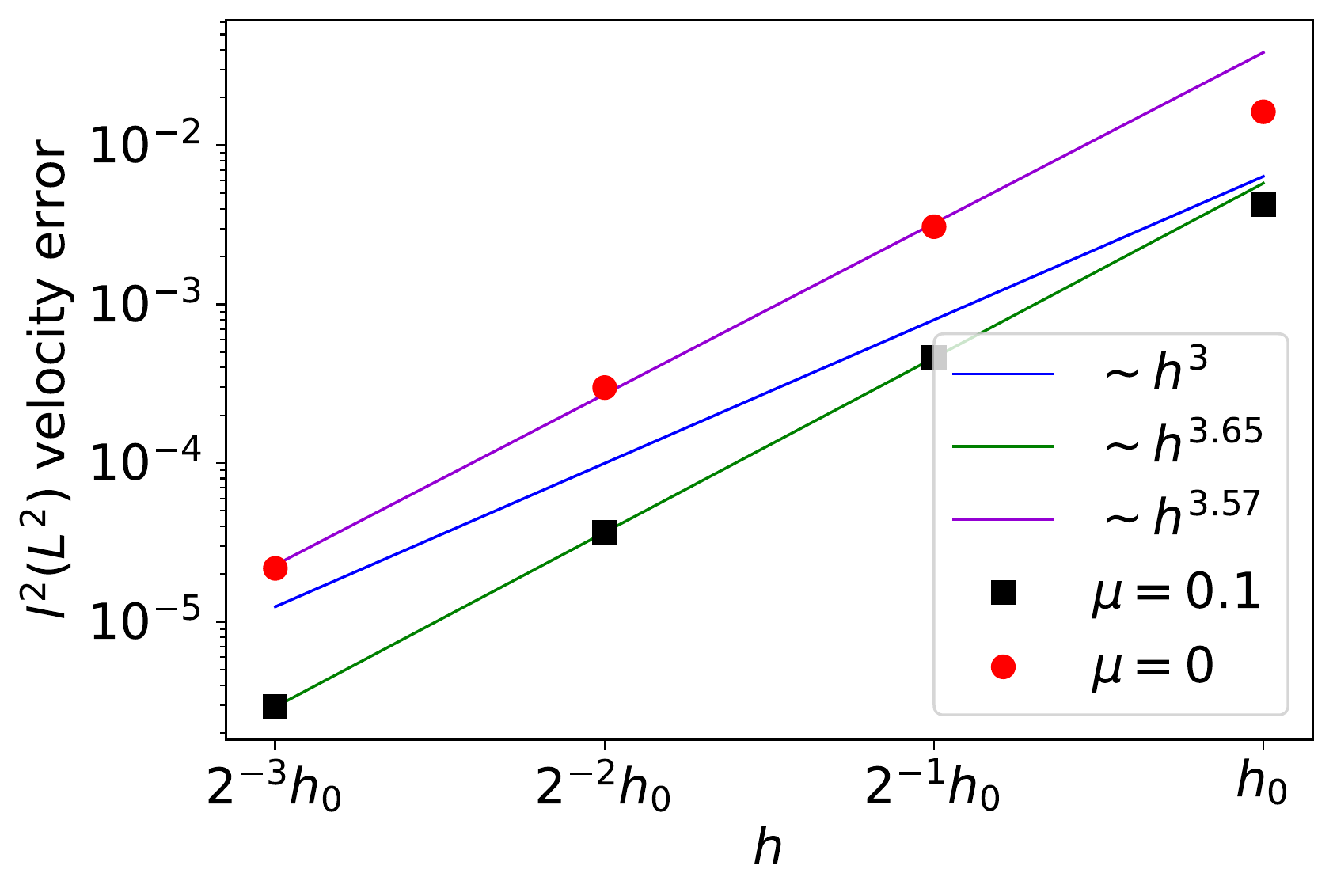}
  \includegraphics[width = 0.45 \textwidth]{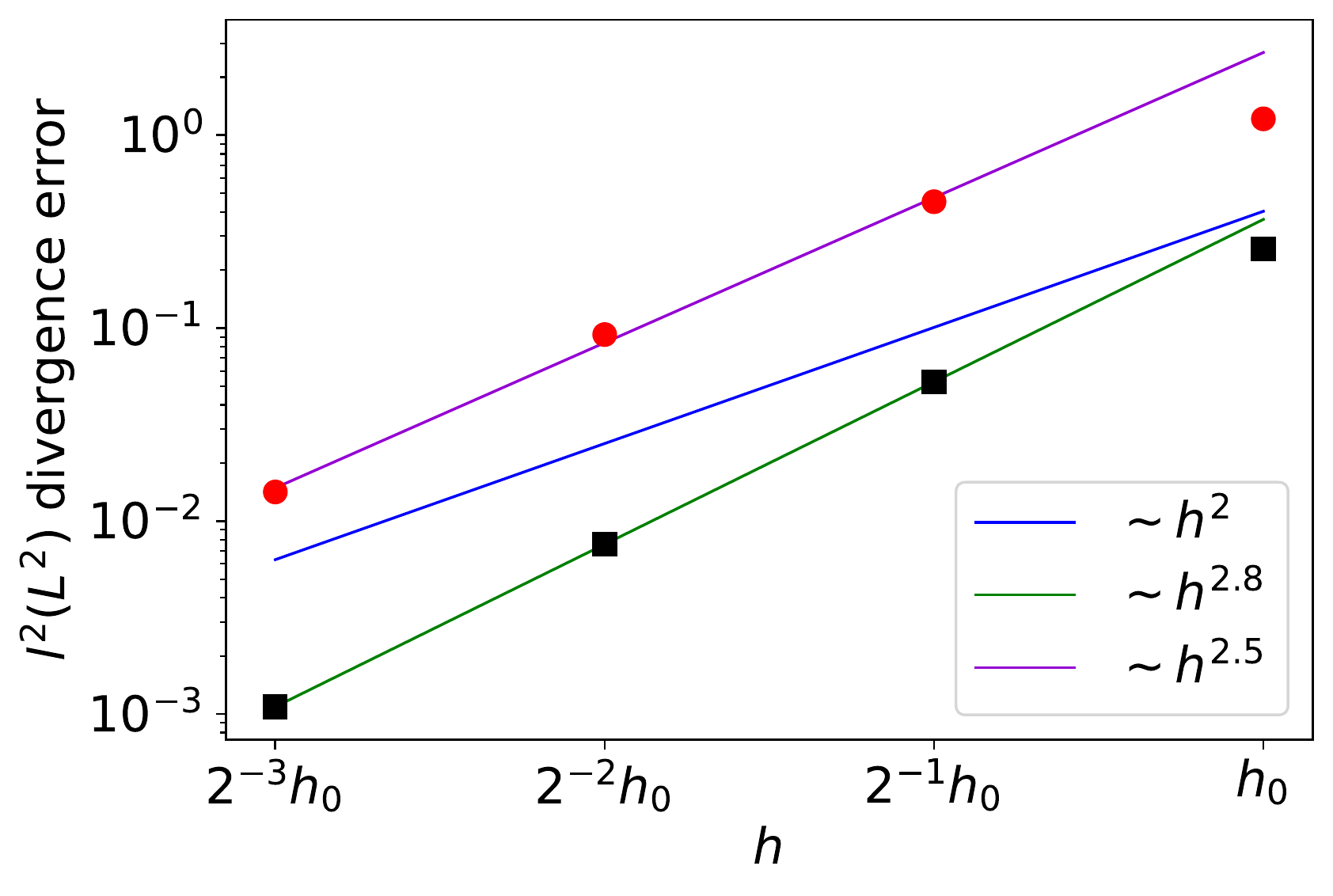}
  \caption{Velocity and divergence errors for the no-flow problem, varying grid size.}
  \label{fig:noflow_grid_u}
\end{figure}

Pressure errors, along with expected
and estimated orders of convergence, are presented in Figure~\ref{fig:noflow_grid_p}.
Since $X=L^2(\Omega)^d$, the error bound \eqref{cota_th_pre_l2} is applicable, which leads to the order
of convergence  $h^{l - 1/2} = h^{3/2}$ for the norm on the left-hand side of  \eqref{cota_th_pre_l2}.
Note that this norm is roughly the gradient of the pressure error multiplied by the mesh width.
Again, the error reductions on the considered meshes are somewhat larger than the predicted ones.
Note that the order $3/2$ is caused only from estimate \eqref{eq:cota_pre} for the FOM pressure error and that
already in the numerical studies of \cite{NS_grad_div} a higher order of convergence for this error was observed.
All other spatial error terms in \eqref{cota_th_pre_l2} are of second order.
The difference between the
velocities with and without grad-div stabilization is too small to greatly
influence the pressure errors. One can observe that the SE-ROM method gives
noticeably smaller pressure errors in the $l^2(L^2)$ norm compared with the SM-ROM method,
but not in the $l^2(H^1_K)$ seminorm. The
$l^2(H^1)$ error (not shown) also does not show a large difference.

\begin{figure}[t]
  \centering
  \includegraphics[width = 0.45 \textwidth]{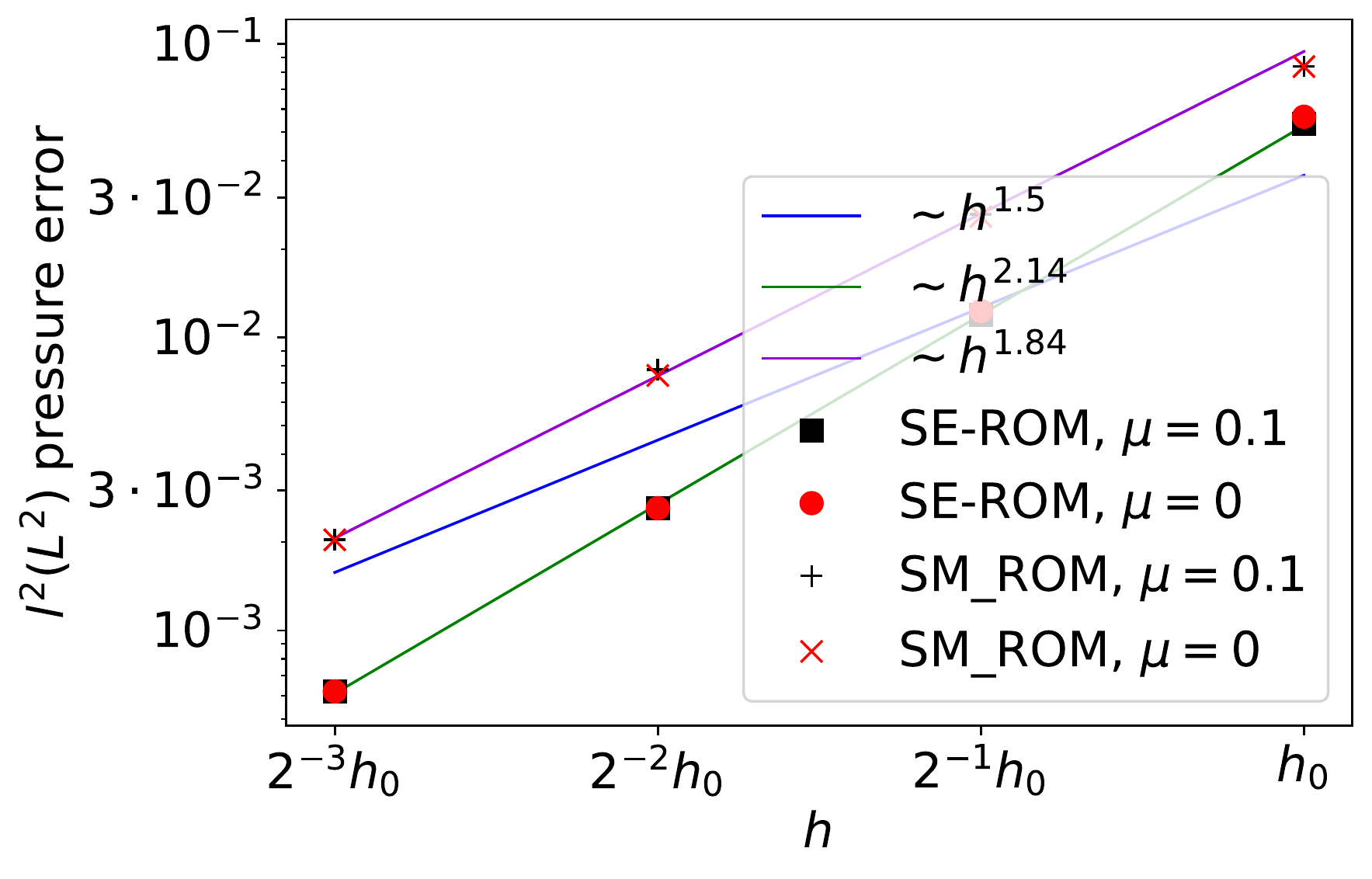}
  \includegraphics[width = 0.45 \textwidth]{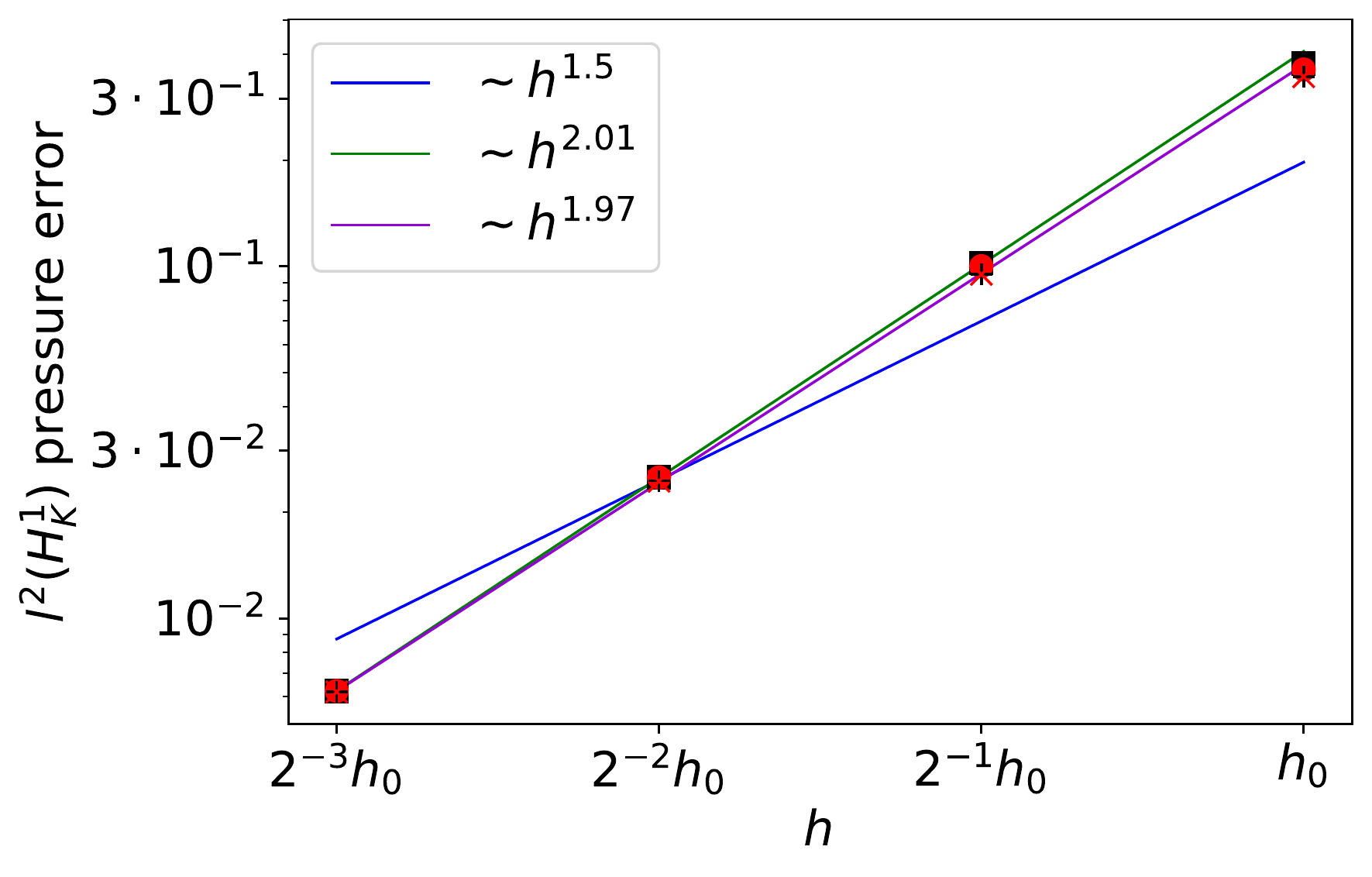}
  \caption{Pressure errors for the no-flow problem, varying grid size.}
  \label{fig:noflow_grid_p}
\end{figure}

\subsubsection{Convergence with Respect to Increasing Ranks} \label{sssec:noflow_rank}

At each refinement level of the mesh, we compared the error of the reduced order
models using $r \in \{2, 3, 4, 6, 8, 12, 16, 24, 32\}$ velocity and pressure
modes.

For the sake of brevity and for concentrating on the topic of this paper, only results for the
pressure will be presented\footnote
{
  As the prescribed velocity is zero, the computed velocity is only a discretization error and
  numerical noise; the ROM results are therefore also only noise and vary very little with the
  velocity ROM's rank.
}.
Notice that in each picture of Figure~\ref{fig:noflow_rank_p}
the horizontal axis is decreasing to the right and is marked with the fraction of remaining % velocity or
pressure eigenvalues
\[
%\mathcal R_u = \sum_{k=r+1}^{d_v} \lambda_k \Bigg/ \sum_{k=1}^{d_v} \lambda_k,\quad
\mathcal R_p = \sum_{k=r+1}^{d_p} \gamma_k \Bigg/ \sum_{k=1}^{d_p} \gamma_k.
\]
\iffalse
\paragraph{Velocity}
As expected, the velocity error behaves somewhat pathologically.
Figure~\ref{fig:noflow_rank_u} exhibits no strong dependence on number of modes
used, as these are entirely noise; the velocity error's dependence on $h$ and
$\mu$ dominates. Indeed, using more modes results in a slight increase in noise,
since the average velocity $\bar \bu_h$ benefits from smoothing effects compared
to the full-order solution at any one instant.

\begin{figure}[h]
  \centering
  \includegraphics[width = 0.45 \textwidth]{fig/no_flow/stabilization_motivated/rank_u_L2_L2_Ru.pdf}
  \includegraphics[width = 0.45 \textwidth]{fig/no_flow/stabilization_motivated/rank_u_div_L2_L2_Ru_legend.pdf}
  \caption{Velocity and divergence errors for the no-flow problem,
    varying number of velocity and pressure modes.}
  \label{fig:noflow_rank_u}
\end{figure}
\fi
It becomes clear that the moving steps in
$p$ do require a substantial number of pressure modes to
approximate them well at all times. For decreasing $\mathcal R_p$, the
graphs level off as the FOM errors eventually dominate the rank errors. There are almost no differences
of the results between using the grad-div stabilization for the velocity simulations or not. Likewise,
both studied methods for computing a POD-ROM pressure behave almost identically.

\begin{figure}[t!]
  \centering
  \includegraphics[width = 0.45 \textwidth]{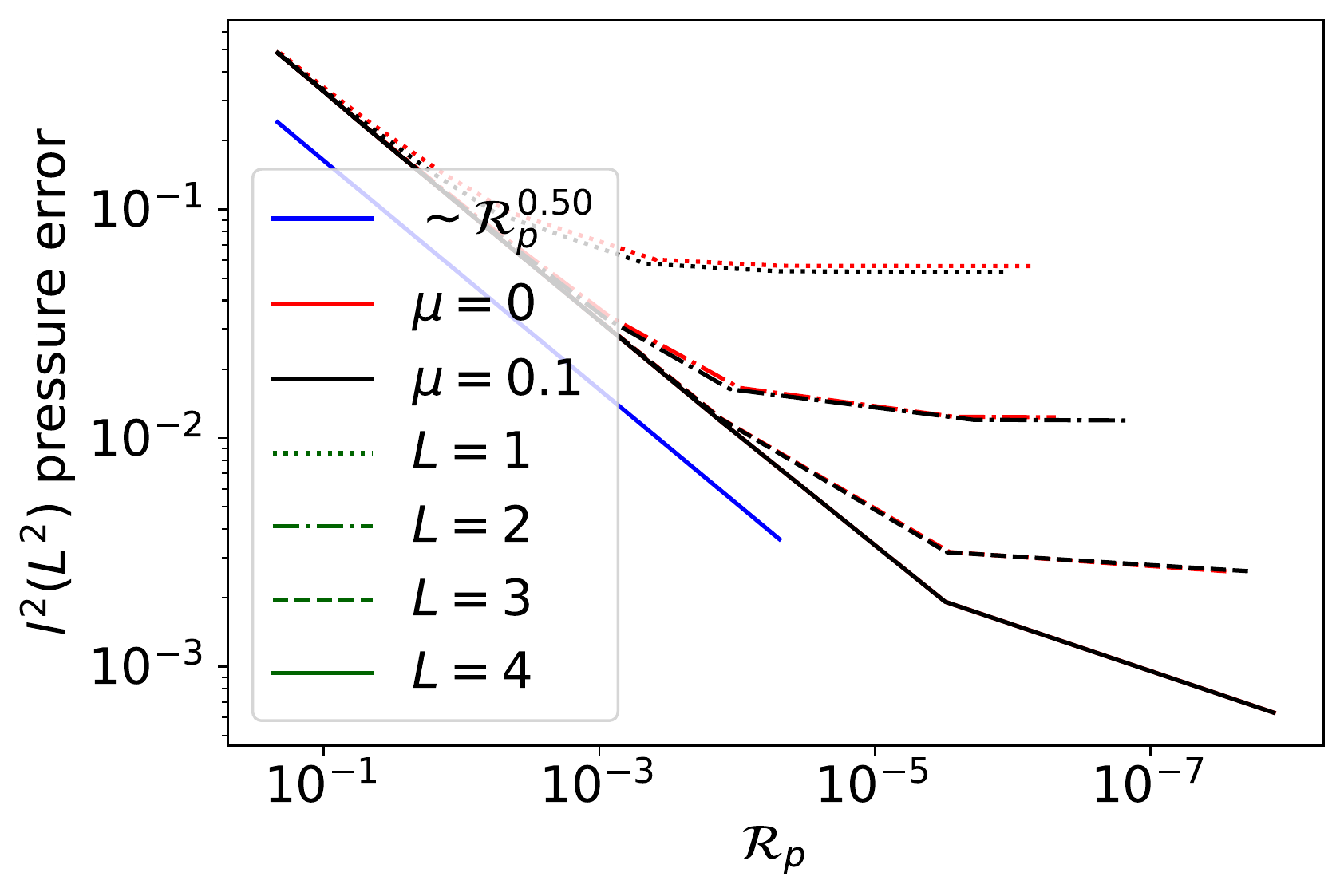}
  \includegraphics[width = 0.45 \textwidth]{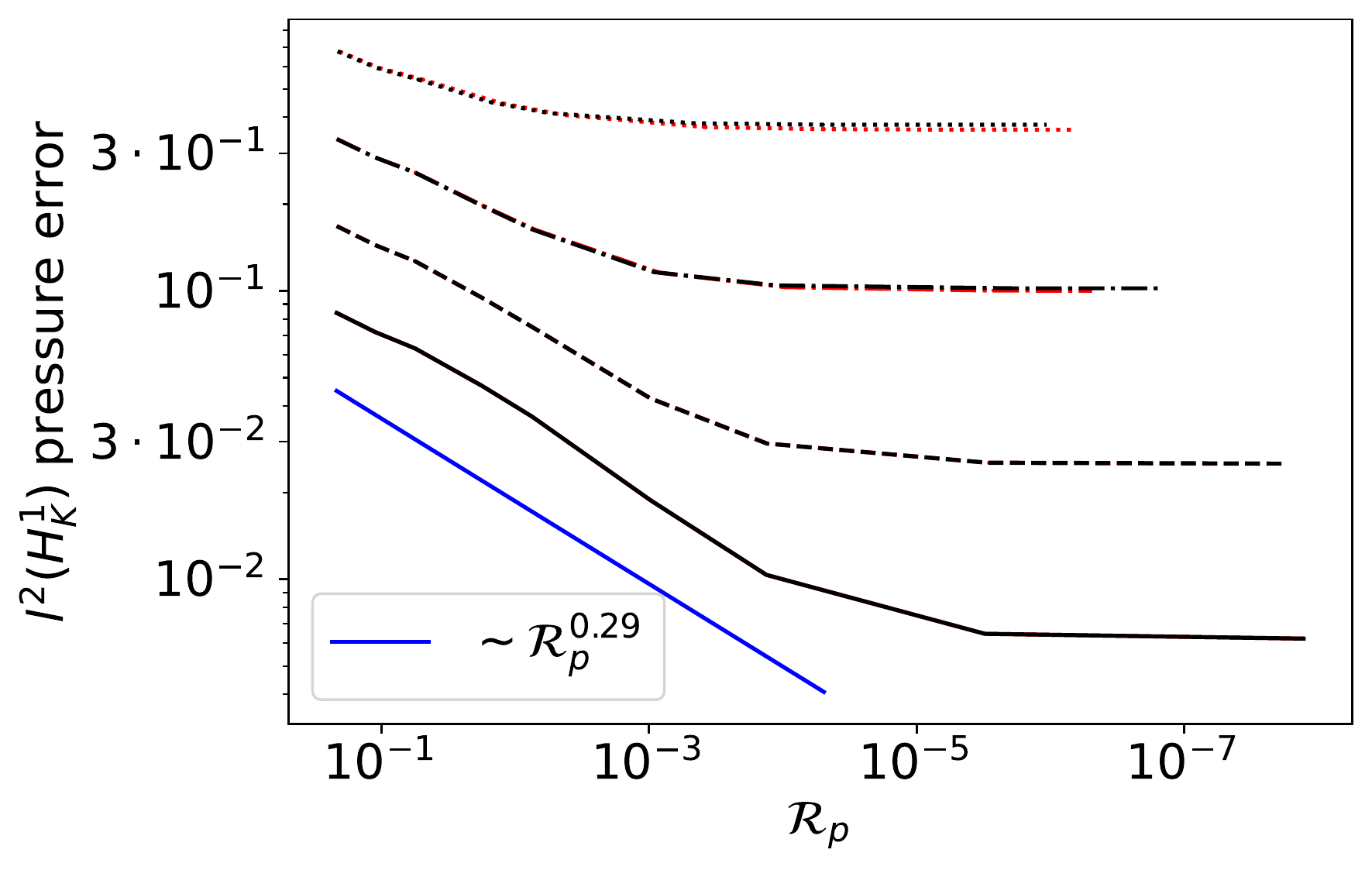}
  \includegraphics[width = 0.45 \textwidth]{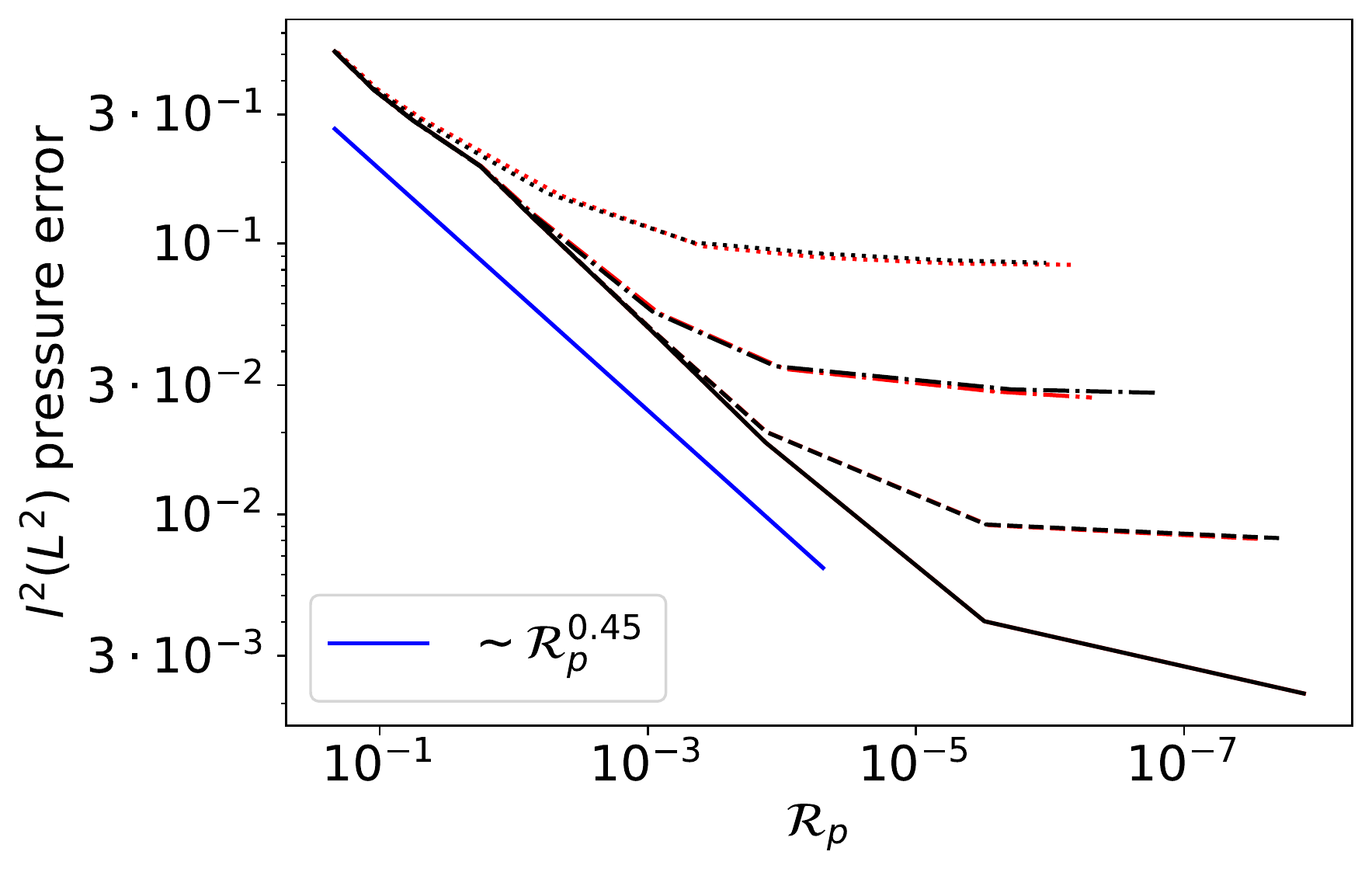}
  \includegraphics[width = 0.45 \textwidth]{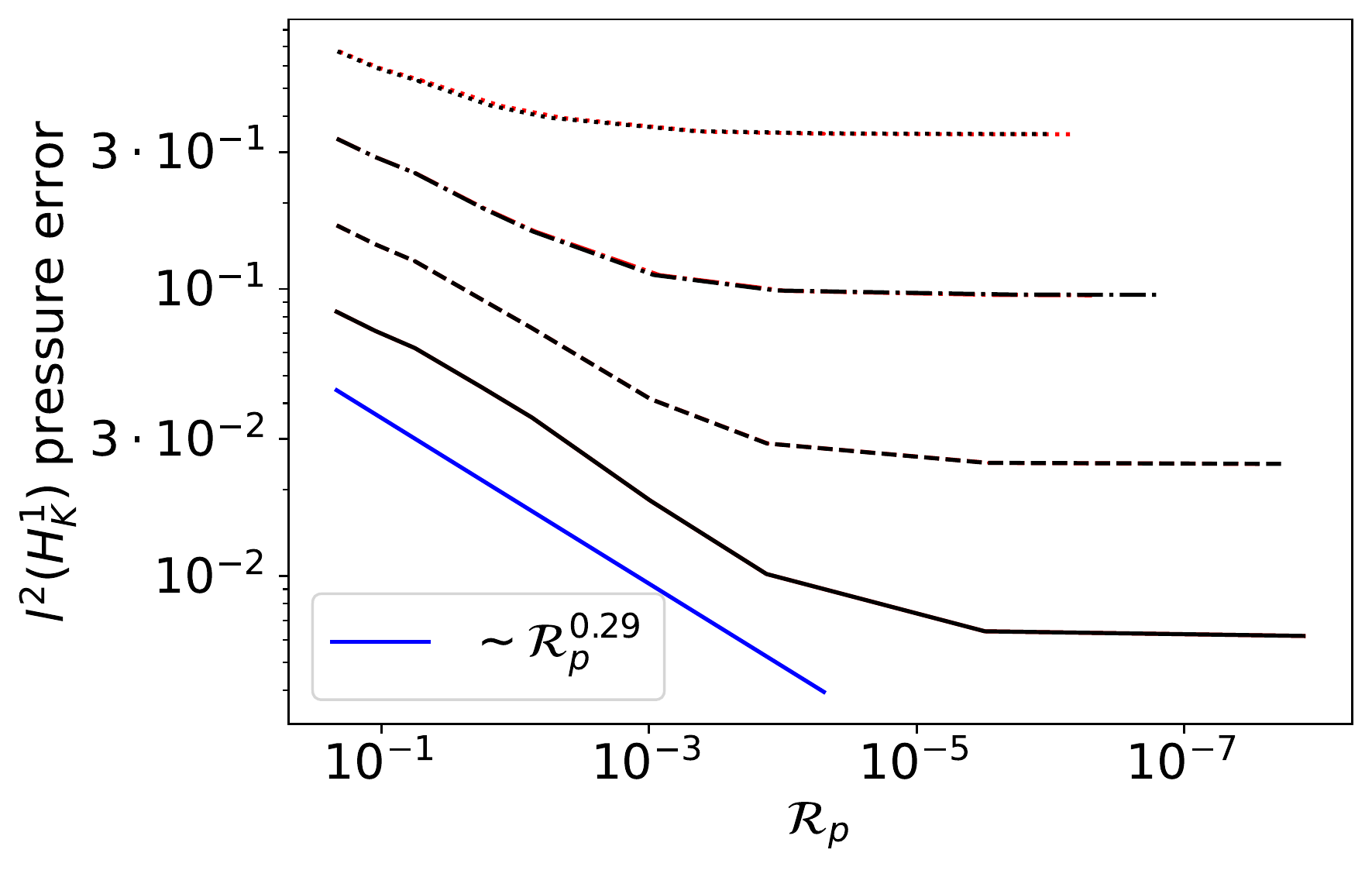}
  \caption{Pressure errors for the no-flow problem,
    varying number of velocity and pressure modes.
    Upper row: SE-ROM. Lower row: SM-ROM.}
  \label{fig:noflow_rank_p}
\end{figure}

\subsubsection{Impact of the Size of the Viscosity Coefficient} \label{sssec:noflow_visc}

Finally, a study on the impact of varying the viscosity coefficient on the errors of
full-rank reduced order simulations will be presented. In this study, $L = 3$ was fixed and the
situations that $\nu \in \{ 1, 0.1, 0.01, 10^{-4}, 10^{-6}, 10^{-8} \}$ were considered.

Figure~\ref{fig:noflow_visc_u} shows the substantial impact of the viscosity on
the presence of velocity noise (recall that $\bu = \boldsymbol 0$), particularly without grad-div stabilization.
In particular, for very low viscosity coefficients the divergence error with $\mu = 0.1$ is
roughly three orders of magnitude smaller than with $\mu = 0$ and the
$l^2(L^2)$ error of the velocity differs by two orders of magnitude. The errors for the simulations with
grad-div stabilization are clearly robust with respect to the size of the viscosity coefficient and they
are of small magnitude.

\begin{figure}[t!]
  \centering
  \includegraphics[width = 0.45 \textwidth]{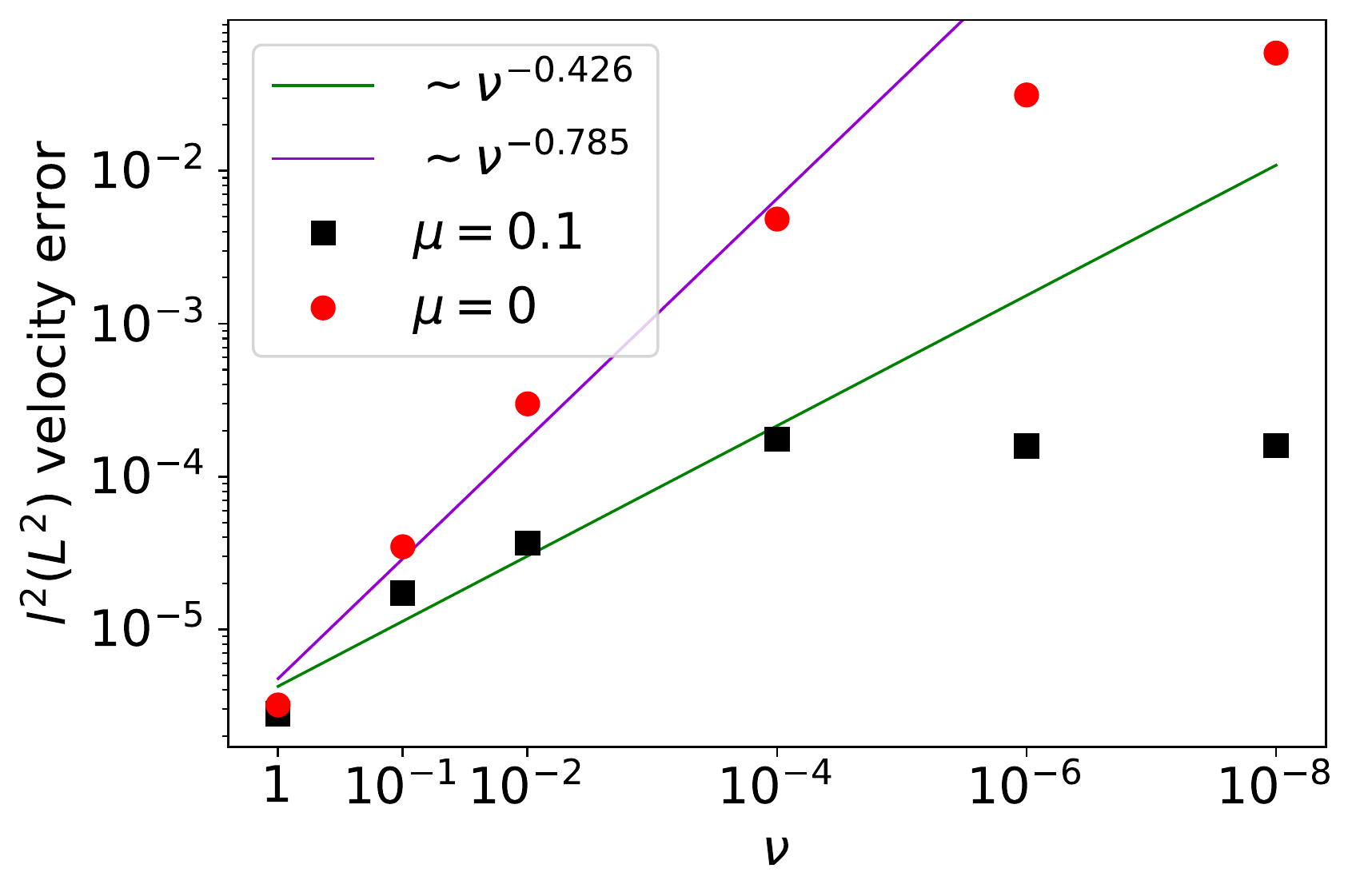}
  \includegraphics[width = 0.45 \textwidth]{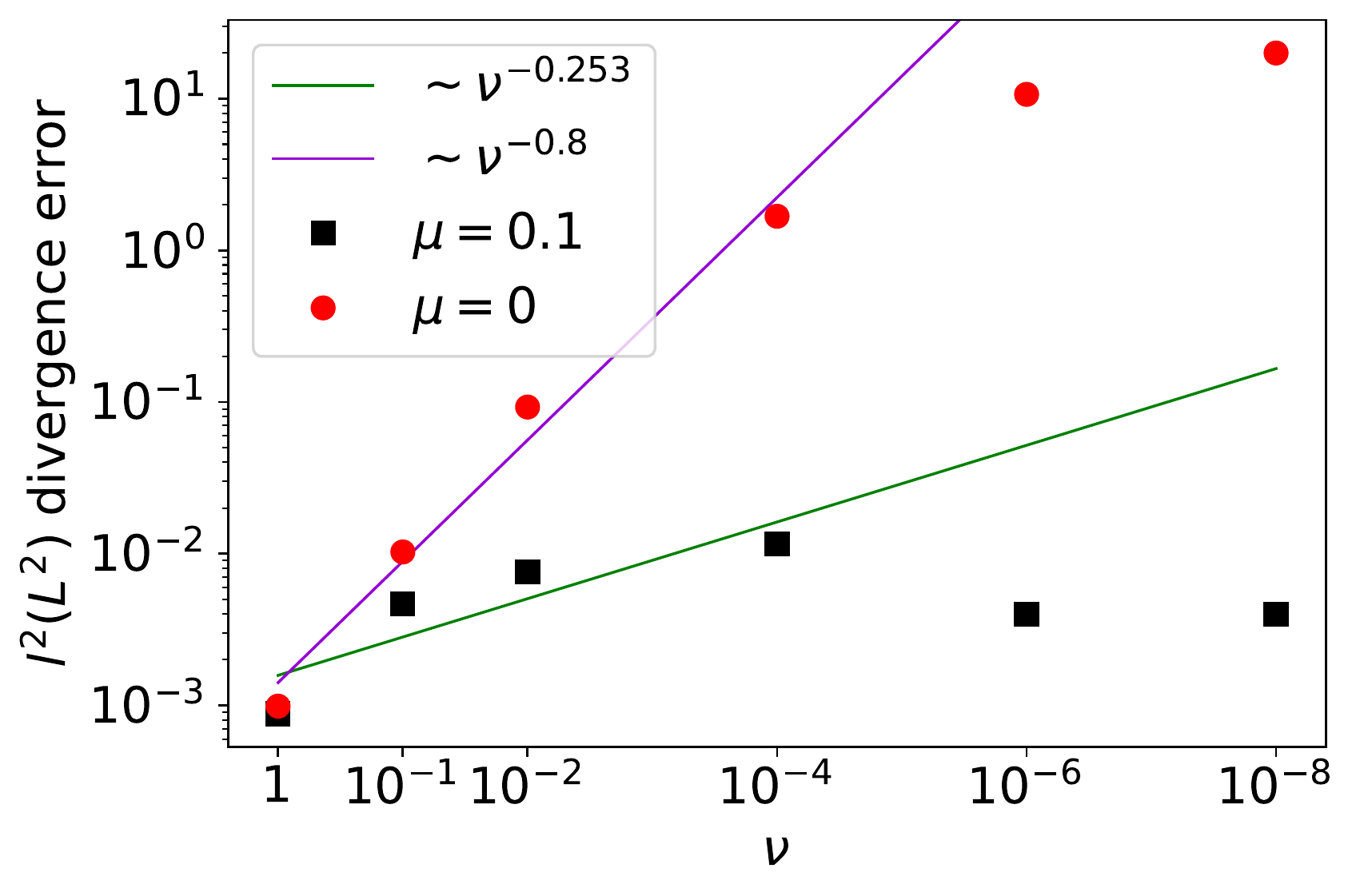}
  \caption{Velocity and divergence errors for the no-flow problem,
    varying viscosity.}
  \label{fig:noflow_visc_u}
\end{figure}

The impact of the value of the viscosity coefficient on the POD-ROM pressure errors is shown in
Figure~\ref{fig:noflow_visc_p}. It can be seen that the errors are notably larger for small viscosity
coefficients if the grad-div stabilization was not applied. For instance, for the case $\nu = 10^{-8}$
we could observe that the FOM simulation without grad-div stabilization became very inaccurate towards
the end of the time interval, leading to pressure snapshots with poor quality. Considering from now on only the simulations
with grad-div stabilization, then the $l^2(L^2)$ errors obtained
with SE-ROM are a little bit smaller than the corresponding errors of the SM-ROM solutions. For $l^2(H_K^1)$,
both methods are of very similar accuracy. Clearly, the pressure errors for SE-ROM and SM-ROM are robust with respect
to small values of the viscosity.

\begin{figure}[t!]
  \centering
  \includegraphics[width = 0.45 \textwidth]{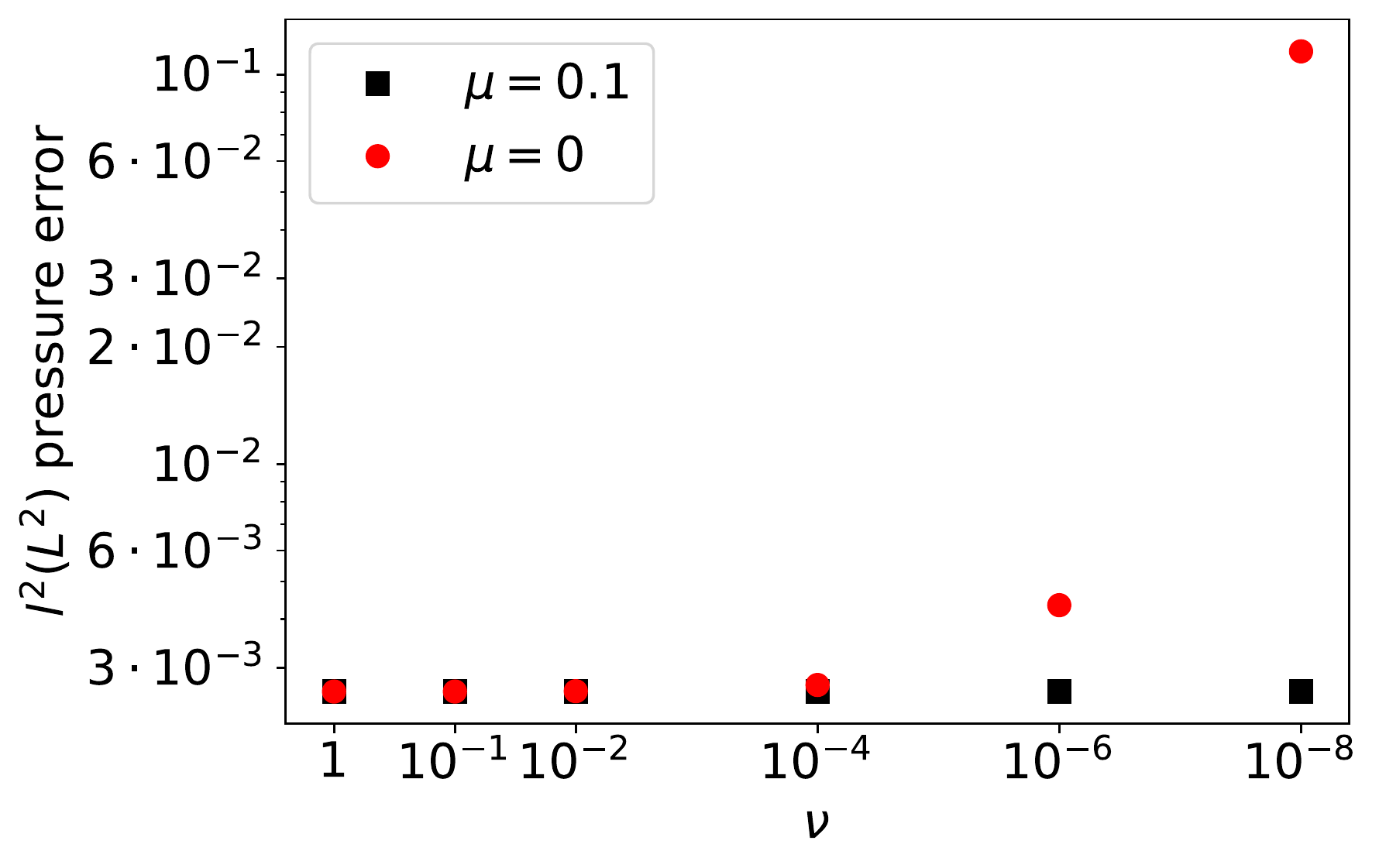}
  \includegraphics[width = 0.45 \textwidth]{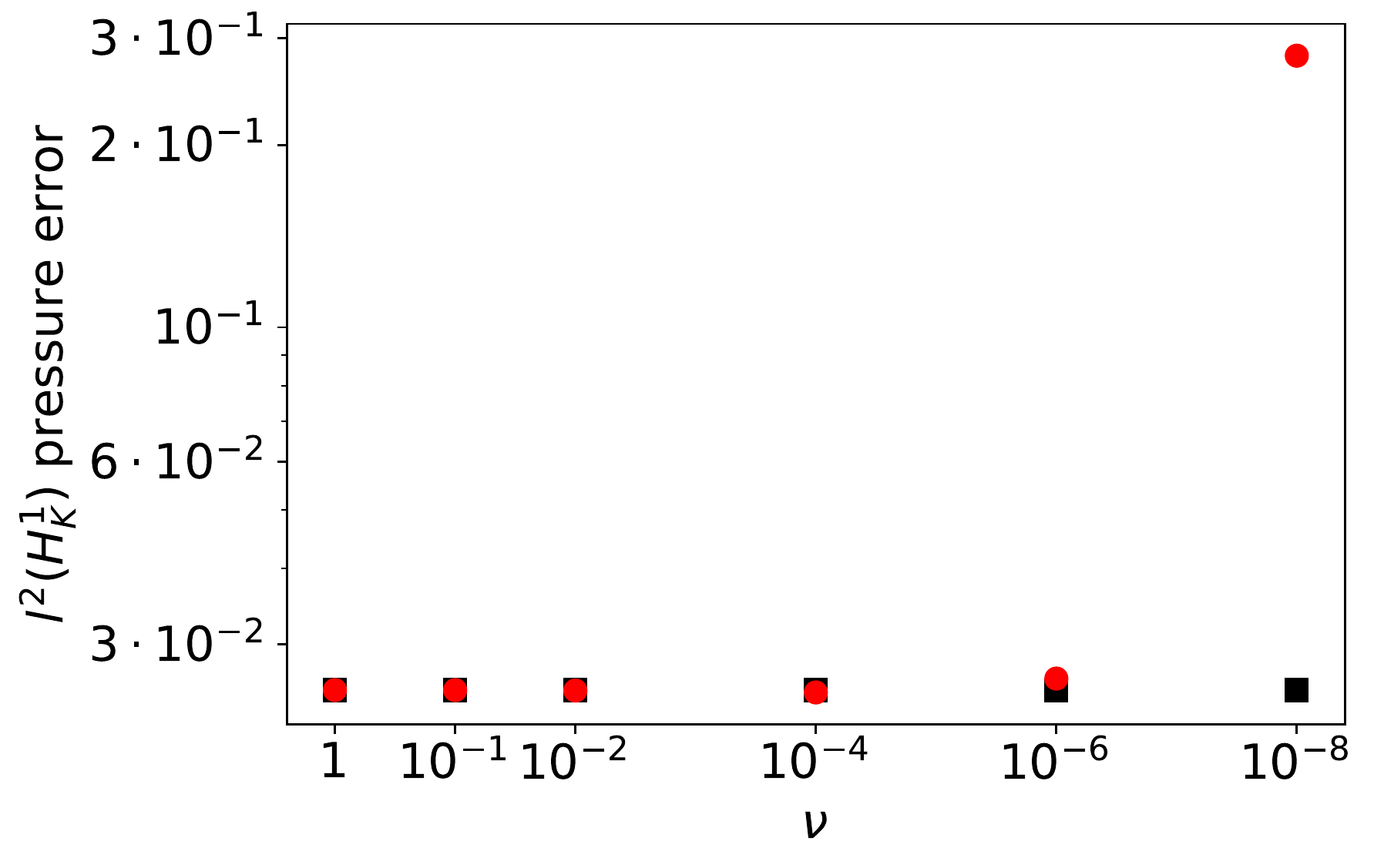}
  \includegraphics[width = 0.45 \textwidth]{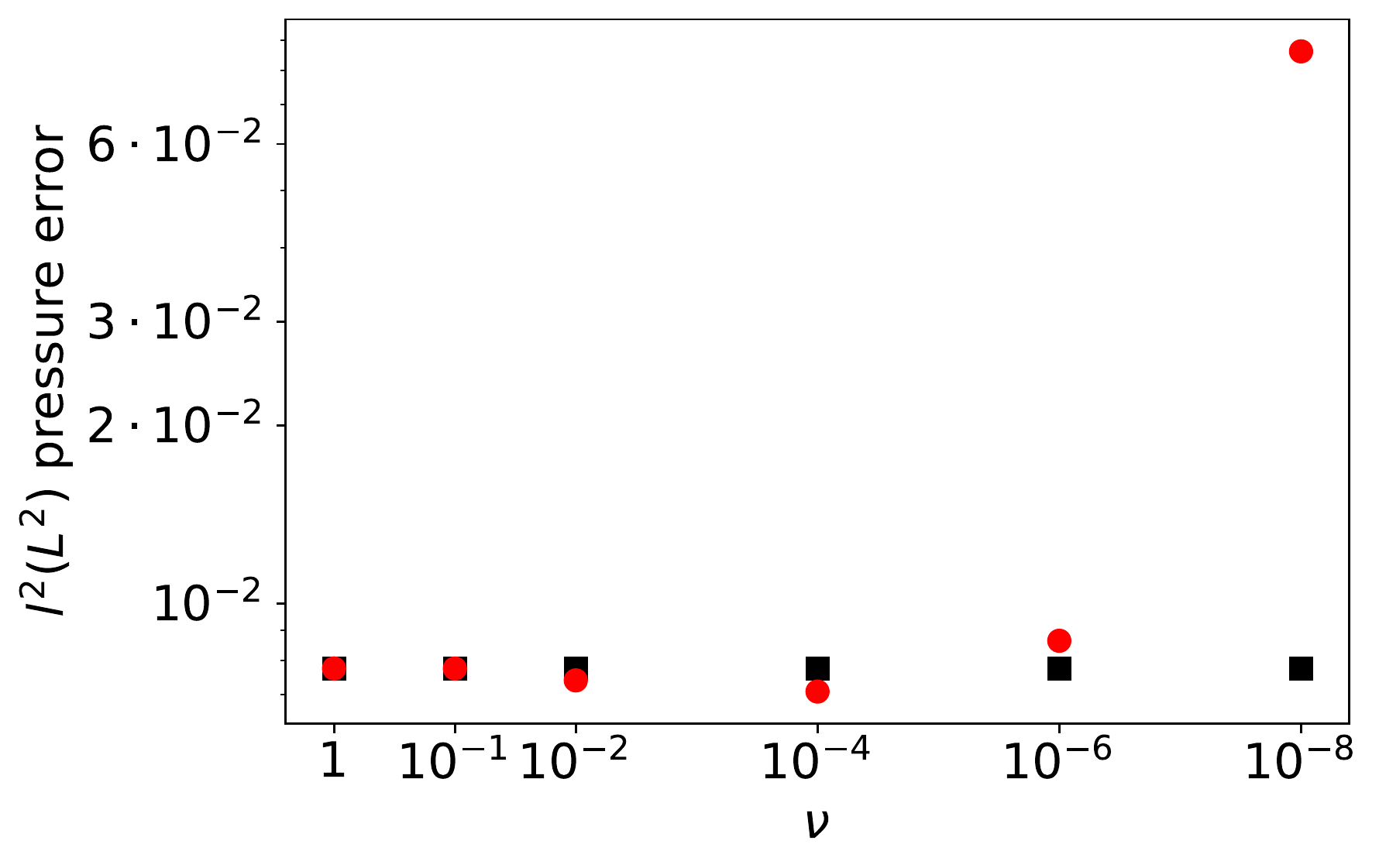}
  \includegraphics[width = 0.45 \textwidth]{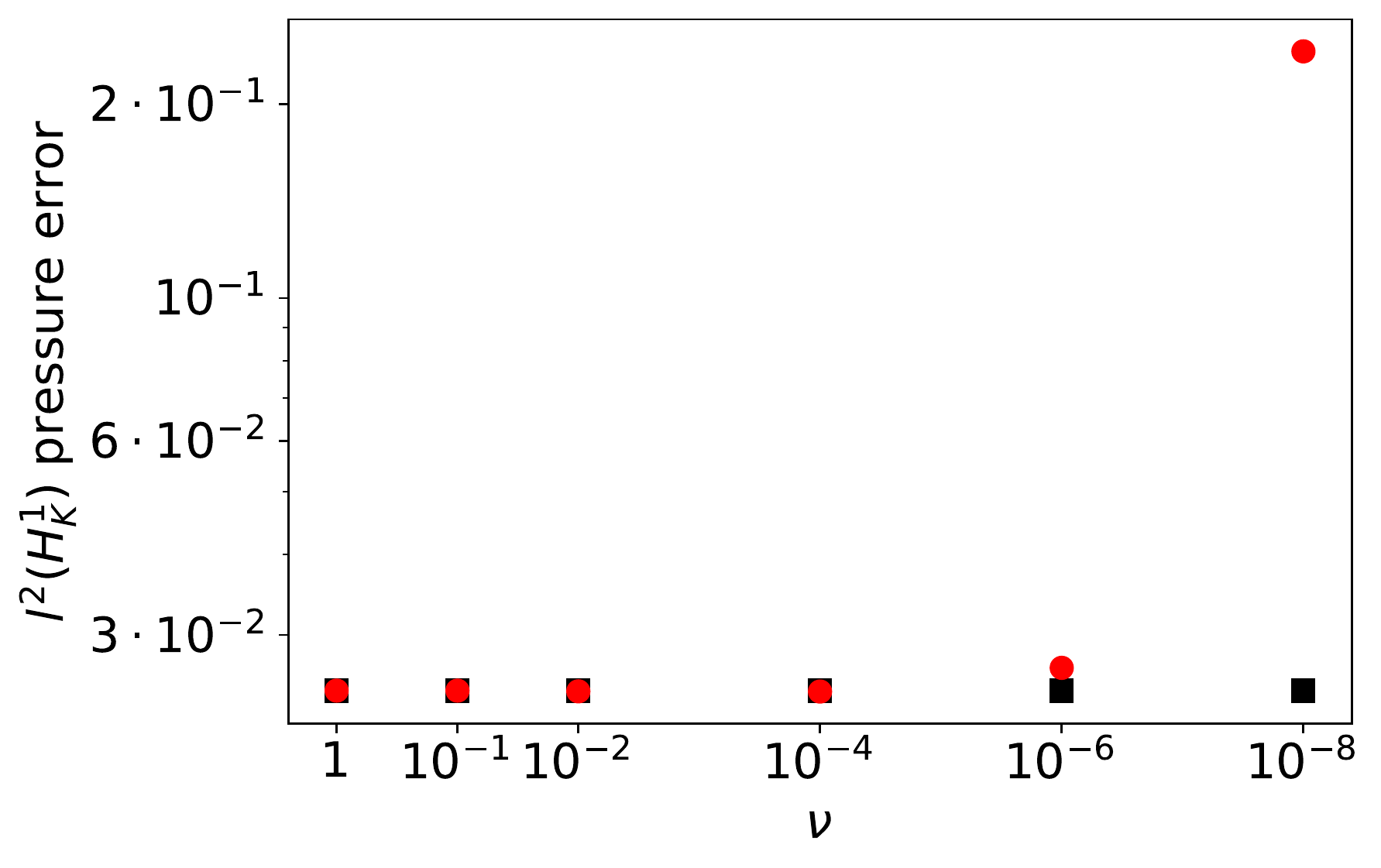}
  \caption{Pressure errors for the no-flow problem,
    varying viscosity.
    Upper row: SE-ROM. Lower row: SM-ROM.}
  \label{fig:noflow_visc_p}
\end{figure}

\subsubsection{Summary}

The well-known benefit of grad-div stabilization on velocity errors could be seen also in this no-flow problem.
It was observed, for small viscosity coefficients, that poor velocity results from FOM simulations without grad-div
stabilization might have a strong impact on POD-ROM pressure errors, because the poor velocity results might induce
inaccurate pressure results and thus lead to pressure snapshots of poor quality. In this respect, using grad-div
stabilization is also beneficial for computing accurate POD-ROM pressure solutions. The numerical studies supported
the analytic results with respect to the robustness of the errors for small viscosity coefficients and with respect
to convergence orders. Concerning the error in $l^2(L^2)$, using SE-ROM led to slightly more accurate results
than using SM-ROM. In all other aspects, both methods behaved very similarly.

\subsection{A Flow around a Cylinder} \label{ssec:cylinder}

A classical example of dynamics that a reduced order model should be able to represent well is the periodic
von K\'arm\'an vortex street that forms on the downstream side of an obstacle placed in the way of a
constant upstream flow of moderate velocity.

We considered the classical benchmark problem defined in \cite{ST96}. Let $\Omega = (0, 2.2) \times (0, 0.41) \setminus \bar B_{0.05}(0.2, 0.2)$
be a rectangle with a slightly off-center disc cut out near the inlet on the left-hand side, representing a cylinder. \iffalse
Prescribe
\begin{itemize}
  \item no-slip boundary conditions $\bu \equiv \boldsymbol{0}$ on the ``cylinder'' boundary
        $$\Gamma_\mathrm{cyl} = \partial B_{0.05}(0.2, 0.2)$$ and the wall
        $$\Gamma_\mathrm{wall} = [0, 2.2] \times \{0, 0.41\},$$
  \item do-nothing boundary conditions on the right-hand edge
        $$\Gamma_\mathrm{out} = \{2.2\} \times (0, 0.41),$$ and
  \item constant inhomogeneous Dirichlet boundary conditions
        $$\bu(\bx) = \left(\frac 6 {0.41^2} x_2 (1 - x_2), 0 \right)$$
        on the left-hand edge
        $$\Gamma_\mathrm{in} = \{0\} \times (0, 0.41).$$
\end{itemize}
With viscosity $\nu = 0.001$ and taking as velocity scale the mean inflow $U_\mathrm{mean} = 1$ and as length scale
the diameter of the cylinder $L = 0.1$, the Reynolds number of the flow relative to these dimensions is
$\mathrm{Re} = 100$. We will investigate the following periodic quantities of interest:
\fi
The inlet boundary condition is prescribed by
\[
\bu(t;(x,y)) = \left(\frac 6 {0.41^2} y (1 - y), 0 \right)~\unitfrac{m}s, \quad y\in [0,0.41].
\]
At the outlet, the so-called do-nothing boundary condition is applied and on all other boundaries,
a homogeneous Dirichlet condition is imposed. The kinematic viscosity of the fluid is assumed to be
$\nu = 0.001~\unitfrac{m^2}{s}$. Taking as characteristic velocity scale the mean inflow $U_\mathrm{mean} = 1~\unitfrac{m}s$ and as characteristic length scale the diameter of the cylinder $\mathcal L = 0.1~\unit{m}$,
the Reynolds number of the flow is $\mathrm{Re} = 100$.

Quantities of interest are the maximal drag and the maximal lift coefficients at the cylinder and the
pressure difference between the front and the back of the cylinder at a certain time within the period
\[
    \Delta P \left(t^* + \frac 1 2 T^*\right)
    = p \left(t^* + \frac 1 2 T^*; (0.15, 0.2) \right)
    - p \left(t^* + \frac 1 2 T^*; (0.25, 0.2) \right),
\]
where $t^*$ is an appropriately defined start of the period and $T^*$ is the length of the period. For the 
exact definitions, we refer to \cite{ST96} or \cite[Example~D.8]{John}. In the numerical studies, the coefficients
were computed by evaluating integrals on a neighborhood of the cylinder, compare  \cite[Example~D.8]{John}.
Since this problem leads to a periodic flow field, the last quantity of interest is the Strouhal number
$\mathrm{St} = \frac{\mathcal L}{UT^*} = 0.1/T^*$, where $T^*$ is estimated by locating pairs of roots
of $c_\mathrm{lift}(t)$.

\begin{figure}[t!]
  \centering
  \includegraphics[width = 0.9 \textwidth]{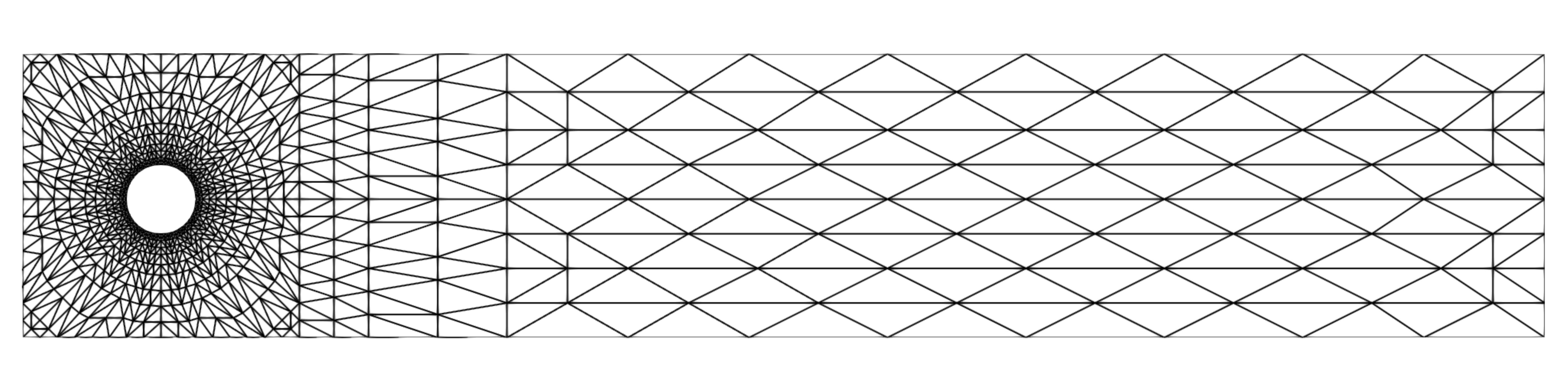}
  \caption{The computational mesh for the flow around a cylinder problem, $L = 1$. Notice that this
    mesh is considerably finer near the cylinder than in the right-hand section of the
    domain.}
  \label{fig:cylinder_mesh}
\end{figure}

We discretized the problem using Taylor--Hood elements with $l=2$ (continuous piecewise quadratic velocities,
continuous piecewise linear pressures) on an irregular triangular mesh. The coarsest version of this mesh is
shown in Figure~\ref{fig:cylinder_mesh}. This mesh was uniformly refined, where in each refinement step 
the vertices on $\Gamma_\mathrm{cyl}$ were 
corrected to lie on the circle. 
Table~\ref{tab:cylinder_mesh} summarizes the mesh statistics. 
As temporal discretization, the BDF2 scheme with a constant time step of $\Delta t = 0.005$ was utilized.

\begin{table}[t!]
  \begin{center}
    \caption{Mesh statistics for the cylinder problem:
      refinement level $L$,
      number of triangular cells $N_K$,
      largest cell diameter $h$,
      velocity and pressure space dimensions.
    }
    \label{tab:cylinder_mesh}
    \begin{tabular}{l|c|c|c|c}
      $L$ & $N_K$ &                $h$ & $\mathrm{dim}(\boldsymbol X^2_h)$ & $\mathrm{dim}(Q^1_h)$ \\ [0.2em]
      \hline
      1 &  1552 & $7.17 \cdot 10^{-3}$ &                              6496 &                   848 \\ [0.2em]
      2 &  6208 & $3.55 \cdot 10^{-3}$ &                             25408 &                  3248 \\ [0.2em]
      3 & 24832 & $1.76 \cdot 10^{-3}$ &                            100480 &                 12704 \\ [0.2em]
      4 & 99328 & $8.79 \cdot 10^{-4}$ &                            399616 &                 50240
    \end{tabular}
  \end{center}
\end{table}

Results were obtained by first developing the flow from homogeneous initial conditions in the time interval
$[0, 8]$, by the end of which the periodic vortex shedding was well established in all simulations. Snapshots
were then generated by a simulation in the interval $[8, 10]$, corresponding to roughly six periods of the
flow's behavior.

As this example includes steady inhomogeneous Dirichlet boundary conditions, we did not include the average of
the snapshots in the computation of the modes. Instead, we applied the POD with respect to the $L^2(\Omega)^2$ inner
product to the FOM snapshots $\tau \partial_t \bu_h^1, \ldots, \tau \partial_t \bu_h^{401}$, where $\tau = 0.3~\unit{s}$
is roughly the length of the period, to construct a linear reduced order velocity space $\bU^r$ and modified
the velocity ROM \eqref{eq:pod_method2} by taking $\bu_r \in \bar \bu_h + \bU^r$, with snapshot average
$\bar \bu_h$, resulting in additional terms involving $\bar \bu_h$ on the right-hand side of the equation in
each time instant.

As before, pressure modes were computed by applying ($L^2(\Omega)$-)POD to the pressure fluctuation snapshots
$p_h^1 - \bar p_h, \ldots, p_h^{401} - \bar p_h$ to give a reduced order pressure space $\mathcal W^r$,
and both SE-ROM and SM-ROM pressures were computed for $p_r \in \bar p_h + \mathcal W^r$.

\begin{table}[t]
  \begin{center}
    \caption{POD statistics for the flow around a cylinder problem:
      refinement level $L$,
      number of velocity modes $d_{v, \mathrm{gd}}$, $d_{v, \mathrm{ngd}}$
      with and without grad-div stabilization,
      number of pressure modes $d_{p, \mathrm{gd}}$, $d_{p, \mathrm{ngd}}$
      with and without grad-div stabilization.
    }
    \label{tab:cylinder_eigenvalues}
    \begin{tabular}{l||c|c||c|c}
      $L$ & $d_{v, \mathrm{gd}}$ & $d_{p, \mathrm{gd}}$ & $d_{v, \mathrm{ngd}}$ & $d_{p, \mathrm{ngd}}$ \\ [0.2em]
      \hline
      1 &                   38 &                   15 &                    52 &                    17 \\ [0.2em]
      2 &                   58 &                   17 &                    72 &                    20 \\ [0.2em]
      3 &                   81 &                   26 &                    82 &                    23
    \end{tabular}
  \end{center}
\end{table}

\begin{figure}[h]
  \centering
  \includegraphics[width = 0.45 \textwidth]{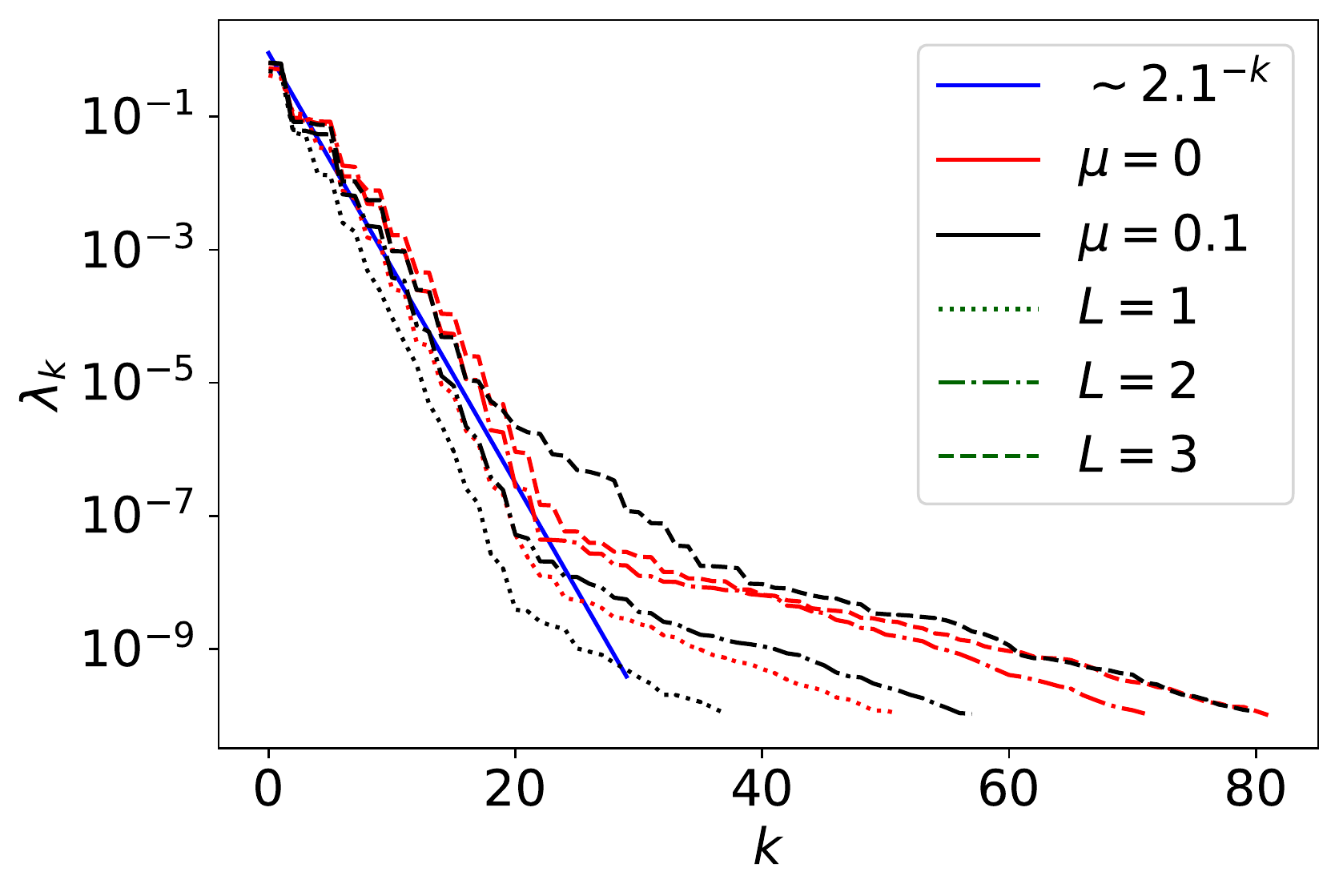}
  \includegraphics[width = 0.45 \textwidth]{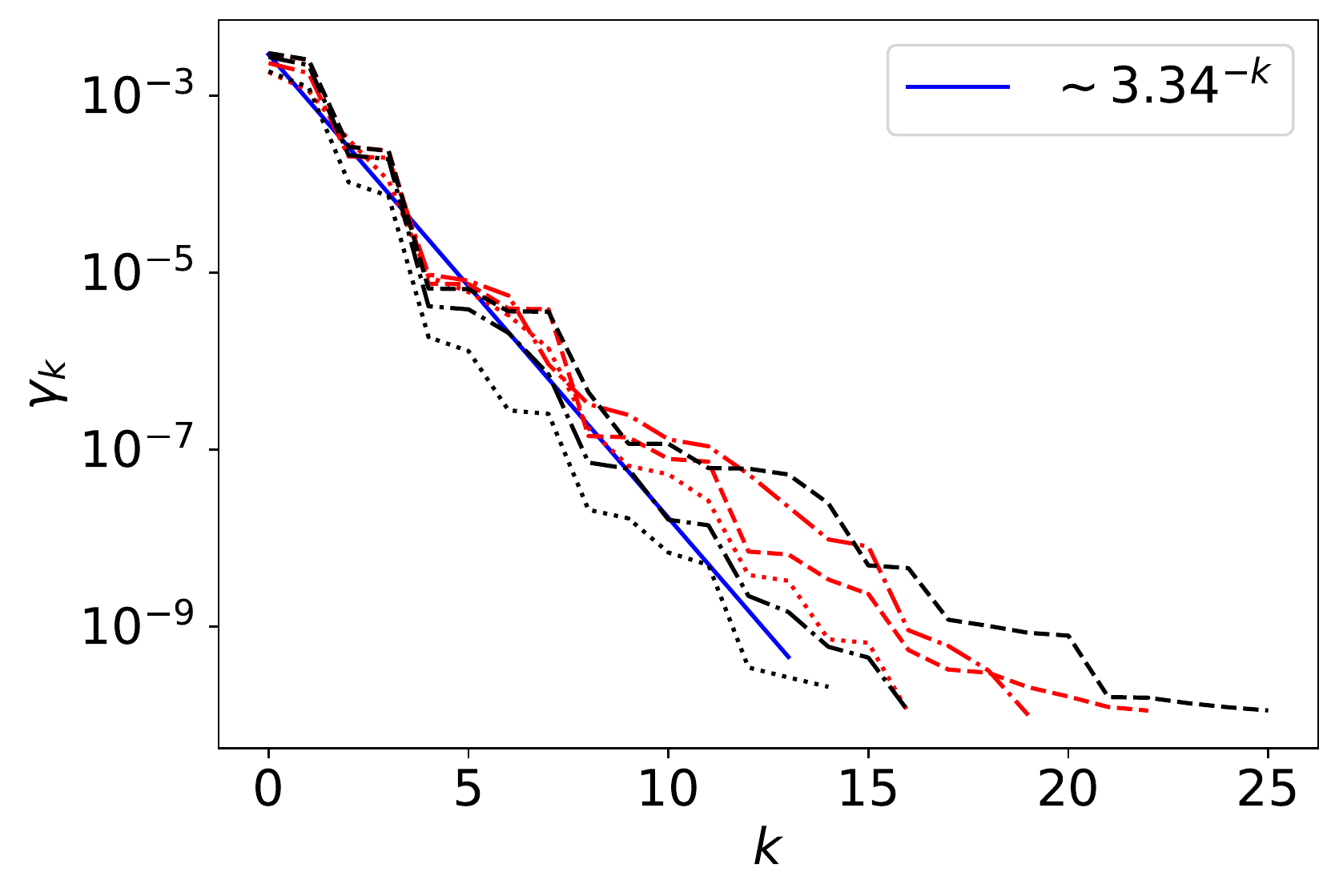}
  \caption{Velocity and pressure eigenvalues for the flow around a cylinder problem, varying grid size.}
  \label{fig:cylinder_eigenvalues}
\end{figure}

Modes corresponding to eigenvalues $\lambda < 10^{-10}$ were again discarded. Table~\ref{tab:cylinder_eigenvalues}
shows the number of eigenvalues depending on the grid size and Figure~\ref{fig:cylinder_eigenvalues} plots their
decay. Note the considerably smaller number of pressure modes.
%and the visible occurrence of eigenspaces of
%dimension~$2$, as previously observed in \cite{TODO}. \todo{saw this somewhere, will look it up and cite}

Reference intervals for the quantities of interest are defined in \cite{ST96}, see Table~\ref{tab:cylinder_results}.
For the sake of brevity, we will present only very few selected results, namely those for $L=3$, 
which are displayed in  Table~\ref{tab:cylinder_results}. As additional reference to compare with, 
we included statistics from a FOM computation at $L = 4$ in the same time interval $[8, 10]$. 

\begin{table}[t]
  \begin{center}
    \caption{Comparison of the flow around a cylinder problem results' to reference intervals. Here $t^*$ is the time in each period
      at which $c_\mathrm{lift}$ reaches its peak, and $T^*$ is the length of the period; this quantity is
      averaged over all available periods.}
    \label{tab:cylinder_results}
    \begin{tabular}{l|c||c|c|c|c}
      Model               & $L$ &    $\mathrm{St}$ & $\max(c_\mathrm{drag})$ & $\max(c_\mathrm{lift})$ & $\Delta P(t^* + \frac 1 2 T^*)$ \\ [0.2em]
      \hline
      Reference interval  &     & $[0.295, 0.305]$ & $[3.22, 3.24]$ & $[0.98, 1.02]$ & $[2.46, 2.50]$ \\ [0.2em]
      FOM, $\mu = 0$      &   4 &           0.3005 &          3.227 &         0.9839 &          2.484 \\ [0.2em]
      \hline
      FOM, $\mu = 0$      &   3 &           0.3004 &          3.227 &         0.9848 &          2.485 \\ [0.2em]
      FOM, $\mu = 0.1$    &   3 &           0.3003 &          3.229 &         0.9914 &          2.487 \\ [0.2em]
      \hline
      SM-ROM, $\mu = 0$   &   3 &           0.3014 &          3.221 &         0.9504 &          2.479 \\ [0.2em]
      SM-ROM, $\mu = 0.1$ &   3 &           0.3029 &          3.239 &         1.0353 &          2.489 \\ [0.2em]
      \hline
      SE-ROM, $\mu = 0$   &   3 &           0.3003 &          3.229 &         0.8925 &          2.473 \\ [0.2em]
      SE-ROM, $\mu = 0.1$ &   3 &           0.3001 &          3.229 &         0.8891 &          2.467 \\ [0.2em]

    \end{tabular}
  \end{center}
\end{table}

The values in Table~\ref{tab:cylinder_results} were obtained with the number of POD modes given in Table~\ref{tab:cylinder_eigenvalues} and they were computed by averaging the values of six subsequent 
periods. For the Strouhal number, the maximal drag coefficient, and the pressure difference, all 
values computed with the POD-ROMs, with and without grad-div stabilization and for both approaches for 
approximating the pressure, are within the respective reference intervals. Results of this type should be expected if all 
important POD modes are used and if the FOM results are already sufficiently accurate. However, 
the situation is different for the maximal lift coefficient, which depends strongly on the pressure 
approximation at the cylinder. In all POD-ROM simulations, the correct order of magnitude was obtained, 
but the results are outside the reference interval, with the results of SM-ROM being closer to it.

\section{Summary and Outlook}\label{sec:summary}

This paper presents the error analysis for two ways of computing a pressure approximation in POD-ROM 
simulations of incompressible flow problems. Based on using a grad-div stabilization for both the 
FOM and the velocity POD-ROM, error bounds were derived whose constants do not blow up as the viscosity
tends to zero. For the supremizer enrichment POD-ROM (SE-ROM), the presented analysis covers a different set of generating
elements of the velocity ROM space than considered so far in the literature. Concerning the 
stabilization-motivated POD-ROM (SM-ROM), the results of the literature are improved considerably with respect 
to several aspects, compare the description in the Introduction. Numerical studies support the analytic 
results and provide an initial comparison of the behavior of both methods for computing a POD-ROM pressure. 

The numerical results comparing SE-ROM and SM-ROM are not yet conclusive. Often, the results are quite 
similar. In the first example, the SE-ROM was more accurate with respect to the error in the $l^2(L^2)$ norm
and in the second example, SM-ROM was more accurate with respect to the maximal lift coefficient. Further 
numerical studies are needed for obtaining a better understanding of the performance of both methods 
in practice. 

\bibliographystyle{abbrv}
\bibliography{references}
\end{document}